\documentclass[11pt]{amsart}

\usepackage{amssymb,latexsym}

\usepackage{graphicx}
\usepackage{enumerate}
\usepackage{enumitem}
\usepackage{comment}
\usepackage{hyperref}
\usepackage{mathrsfs}

\makeatletter

\@namedef{subjclassname@2010}{

  \textup{2010} Mathematics Subject Classification}
\usepackage{amssymb, amsmath,amsthm}
\newtheorem{thm}{Theorem}[section]
\newtheorem{prop}[thm]{Proposition}

\newtheorem{cor}[thm]{Corollary}
\newtheorem{lem}[thm]{Lemma}
\theoremstyle{definition}
\newtheorem{rem}[thm]{Remark}
\newtheorem{def1}[thm]{Definition}

\newcommand{\bk}{\backslash}
\newcommand{\mc}{\mathcal}
\newcommand{\mb}{\mathbb}
\newcommand{\mf}{\mathfrak}

\newcommand{\cond}[1]{\text{cond}(#1)}
\newcommand{\condH}[1]{\textnormal{cond}_H(#1)}
\renewcommand{\deg}[1]{\textnormal{deg}(#1)}

\newcommand{\He}{\mb{H}}

\newcommand{\llf}{\left\lfloor}
\newcommand{\lla}{\left\langle}
\newcommand{\e}{\varepsilon}

\newcommand{\rrf}{\right\rfloor}
\newcommand{\rra}{\right\rangle}
\newcommand{\mbf}{\boldsymbol}
\renewcommand{\ss}{\substack}
\newcommand{\ra}{\to}
\newcommand{\asum}{\sideset{}{^{\ast}}\sum}

\renewcommand{\bar}{\overline}
\renewcommand{\bmod}{\textnormal{ mod }}
\renewcommand{\Re}{\textnormal{Re}}

\begin{document}
\title[Correlations of multiplicative functions in function fields]{Correlations of multiplicative functions in function fields}

\author
{Oleksiy Klurman}
\address{School of Mathematics,
University of Bristol, Woodland Road, Bristol, BS8 1UG, UK}
\email{lklurman@gmail.com}

\author{Alexander P. Mangerel}
\address{Department of Mathematical Sciences, Durham University, Upper Mountjoy Campus, Stockton Road, Durham, DH1 3LE, UK}
\email{smangerel@gmail.com}

\author{Joni Ter\"av\"ainen}
\address{Department of Mathematics and Statistics \\
University of Turku,  20014 Turku\\
Finland}
\email{joni.p.teravainen@gmail.com}

\begin{abstract}
We develop an approach to study character sums, weighted by a multiplicative function $f\colon \mathbb{F}_q[t]\to S^1$, of the form
\begin{align*}
\sum_{\ss{\deg{G} = N \\ G \text{ monic}}}f(G)\chi(G)\xi(G),   
\end{align*}
where $\chi$ is a Dirichlet character and $\xi$ is a short interval character over $\mathbb{F}_q[t].$
We then deduce versions of the Matom\"aki--Radziwi\l{}\l{} theorem and Tao's two-point logarithmic Elliott conjecture over function fields $\mathbb{F}_q[t]$, where $q$ is fixed. The former of these improves on work of Gorodetsky, and the latter extends the work of Sawin--Shusterman on correlations of the M\"{o}bius function for various values of $q$.

Compared with the integer setting, we encounter a different phenomenon, specifically a low characteristic issue in the case that $q$ is a power of $2$.
% as well as the need for a class of ``pretentious" functions called \emph{Hayes characters} that involve two different $\mb{F}_q[t]$-analogues of Archimedean characters.

As an application of our results, we give a short proof of the function field version of a conjecture of K\'atai on classifying multiplicative functions with small increments, with the classification obtained and the proof being different from the existing one in the integer case.

In a companion paper, we use these results to characterize the limiting behavior of partial sums of multiplicative functions in function fields and in particular to solve a ``corrected" form of the Erd\H{o}s discrepancy problem over $\mathbb{F}_q[t]$.

\end{abstract}


\maketitle

\section{Introduction and Results}

In the integer setting, there has been a lot of progress in understanding short sums
\begin{align}\label{eq_shortsum}
\sum_{x < n\leq x+H}f(n),\quad \text{ with } 1 \leq H \leq x
\end{align}
of multiplicative functions $f\colon \mathbb{N}\to \mathbb{C}$, as well as their correlations
\begin{align}\label{eq_correlation}
\frac{1}{x}\sum_{n\leq x}f_1(n)f_2(n+h),\quad  \text{ for } h \geq 1.   
\end{align}
See~\cite{mr-annals},~\cite[Theorem A.1]{mrt},~\cite{mrII} for some papers dealing with~\eqref{eq_shortsum} and~\cite{mrt},~\cite{tao},~\cite{klurman},~\cite{teravainen} for some papers dealing with~\eqref{eq_correlation}. These results have also led to a number of applications, including a solution by Tao~\cite{TaoEDP} to the famous Erd\H{o}s discrepancy problem.

Let $q$ be a fixed prime power and denote by $\mathbb{F}_q[t]$ the ring of polynomials in $t$ over $\mb{F}_q$. Our focus in this paper is on analogues of~\eqref{eq_shortsum} and~\eqref{eq_correlation} over $\mb{F}_q[t]$. These results have applications, in particular, to the Erd\H{o}s discrepancy problem over $\mathbb{F}_q[t]$, which we study in our follow-up paper~\cite{kmt-edp}. In the course of the proofs of our main results, we develop a substantial amount of pretentious number theory over $\mathbb{F}_q[t].$

Let $\mathcal{M}$ denote the set of monic polynomials in $\mathbb{F}_q[t].$ Also, denote by $\mathcal{M}_{\leq N}$ and $\mathcal{M}_N$ the sets of monic polynomials of degree $\leq N$ or $=N$, respectively. Let $\mathcal{P}$ be the set of irreducible monic polynomials in $\mathbb{F}_q[t].$ Again, define $\mathcal{P}_{\leq N}$ and $\mathcal{P}_N$ similarly. Finally, let $\mathbb{U}$ stand for the unit disc of the complex plane. 

By a Dirichlet character $\chi \colon \mathbb{F}_q[t]\to \mathbb{C}$ modulo $M\in \mathcal{M}$ we mean a multiplicative homomorphism $\chi \colon (\mathbb{F}_q[t]/M\mathbb{F}_q[t])^{\times}\to \mathbb{C}\setminus \{0\}$, extended to all of $\mathbb{F}_{q}[t]$ by setting $\chi(G)=0$ whenever $G$ and $M$ are not coprime.

We first describe our result on short sums of multiplicative functions. This provides an analogue of the celebrated Matom\"aki--Radziwi\l{}\l{} theorem~\cite{mr-annals} in function fields.

Matom\"aki and Radziwi\l{}\l{} showed that, for any bounded, real-valued multiplicative function $f\colon \mb{N} \to [-1,1]$, one has
\begin{align*}
\frac{1}{X}\int_X^{2X} \Big|\frac{1}{H}\sum_{x < n \leq x + H} f(n) - \frac{1}{X}\sum_{X < n \leq 2X} f(n)\Big|^2\, dx = o(1),
\end{align*}
as soon as $H=H(X)\to \infty$ with $X$. Thus the short sums of $f$ over $[x,x+H]$ are almost always asymptotic to the corresponding long sum of $f$ over $[X,2X]$, which can either be understood asymptotically or upper bounded non-trivially by Hal\'{a}sz's theorem (see~\cite[Section III.4.3]{Ten} for further details).

In function fields, the role of a short interval is played by\footnote{For ease of comparison with prior function field literature, specifically the work of Keating and Rudnick~\cite{kea_rud},~\cite{kea_rud2}, we note that our short interval $I_H(G_0)$ corresponds to $I(G_0;H-1)$ in the notation of~\cite{kea_rud}.}
\begin{align*}
I_H(G_0):=\{G\in \mathcal{M} \colon \,\, \deg{G-G_0}<H\}, \quad G_0 \in \mc{M}.
\end{align*}
We prove a function field version of the aforementioned result for sums over such short intervals, following a line of approach which differs somewhat from the result over the integers. We state this as follows. 

\begin{thm}[Matom\"{a}ki--Radziwi\l \l \ theorem for function fields, real case] \label{thm_mr_real}
Let $f\colon \mc{M} \to [-1,1]$ be a multiplicative function. Let $N$ be large and let $1 \leq H \leq N-N^{3/4}$ with $H = H(N) \to \infty$ as $N\ra \infty$. \\
(i) If $q$ is odd, we have
\begin{align*}q^{-N}\sum_{G_0 \in \mc{M}_N} \Big|q^{-H} \sum_{\substack{G \in I_H(G_0)}} f(G) - q^{-N} \sum_{G \in \mc{M}_N} f(G)\Big|^2 \ll \frac{\log H}{H} + N^{-1/18+o(1)}.
\end{align*}
(ii) If $q$ is even, we have
\begin{align*}
&q^{-N}\sum_{G_0 \in \mc{M}_N} \Big|q^{-H} \sum_{\substack{G \in I_H(G_0)}} f(G) -  q^{-N}\sum_{G \in \mc{M}_N} f(G)\overline{\chi_1^{\ast}}(G)\Big|^2\ll \frac{\log H}{H} + N^{-1/18+o(1)},
\end{align*}
where $\chi_1 \bmod{t^{N-H+1}}$ is a real character that minimizes the map 
$$
\chi \mapsto \min_{\theta \in [0,1]} \sum_{P \in \mc{P}_{\leq N}} \frac{1-\text{Re}(f(P)\bar{\chi}(P))}{q^{\deg{P}}}, \quad \chi \bmod{t^{N-H+1}},
$$
and where $\chi_1^{\ast}$ is the completely multiplicative function satisfying $\chi_1^{\ast}(t) = 1$ and $\chi_1^{\ast}(G) := \chi_1(t^{\deg{G}}G(1/t))$ for all $G$ coprime to $t$. 
\end{thm}

\textbf{Remarks.}

\begin{itemize}
    \item The long sum $\sum_{G\in \mc{M}_N}f(G)$ appearing in Theorem~\ref{thm_mr_real} is very well-understood, as in the integer setting. This is thanks to a version of Hal\'asz's theorem over function fields, established by Granville, Harper and Soundararajan~\cite{GrHaSoFF}.

\item The savings $(\log H)/H$ obtained is of the same quality as the $h$-dependence found in~\cite[Theorems 1-2]{mr-annals} (replacing $H$ by $\log h$ there). This term arises from a sieve-theoretic bound that allows us to restrict the support of the intervening sums in our analysis to polynomials $G$ having prime factors with degrees in specified intervals, depending on $H$ (see Lemma~\ref{RestricToS} below). This term is not expected to be optimal in general, and standard ``square-root cancellation'' heuristics for sufficiently pseudo-random multiplicative functions (such as the Liouville function $\lambda  \colon  \mc{M} \ra \{-1,1\}$, the completely multiplicative function defined at all prime polynomials $P$ by $\lambda(P) = -1$) suggests a bound of the shape $q^{-H(1/2-\e)}$. In this connection it is worth noting that in the integer setting, Chinis~\cite[Theorem 1.2]{Chi} has shown that, assuming the Riemann Hypothesis, the short sums $h^{-1}\sum_{x < n \leq x+h} \lambda(n)$ of the Liouville function $\lambda$ exhibit a corresponding error term of the quality $h^{-1/2+\e}$ in mean square whenever $h \geq (\log X)^{A}$ with $A = A(\e) > 0$. It might be interesting to pursue a similar result in the $\mb{F}_q[t]$ setting.
    
    \item Note that, interestingly, a low-characteristic issue emerges in the Matom\"aki--Radziwi\l{}\l{} theorem: in $\mathbb{F}_2[t]$, for instance, a real-valued multiplicative function can indeed have different mean values on short and long intervals. This is the reason why we have stated the cases of $q$ odd and even separately in Theorem~\ref{thm_mr_real}. Functions of the form $\chi_1^{\ast}$, where $\chi_1$ is a character modulo a power of $t$, are examples of \emph{short interval characters}; see Definition~\ref{defn3} below, as well as Subsections~\ref{sec_involution} and~\ref{HayesSubSec} for further details relating to the transformation $\chi_1\mapsto \chi_1^{\ast}$.
    
    \item Theorem~\ref{thm_mr_real} can be viewed as generalizing and strengthening the work of Gorodetsky~\cite[Theorem 1.3]{Gor}, who proved that for any \emph{factorization function}\footnote{A function $f(G)$ is called a factorization function if it only depends on the values of $\deg{P}$ and $v_P(G)$, where $P$ runs through the irreducible divisors of $G$, and $v_P(G)$ denotes the largest integer $k$ with $P^{k}\mid G$.} $f$ and for $H = H(N)$ satisfying $H\log \log N/\log N\to \infty$, the sum of $f$ over a short interval $I_H(G_0)$ is almost always asymptotic to the corresponding long sum. Neither the class of factorization functions nor the class of multiplicative functions contains the other, but their intersection contains several interesting number theoretic functions; for example, one of the most important functions in both classes is the M\"obius function $\mu \colon \mathcal{M} \to \{-1,0,+1\}$ (defined as $\mu(G) := (-1)^s$ if $G$ is squarefree and has $s$ irreducible factors, and $\mu(G) := 0$ otherwise). In Theorem~\ref{thm_mr_real}, we do not have any lower bound on how quickly the length $H$ of the interval must grow, which is vital when we use this result to deduce Theorem~\ref{LogEllFF1}.  
    
\end{itemize}

We in fact establish a slightly more general version of Theorem~\ref{thm_mr_real} (namely, Theorem~\ref{MRFF}) that applies to bounded complex-valued multiplicative functions as well, but omit the more complicated statement here for the sake of simplicity.

It is also natural to study the variance of multiplicative functions in arithmetic progressions; see~\cite{hooley},~\cite{harper-sound} for some works on this topic. In the integer setting, an estimate for the variance of a multiplicative function in arithmetic progressions that is of comparable strength to the Matom\"aki--Radziwi\l{}\l{}  theorem was established in~\cite{klu_man_ter}. Here, we generalize this result to function fields, obtaining in fact a stronger version that does not involve exceptional\footnote{It should be noted that if one assumes GRH (generalized Riemann hypothesis) in the integer setting then~\cite[Theorem 1.4]{klu_man_ter} also gives non-trivial estimates for the corresponding variance for \emph{all} moduli $q$ without exception, at least as long as $q$ does not have too many ``small'' prime factors. Note crucially that such a constraint on $Q$ in our setting is not needed for our result as it is stated, simply because our savings are given relative to the (possibly worse-than-trivial) bound $q^{2N-\deg{Q}}$, rather than the sharper $\phi(Q)q^{2(N-\deg{Q})}$. It is the factor $\phi(Q)q^{-\deg{Q}}$, ignored here, that is affected by the primes of small degree.} moduli. For multiplicative factorization functions, this also improves on a corresponding result of Gorodetsky~\cite[Theorem 1.3]{Gor}.
\begin{thm}[Variance of multiplicative functions in arithmetic progressions] \label{VarAPs}
Let $1 \leq H \leq N-N^{3/4}$, such that $H= H(N) \ra \infty$ as $N \ra \infty$. Let $f\colon \mc{M} \ra \mb{U}$ be a multiplicative function. For every $Q \in \mc{M}_{N-H}$ there is a character $\chi_1$ modulo $Q$ such that
$$
\asum_{A \bmod{Q}} \Big|\sum_{\substack{G \in \mc{M}_N \\ G\equiv A \bmod{Q}}} f(G) - \frac{\chi_1(A)}{\phi(Q)} \sum_{G \in \mc{M}_N} f(G)\bar{\chi_1}(G)\Big|^2 \ll \Big(\frac{\log H}{H} + N^{-1/18+o(1)}\Big) q^{2N-\deg{Q}}.
$$
Precisely, $\chi_1$ is any character modulo $Q$ that minimizes the map
$$
\chi \mapsto \min_{\theta \in [0,1]} \sum_{P \in \mc{P}_{\leq N}} q^{-\deg{P}}\Big(1-\textnormal{Re}(f(P)\bar{\chi}(P) e^{-2\pi i\theta \deg{P})}\Big).
$$
\end{thm}

Next, we turn to our result on two-point correlations of multiplicative functions in function fields, with the objective of analogizing Tao's groundbreaking work in~\cite{tao}. Tao's result states that if $f_1,f_2\colon \mathbb{N}\to \mathbb{U}$ are multiplicative functions such that at least one of $f_1$ and $f_2$, say $f_1$, satisfies the non-pretentiousness assumption
\begin{align*}
\inf_{|t|\leq x}\sum_{p\leq x}\frac{1-\Re(f_1(p)\bar{\chi}(p)p^{-it})}{p}\to \infty  \text{ as $x \ra \infty$} 
\end{align*}
for any fixed Dirichlet character $\chi$, then we have
\begin{align*}
\frac{1}{\log x}\sum_{n\leq x}\frac{f_1(n)f_2(n+h)}{n}=o(1)    
\end{align*}
for any fixed $h\neq 0$. The analogue of the logarithmic weight $n\mapsto 1/n$ in function fields is $G\mapsto q^{-\deg{G}}$. 

Tao's result implies that if $f_1$ does not pretend to be a twisted Dirichlet character $n \mapsto \chi(n)n^{it}$, then the autocorrelations of $f_1$ are small. It turns out that in the function field setting there are two collections of Archimedean characters that play a role similar to $n\mapsto n^{it}$, namely the  characters $G \mapsto e^{2\pi i \theta \deg{G}}$ as well as the \emph{short interval characters}, to be defined presently (the group they generate will be discussed in further detail in Section~\ref{sec_prelim1}). Dirichlet characters twisted by at least one of these functions provide obstructions to $f_1\colon \mathcal{M}\to \mathbb{U}$ having small autocorrelations in the setting of $\mb{F}_q[t]$. While in terms of phenomenology this is consistent with the integer setting, some of the arguments in the function field setting require some additional care to address both types of twists. 

\begin{def1}\label{defn3}
A multiplicative function $\xi\colon \mathcal{M}\to \mathbb{C}$ which is not identically zero is called a \emph{short interval character} if there exists $\nu$ such that $\xi(A)=\xi(B)$ whenever the $\nu+1$ highest degree coefficients of $A$ and $B$ agree (that is, $A/t^{\deg{A}}-B/t^{\deg{B}}$ is a rational function of degree $<-\nu$). If $\nu$ is the smallest positive integer with this property then we refer to $\nu$ as the \emph{length} of $\xi$, and write $\text{len}(\xi) = \nu$. 
\end{def1}

\begin{thm} [Two-point logarithmic Elliott conjecture in function fields]\label{LogEllFF1}
Let $A,B \in \mb{F}_q[t] \bk \{0\}$ be fixed, with $A$ monic. Let $f_1,f_2\colon \mathcal{M} \to \mb{U}$ be multiplicative functions. Assume that $f_1$ satisfies the non-pretentiousness assumption
\begin{align}\label{nonpret}
\min_{M \in \mc{M}_{\leq W}} \, \min_{\psi \bmod{M}} \min_{\substack{\xi\,\, \textnormal{short} \\ \textnormal{len}(\xi) \leq N}}\min_{\theta\in [0,1]} \sum_{P \in \mc{P}_{\leq N}} \frac{1-\textnormal{Re}(f_1(P)\bar{\psi}(P)\bar{\xi}(P)e^{-2\pi i\theta\deg{P}})}{q^{\deg{P}}}\to \infty,
\end{align}
as $N \to \infty$ for every fixed $W\geq 1$. Then
\begin{align}\label{correlation2}
\frac{1}{N}\sum_{G \in \mc{M}_{\leq N}} q^{-\deg{G}} f_1(G)f_2(AG+B) = o(1).
\end{align}
as $N \to \infty$.

Moreover, if $f_1$is real-valued and $q$ is odd, then the same conclusion follows provided only that
\begin{align}\label{nonpret-b}
\min_{M \in \mc{M}_{\leq W}} \, \min_{\psi \bmod{M}} \min_{\theta\in \{0,1/2\}} \sum_{\substack{P \in \mc{P}_{\leq N} }} \frac{1-\textnormal{Re}(f_1(P)\bar{\psi}(P)e^{-2\pi i\theta\deg{P}})}{q^{\deg{P}}} \to \infty
\end{align}
as $N\to \infty$.
\end{thm}

\begin{rem}
Observe that if $\xi$ is a short interval character of length $\nu$, $m \geq 2\nu$ and $\deg{B} < \nu$ then $\xi(AG+B) = \xi(A)\xi(G)$ for any $G \in \mc{M}_m$. As $\xi(A) \in S^1$, it follows that as $N \ra \infty$,
$$
\Big|\frac{1}{N} \sum_{G \in \mc{M}_{\leq N}} q^{-\deg{G}}\xi(G)\bar{\xi}(AG+B)\Big| \geq 1 - \frac{2\nu}{N},
$$
so short interval characters clearly present a class of functions with large two-point correlations. This explains why our non-pretentiousness assumption must rule out significant correlations of $f_1$ with such characters.
\end{rem}

Since the M\"obius function $\mu\colon \mathcal{M} \to \{-1,0,+1\}$ is non-pretentious in the sense of~\eqref{nonpret} (by an application of Lemma~\ref{thm_GRH_hayes} below), this result has the following corollary regarding Chowla's conjecture in function fields. 

\begin{cor}[Two-point logarithmic Chowla conjecture in function fields]\label{cor:chowla}
Let $B \in \mb{F}_q[t] \bk \{0\}$ be fixed. Let $\mu\colon \mathbb{F}_q[t]\to \{-1,0,+1\}$ be the M\"obius function. Then as $N \ra \infty$,
\begin{align*}
\frac{1}{N}\sum_{G \in \mc{M}_{\leq N}} q^{-\deg{G}} \mu(G)\mu(G+B) = o(1).
\end{align*}
\end{cor}

\textbf{Remarks.}

\begin{itemize}

\item Theorem~\ref{LogEllFF1} indicates that functions $f$ that pretend to be twisted products of Dirichlet and short interval characters  $\chi\xi e_{\theta}(G)$ (where $e_{\theta}(G):=e^{2\pi i \theta \deg{G}}$) are obstructions to the autocorrelations of $f$ being small. This shows a different phenomenon compared to \emph{mean values} of multiplicative functions in function fields, wherein the only obstructions to the mean value being small are functions pretending to be $e_{\theta}$ (see for instance Lemma~\ref{HalThmFF} below); this is not necessarily unexpected since the problem of estimating mean values is not one that relates to short interval averages.

    \item Theorem~\ref{LogEllFF1} and Corollary~\ref{cor:chowla} compare to previous results as follows. A recent groundbreaking result of Sawin and Shusterman~\cite{sawin-shusterman} established the Chowla conjecture in function fields in the form
\begin{align*}
\frac{1}{q^N}\sum_{G \in \mc{M}_{\leq N}} \mu(G+B_1)\cdots\mu(G+B_k) = o(1),
\end{align*}
as $N \ra \infty$ for any $k\geq 1$ and any distinct $B_1,\ldots, B_k\in \mathbb{F}_q[t]$
in the large field case $q>p^2k^2e^2$, where $p=\textnormal{char}(\mathbb{F}_q)$. In particular, if $q=p^{a}$, then we must have $a\geq 3$ for this condition to hold. Theorem~\ref{LogEllFF1} is somewhat orthogonal to this result in the sense that, despite being limited to two-point correlations, it works for \emph{any} non-pretentious multiplicative functions, unlike the theorem in~\cite{sawin-shusterman} which is specific to the M\"obius function, and Theorem~\ref{LogEllFF1} works in \emph{any} finite field $\mathbb{F}_q$, which will be important for us. We also point out that the 1-point case ($f_2\equiv 1$) of Theorem~\ref{LogEllFF1} is (a logarithmic version of) Hal\'asz's theorem in function fields, proved by Granville, Harper and Soundararajan in~\cite{GrHaSoFF}.

In a different direction, when $N$ is fixed and $q \ra \infty$, Gorodetsky and Sawin~\cite[Theorem 3]{GorSaw} obtained cancellation in two-point correlations $q^{-N} \sum_{G \in \mc{M}_N} \alpha(G) \beta(G+B)$, where $\alpha$ and $\beta$ are factorization functions; in the $q$-limit this yields cancellation for the unweighted sums $q^{-N} \sum_{G \in \mc{M}_{\leq N}} \mu(G)\mu(G+B)$, for example (see~\cite[Theorem 2]{GorSaw} for the precise statement).
\end{itemize}

Lastly, we describe a short application of our results on Elliott's conjecture to the function field analogue of a question of K\'atai.  K\'atai~\cite{katai} conjectured in 1983 that if $f\colon \mathbb{N}\to S^1$ is completely multiplicative and the consecutive values of $f$ are close to each other on average, in the sense that
\begin{align}\label{small}
 \sum_{n\leq x}|f(n+1)-f(n)|=o(x),   
\end{align}
then $f(n)=n^{it}$ for some real number $t.$ This was proved in~\cite{klurman} by the first author. Later, the result was generalized by K\'atai and Phong~\cite{kp} who proved that if $f,g\colon \mathbb{N}\to S^1$ are completely multiplicative and
\begin{align}\label{small2}
 \sum_{n\leq x}|g(2n+1)-zf(n)|=o(x)   \end{align}
for some complex number $z$, then $f(n)=g(n)=n^{it}$. Since in the function field setting there are two varieties of Archimedean characters, namely $e_{\theta}$ and short interval characters $\xi$, our classification of completely multiplicative functions satisfying~\eqref{small} (and in fact more generally~\eqref{small2}) in function fields takes a slightly different form.

\begin{thm}[K\'atai's conjecture in function fields]\label{thm_katai} Let $f\colon \mc{M}\to S^1$ be completely multiplicative, and let $Q\in \mc{M}$. Let $z\in S^1$. Suppose that
\begin{align}\label{katai}
\sum_{G\in \mc{M}_{\leq N}}|f(QG+1)+zf(G)|=o(q^N)    
\end{align}
as $N\to \infty$. Then there exist $\theta\in [0,1)$ and a short interval character $\xi\colon \mc{M}\to \mathbb{U}$ such that $f(G)=\xi(G)e^{2\pi i\theta \deg{G}}$. Conversely, any function of this form satisfies~\eqref{katai} for some $z$. 
\end{thm}

The proof we give for this result is different\footnote{While the method of proof in~\cite[Section 5]{klurman} could in principle be adapted to the function field setting, it would require a function field derivation of binary correlation formulae for multiplicative functions, as in~\cite[Corollary 3.4]{klurman}, which would likely lengthen this paper even further.} in various aspects from the proof in the integer setting in~\cite{klurman}, and could be translated to produce a new proof of K\'{a}tai's result over the integers.

\subsection{Proof ideas}

The proof of Matom\"aki and Radziwi\l{}\l{}~\cite{mr-annals} in the integer setting uses harmonic analysis methods that do not translate directly to function fields. In particular, the characters that control the short sum behavior in function fields are not the Archimedean characters $n^{it}$ as in the integer setting, but rather the short interval characters from Definition~\ref{defn3}.  For our result on the variance in arithmetic progressions, in turn, the set of characters that controls it are the Dirichlet characters. Thus, in order to deal with both theorems simultaneously, we study character sums weighted by $f$ of the form
\begin{align}
\sum_{G\in \mathcal{M}_N}f(G)\chi(G)\xi(G),   
\end{align}
where $\chi$ is a Dirichlet character and $\xi$ is a short interval character. Products of Dirichlet characters and short interval characters are called \emph{Hayes characters} (the same terminology is used in~\cite{Gor} and stems from the fact that Hayes introduced these characters in~\cite{Hayes}). Roughly speaking, we are able to follow the proof strategy of~\cite{klu_man_ter} with this set of characters rather than Dirichlet characters alone. In~\cite{klu_man_ter}, however, our results only applied to characters whose modulus lies outside a small set of exceptional moduli, because of our incomplete understanding of zero-free regions for Dirichlet $L$-functions. In the function field setting, however, we can make use of a consequence of Weil's Riemann hypothesis due to Rhin~\cite{Rhin} that shows that the $L$-functions corresponding to Hayes characters satisfy GRH, which implies that there are no exceptional moduli in this setting.

We gain some information in passing from the physical to the Fourier space versions of the problem by applying an involution (which we learned from the work of Keating and Rudnick~\cite{kea_rud},~\cite{kea_rud2} and which appears earlier in the work of Hayes~\cite[e.g., pp. 115-116]{Hayes}) that relates short interval sums to sums over arithmetic progressions, that is,
$$
\sum_{\substack{\deg{G} = N \\ G \in I_H(G_0)}} f(G) \leftrightarrow \sum_{\substack{\deg{G} = N \\ G \equiv A(G_0) \bmod{t^{N-H+1}}}} f^{\ast}(G),
$$
where $A(G_0)$ is a residue class modulo $t^{N-H+1}$ determined by $G_0$, and $f^{\ast}$ is a kind of dual to $f$ under the correspondence\footnote{Strictly speaking, one needs to restrict to $G$ with $G(0) = 1$ for this to work.}, see Subsection~\ref{sec_involution} for further details (as well as~\cite[Section 5]{kea_rud} for a nice exposition of this idea). For example, this allows us to gain some insight, in the case that $q$ is even in Theorem~\ref{thm_mr_real}, about the nature of the main term in the variance.

For proving our two-point Elliott result, we in fact need a generalized version of our Matom\"aki--Radziwi\l{}\l{} theorem in function fields, where we twist the multiplicative function by an additive character, thus looking at the short exponential sum
\begin{align}\label{eq_shortexpsum}
\sum_{\substack{G\in \mc{M}_N \cap I_H(G_0)}}f(G)e_{\mathbb{F}}(\alpha G)  
\end{align}
for almost all $G_0$ (see Section~\ref{sec:notation} for the relevant notation). This exponential sum is analyzed by adapting the approach of Matom\"aki--Radziwi\l{}\l{}--Tao from~\cite{mrt} to function fields (see Theorem~\ref{ShortExpSumFFComplex}). In particular, this involves performing the circle method in function fields, which is perhaps of independent interest.

To complete the proof, we develop a version of Tao's entropy decrement argument from~\cite[Section 3]{tao} that allows us to express the two-point correlation as a two-variable correlation. By a bit of Fourier analysis, we can reduce the necessary estimate for this two-variable correlation sum to the estimate for~\eqref{eq_shortexpsum} that we proved. 

We use our correlation results in the proof of Theorem~\ref{thm_katai} in order to reduce our classification of functions $f$ to functions that are pretentious to a Hayes character $\chi \xi e_{\theta}$. We eschew the need for correlation formulas, as found in~\cite{klurman}, by instead appealing to a concentration inequality that forces $f\bar{\chi \xi} e_{-\theta}$ to be close to $1$ along structured sequences of polynomials. A judicious construction of such sequences leads to Theorem~\ref{thm_katai}.

\subsection{Structure of the Paper}
The paper is organized as follows. In Section~\ref{sec_prelim1}, we present some preliminary lemmas on the pretentious distance, the involution mentioned above, and Hayes characters. In Section~\ref{sec_prelim2} we introduce the remaining relevant preliminaries relating especially to mean square and pointwise estimates for character sums that will be needed in the proofs of Theorems~\ref{VarAPs},~\ref{MRFF} and~\ref{LogEllFF1}. In Section~\ref{VarSec}, we prove Theorem~\ref{VarAPs} using these lemmas. The proof of the Matom\"aki--Radziwi\l{}\l{} theorem (Theorem~\ref{thm_mr_real}) proceeds completely analogously 
and is described in Section~\ref{MRTHMSec}. In Section~\ref{sec_shortexp}, we establish cancellation in exponential sums over short intervals weighted by any non-pretentious multiplicative function. Finally, in Section~\ref{LogEllSec} we adapt the entropy decrement argument of~\cite{tao} to the function field setting and apply the short exponential sum estimate for multiplicative functions from Section~\ref{sec_shortexp} to establish Theorem~\ref{LogEllFF1}. Section~\ref{KataiSec} is then devoted to the proof of our application, Theorem~\ref{thm_katai}, on K\'atai's conjecture.

\subsection{Acknowledgments} This work began when the authors were in residence for the ``Probability in Number Theory'' Workshop at CRM in the spring of 2018, and continued in particular at the ``Sarnak's Conjecture'' workshop at AIM that winter. We would like to thank both institutions for their hospitality and for excellent working conditions. 

We are grateful to the anonymous referee for a very careful reading of the paper and for numerous comments leading to improvements to the exposition. We would also like to thank Andrew Granville and Maksym Radziwi\l{}\l{}  for their encouragement, and Ofir Gorodetsky for pointing out a correction to a smooth number estimate referenced in an earlier version of this paper. The third author was supported by a Titchmarsh Fellowship and Academy of Finland grant no. 340098.

\section{Notation}\label{sec:notation}

Throughout the paper, $p$ is the characteristic of $\mb{F}_q$, and $q = p^k$ for some $k \geq 1$. 

We denote by $\mc{M}$ the set of monic polynomials in $\mb{F}_q[t]$ (we do not denote $q$ dependence in $\mathcal{M}$, since it will always be clear from the context), and $\mc{P}$ the set of monic irreducible (prime) polynomials in $\mb{F}_q[t]$. For $N \in \mb{N}$, we write $\mc{M}_N$, $\mc{M}_{\leq N}$ and $\mc{M}_{< N}$ to denote, respectively, the set of monic polynomials of degree exactly $N$, less than or equal $N$ and strictly less than $N$. Analogously, we define $\mc{P}_N$, $\mc{P}_{\leq N}$ and $\mc{P}_{<N}$ to be the corresponding sets of monic irreducible polynomials. We denote the degree of $M \in \mb{F}_q[t]$ by $\text{deg}(M)$.

Given two polynomials $F,G\in \mathcal{M}$, not both zero, we define their greatest common divisor $(F,G)$ as the unique polynomial $D\in \mathcal{M}$ such that $D\mid F, D\mid G$ and such that for any $D'\in \mathcal{M}$ satisfying $D'\mid F, D'\mid G$ we have $D'\mid D$. The least common multiple $[F,G]$ of $F$ and $G$ is in turn defined by $[F,G]:=FG/(F,G)$.

Typically, $G$ will be used to denote an element of $\mc{M}$, whereas $R$ or $P$ denotes an element of $\mc{P}$ and $M$ denotes an element of $\mb{F}_q[t]$, monic or otherwise.

Given two polynomials $G_0,G \in \mc{M}$ and a parameter $H \geq 1$, we write 
$$
I_H(G_0):= \{G \in \mc{M}  \colon  \deg{G-G_0} < H\}
$$ 
to denote the short interval centred at $G_0$ of size $H$. 

As usual, given $t \in \mb{R}$ we write $e(t) := e^{2\pi i t}$. Given a parameter $\theta \in [0,1]$ and a polynomial $G \in \mb{F}_q[t]$, we also write $e_{\theta}(G) := e(\theta \deg{G})$. Finally, given an element $\alpha \in \mb{K}_{\infty}(t)$ (see Section~\ref{sec_shortexp}) with formal Laurent series $\alpha = \sum_{k = N}^{\infty} a_{-k}(\alpha) t^{-k}$, we define $e_{\mb{F}}(\alpha) := e(\text{tr}_{\mb{F}_q/\mb{F}_p} a_{-1}(\alpha)/p)$, where $\text{tr}_{\mb{F}_q/\mb{F}_p}$ denotes the usual field trace. We also define $\lla \alpha \rra := q^{-N}$. 

Throughout the paper, we write $\mb{U} := \{z \in \mb{C}  \colon  |z| \leq 1\}$ and $S^1 := \{z \in \mb{U}  \colon  |z| = 1\}$. We say that $f\colon \mathcal{M}\to \mathbb{C}$ is multiplicative if $f(G_1G_2)=f(G_1)f(G_2)$ whenever $G_1,G_2$ are coprime. Given multiplicative functions $f,g\colon \mc{M} \to \mb{U}$, we define the pretentious distance between them by
\begin{align}\label{eqq36}
\mb{D}(f,g;N) := \Big(\sum_{P \in \mc{P}_{\leq N}} q^{-\deg{P}}(1-\text{Re}(f(P)\bar{g}(P)))\Big)^{1/2},
\end{align}
and define $\mb{D}(f,g;M,N)$ similarly, but with the summation being over $P\in \mc{P}_{\leq N}\setminus \mc{P}_{\leq M}$. We also set
$$
\mc{D}_f(N) := \min_{\theta \in [0,1]} \mb{D}(f, e_{\theta};N)^2.
$$

Given a monic polynomial $G \in \mc{M}$ with $G(0) = 1$, we put $G^{\ast}(t) := t^{\deg{G}}G(1/t)$ (see Subsection~\ref{sec_involution} for further discussion). For a multiplicative function $f\colon \mc{M} \to \mb{U}$, we define the associated multiplicative function $f^{\ast} \colon \mc{M} \to \mb{U}$ as $f^{\ast}(G) := f(G^{\ast})$, whenever $G(0) = 1$, and set $f^{\ast}(G)=0$ otherwise.

Given a Dirichlet character $\chi$ modulo $Q$ (defined above), we define its conductor\footnote{This is strictly speaking an abuse of notation/terminology, as the conductor of a function field Dirichlet character $\chi$ ought to be a polynomial of least degree that is a period for $\chi$. Here, we use it as an integer-valued measure of complexity of the character, which will be convenient for us in various estimates in the sequel.} as $\cond{\chi}:=\deg{Q'}$ if $Q'\mid Q$ is such that $\chi(M)$ agrees with a primitive Dirichlet character $\chi' \bmod{Q'}$ for all $M$ coprime to $Q$. In this case, we say that $\chi'$ induces $\chi$. We write $\mc{X}_Q$ to denote the set of Dirichlet characters modulo $Q$. 

A \emph{Hayes character} is a character of the form $\tilde{\chi} = \psi_Q \xi_{\nu}$, where $\psi_Q$ is a Dirichlet character to modulus $Q$ induced by a primitive character to some modulus $Q'$ and $\xi_{\nu}$ is a length $\nu$ short interval character for some $\nu\geq 0$ (in Section~\ref{sec_prelim1} we first give a different definition and then note that it is equivalent to this one). We define the conductor of $\tilde{\chi}$ by $\textnormal{cond}_H(\tilde{\chi}) := \deg{Q'} + \nu$. We say that $\tilde{\chi}$ is \emph{non-principal} if $\textnormal{cond}_H(\tilde{\chi})>0$. We further say that $\tilde{\chi}'$ \emph{induces} $\tilde{\chi}$ if $\tilde{\chi}'=\chi'\xi'$ and $\tilde{\chi}=\chi\xi$, with the Dirichlet character $\chi'$ inducing $\chi$ and $\xi'=\xi$. We also write $\mc{X}_{Q,\nu}$ to denote the collection of Hayes characters of the form $\psi_Q\xi_{\nu}$, where $\psi$ has modulus $Q$ and $\xi$ has length $\nu$. See Subsection~\ref{HayesSubSec} for further discussion.

We will sometimes write $\mu_k$ to denote the set of roots of unity of order $k$, where $k \in \mb{N}$. 

The functions $\Lambda$,  $\omega$, $\lambda$, $\mu$, $\phi$, $\textnormal{rad}$ and $\nu_P$, defined on $\mc{M}$, are the analogues of the corresponding arithmetic functions in the number field setting. Thus 
\begin{itemize}
    \item $\Lambda(G)= \deg{P}$ if $G=P^k$ for some $k\geq 1$ and  $P\in \mathcal{P}$, and $\Lambda(G)=0$ otherwise. 
    \item $\omega(G)$ is the number of distinct irreducible divisors of $G$.
    \item $\lambda\colon \mathcal{M}\to \{-1,+1\}$ is the completely multiplicative function with $\lambda(P)=-1$ for all $P\in \mathcal{P}$.
    \item $\mu\colon \mathcal{M}\to \{-1,0,+1\}$ is given by $\mu(G)=(-1)^{\omega(G)}$ for $G$  not divisible by $P^2$ for any $P\in \mathcal{P}$, and $\mu(G)=0$ otherwise. 
    \item $\phi(G)$ is the size of the finite multiplicative group $(\mathbb{F}_q[t]/G\mathbb{F}_q[t])^{\times}$.
    \item $\textnormal{rad}(G)=1$ if $G=1$ and $\textnormal{rad}(G)=P_1\cdots P_k$ if $P_1,\ldots, P_k$ are the distinct irreducible factors of $G$.
    \item $\nu_P(G)$, for $P\in \mathcal{P}$, is the largest integer $k$ such that $P^k\mid G$.
\end{itemize}

Throughout this paper, the cardinality $q$ of the underlying finite field $\mb{F}_q$ is fixed. For the sake of convenience we have chosen to omit mention of dependencies of implicit constants in our estimates on $q$. In particular, the implicit constants in any estimate may depend on $q$ throughout this paper.

\section{Preliminaries I: Multiplicative Functions and Hayes Characters} \label{sec_prelim1}

In this section we establish some auxiliary lemmas, specifically related to multiplicative functions, that will be necessary in the proofs of Theorems~\ref{thm_mr_real} and~\ref{LogEllFF1}. Recall the definition of Hayes characters from Section~\ref{sec:notation}. 

\subsection{Lemmas on character sums}

When working over $\mathbb{F}_q[t]$, we have the generalized Riemann hypothesis at our disposal, arising from an application of Weil's Riemann hypothesis for curves over finite fields (see~\cite[p. 134]{Wei}).\footnote{Even though GRH is useful for us in certain parts of our arguments, specifically Lemmas~\ref{lem_hayesdist} and~\ref{hayesbound}, we emphasize that it is not the main driving force behind the proofs of our results. As noted implicitly in~\cite{mr-annals} and explicitly in~\cite{klu_man_ter}, obtaining non-trivial bounds on the variance in short intervals and arithmetic progressions, respectively, in the integer setting only requires the existence of sufficiently wide zero-free regions, for example of Korobov-Vinogradov type, for Dirichlet $L$-functions to the left of $\text{Re}(s) = 1$ (and for $L$-functions of Hayes characters in our setting).} 

\begin{lem}[Rhin]\label{thm_GRH_hayes}
Let $N \geq 1$. Let $\tilde{\chi}$ be a non-principal Hayes character. Then
\begin{equation}\label{HayesRH}
\sum_{G \in \mc{M}_N} \tilde{\chi}(G)\Lambda(G) \ll \textnormal{cond}_{H}(\tilde{\chi})q^{N/2}.
\end{equation}
\end{lem}

\begin{proof}
This is~\cite[Theorem 3]{Rhin}.
\end{proof}

A useful corollary of Lemma~\ref{thm_GRH_hayes} is the following. 

\begin{lem}[A pretentious distance bound] \label{lem_hayesdist}
Let $N \geq 3$, $A\geq 1$. Let $\tilde{\chi}$ be a non-principal Hayes character of conductor $\text{cond}_H(\tilde{\chi}) \leq N^{A}$. Then
$$
\max_{\theta\in [0,1]}\Big|\sum_{P \in \mc{P}_{\leq N}} \tilde{\chi}(P)e_{\theta}(P) q^{-\deg{P}} \Big| \ll_A \log\log N.
$$
\end{lem}

\begin{proof}
Splitting the sum according to degree, then separating the contribution of $\deg{P} \leq 10A\log N$ from its complement, we get
\begin{align*}
&\sum_{P \in \mc{P}_{\leq N}} \tilde{\chi}(P)e_{\theta}(P)q^{-\deg{P}} \\
&= \sum_{d \leq 10A\log N} e(\theta d)q^{-d}\sum_{P \in \mc{P}_d} \tilde{\chi}(P) + \sum_{10A\log N < d \leq N} e(\theta d)q^{-d}\sum_{P \in \mc{P}_d} \tilde{\chi}(P) =: T_1 + T_2.
\end{align*}
We bound the first sum trivially using the prime polynomial theorem, yielding
$$
T_1 \ll \sum_{d \leq 10A\log N} q^{-d} |\mc{P}_d| \ll \sum_{d \leq 10A\log N} \frac{1}{d} = \log \log N + O_A(1).
$$
We now consider $T_2$. Replacing $\sum_{P \in \mc{P}_d} \tilde{\chi}(P)$ by $\sum_{G \in \mc{M}_d} \tilde{\chi}(G)\Lambda(G)/d$ in the inner sum over primes in $T_2$ incurs an error of size $O(\sum_{d \leq N}q^{-d/2}) = O(1)$ from terms $P^k$ with $k \geq 2$. This sum can thus be expressed as
$$
T_2 = \sum_{10A\log N < d \leq N} \frac{e(\theta d)}{dq^d} \sum_{G \in \mc{M}_d} \Lambda(G) \tilde{\chi}(G) + O(1).
$$
By Lemma~\ref{thm_GRH_hayes}, we can bound this as
\begin{align*}
|T_2| &\leq \sum_{10A\log N < d \leq N} \frac{1}{dq^d} \Big|\sum_{G \in \mc{M}_d} \Lambda(G)\tilde{\chi}(G)\Big| + O(1) \ll \sum_{10A\log N < d \leq N}\frac{N^A}{dq^{d/2}} + O(1) \\
&\ll N^{A}\cdot 2^{-5A\log N} + 1 \ll 1.
\end{align*}
Combining the contributions from $T_1$ and $T_2$, we obtain the claim.
\end{proof}

We will also need a bound on sums of Hayes characters over $\mathcal{M}$ (as opposed to $\mathcal{P}$). 

\begin{lem}[Pointwise bound for character sums over monics]\label{hayesbound} Let $M>N\geq 1$. Let $\tilde{\chi}$ be either a non-principal Dirichlet character or a non-principal short interval character of conductor $M$. Then we have
\begin{align}
 \sum_{G\in \mc{M}_N}\tilde{\chi}(G)\ll q^{N/2}\binom{M-1}{N}.
\end{align}
\end{lem}

\begin{rem}
This lemma will be applied in particular when $M \leq (1+o(1))N$. For Dirichlet characters, we could instead have appealed to the P\'{o}lya-Vinogradov inequality 
(see~\cite[Proposition 2.1]{hsu}) to produce a sharper bound in this range, rather than applying Weil's RH; however, a corresponding result for general Hayes characters does not exist in the literature. For this reason, we have resorted to appealing to RH instead. 

Note that the same quality bound, with an essentially identical proof, appears as~\cite[Lemma 2.1]{FaifRud} in the context of Dirichlet characters; for the sake of completeness we include the short proof of the general case.
\end{rem}

\begin{proof}
By the GRH for $L$-functions corresponding to Hayes characters (\cite{Rhin}), we can write the $L$-function
\begin{align*}
\mathcal{L}(z,\tilde{\chi})=\sum_{G\in \mc{M}}\tilde{\chi}(G)z^{\deg{G}}    
\end{align*}
as
\begin{align}\label{lfunction}
\mathcal{L}(z,\tilde{\chi})=\prod_{j=1}^{M-1}(1-\alpha_iz)    
\end{align}
for some $\alpha_i=\alpha_i(\tilde{\chi})$ that all have modulus either $1$ or $q^{1/2}$. Now, the sum in question is the coefficient of $z^N$ on the right of~\eqref{lfunction}, which by Vieta's formulae is equal to
\begin{align*}
\sum_{\substack{S\subseteq [1,M-1]\cap \mathbb{N}\\|S|=N}}\prod_{j\in S}(-\alpha_j).    
\end{align*}
This is trivially bounded in absolute value by $q^{N/2}\binom{M-1}{N}$, which yields the claim.
\end{proof}

\subsection{Multiplicative Functions in Function Fields}
Let $f\colon \mc{M} \to \mb{U}$ be a 1-bounded multiplicative function. Define the Dirichlet series corresponding to $f$ by
\begin{align}\label{eqq19}
L(s,f) := \sum_{N \geq 0}\sum_{G \in \mc{M}_N} f(G)q^{-\deg{G}s} = \prod_{P \in \mc{P}} \sum_{k \geq 0} f(P^k)q^{-k\deg{P}s},
\end{align}
for $\Re(s)>1$; in this region both expressions converge absolutely. 

Recall the pretentious distance
$$
\mb{D}(f,g;N) := \Big(\sum_{P \in \mc{P}_{\leq N}} q^{-\deg{P}}\Big(1-\text{Re}(f(P)\bar{g(P)})\Big)\Big)^{\frac{1}{2}}.
$$
One can show~\cite{klurman2017mean} that $\mb{D}$ satisfies a triangle inequality of the shape
$$
\mb{D}(f,h;N) \leq \mb{D}(f,g;N) + \mb{D}(g,h;N),
$$
for any $f,g,h \colon \mathcal{M}\to \mathbb{U}$ multiplicative. Define also 
$$
\mc{D}_f(N) := \min_{\theta \in [0,1]} \mb{D}(f,e_{\theta};N)^2.
$$ 
The following variant of Hal\'{a}sz's theorem then holds:
\begin{thm}[Hal\'{a}sz's theorem in function fields]\label{HalThmFF}
Let $N \geq 1$. Let $f\colon\mc{M} \to \mb{U}$ be multiplicative. Then
$$\frac{1}{q^N} \sum_{G \in \mc{M}_N} f(G) \ll (1+\mc{D}_f(N))e^{-\mc{D}_f(N)}.$$
\end{thm}

\begin{proof}
We will reduce this to the Granville--Harper--Soundararajan formulation of Hal\'asz's inequality in~\cite{GrHaSoFF}. Define the multiplicative function $\tilde{f}_N$ on prime powers by\footnote{This is technically different from the definition of $f^{\perp}$ used in~\cite{GrHaSoFF}. However, it is always true that, in the notation there, $\Lambda_{\tilde{f}_N}(P) = \Lambda_{f^{\perp}}(P)$, and the difference lies only in values at powers $P^k$ with $k \geq 2$. It is easy to check, then, that $|L(s,f^{\perp})|$ and $|L(s,\tilde{f}_N)|$ differ in at most a factor of an absolute constant whenever $f$ is 1-bounded and $\Re(s)=1$.}
$$
\tilde{f}_N(P^k) := \begin{cases} f(P^k) &\text{ if $\deg{P^k} \leq N$} \\ 0 &\text{ otherwise.} \end{cases}$$
Then,~\cite[Corollary 1.2]{GrHaSoFF} (in the case $\kappa = 1$) shows that
$$\frac{1}{q^N} \sum_{G \in \mc{M}_N} f(G) \ll (1+M)e^{-M},$$
where $M := \min_{\Re(s)=1} \log\Big(2N/|L(s,\tilde{f}_N)|\Big)$. Now, the prime polynomial theorem gives
\begin{align*}
&\sum_{P \in \mc{P}_{\leq N}} q^{-\deg{P}} = \sum_{d \leq N} q^{-d} |\mc{P}_d| = \sum_{d \leq N} q^{-d}\Big(\frac{q^d}{d} + O(q^{d/2})\Big)\\
&= \sum_{d \leq N} \frac{1}{d} + O\Big(\sum_{d \leq N} q^{-d/2}\Big) = \log N + O(1).
\end{align*}
Moreover, if $s_0$ maximizes $|L(s,\tilde{f}_N)|$ on $\Re(s)=1$ and $q^{-s_0}=e(\theta)/q$ for some $\theta \in [0,1]$, then
\begin{align*}
\log|L(s_0,\tilde{f}_N)| &= \log\Big|\prod_{P \in \mc{P}_{\leq N}} \Big(1+f(P)e(\theta \deg{P})q^{-\deg{P}} + O\Big(\sum_{k \geq 2} q^{-k\deg{P}}\Big)\Big)\Big| \\
&= \sum_{P \in \mc{P}_{\leq N}} \text{Re}(f(P)e_{\theta}(P))q^{-\deg{P}} + O(1).
\end{align*}
It follows that
\begin{align*}
&M = \min_{\Re(s)=1} \log(2N/|L(s,\tilde{f}_N)|) = \min_{\theta \in [0,1]} \sum_{P \in \mc{P}_{\leq N}} q^{-\deg{P}}\Big(1-\text{Re}(f(P)e_{\theta}(P))\Big) +O(1)\\
&= \mc{D}_f(N) + O(1).
\end{align*}
The claim follows immediately.
\end{proof}
Using Lemma~\ref{thm_GRH_hayes}, we can also show that for any $N \geq 3$, there is at most one Hayes character $\tilde{\chi}$ with $\text{cond}_H(\tilde{\chi}) \leq N$ for which $\mc{D}_{f\bar{\tilde{\chi}}}(N)$ can be ``small'' in some sense. In what follows, we denote $\tilde{\chi}_1\sim \tilde{\chi}_2$ if $\tilde{\chi}_1$ and $\tilde{\chi}_2$ are induced by the same Hayes character, and otherwise write $\tilde{\chi}_1\nsim \tilde{\chi}_2$.

\begin{lem}[Repulsion of pretentious distance]\label{RepulsionFF}
Let $N \geq 3$.  Let $f\colon \mc{M} \to \mb{U}$ be multiplicative. Let $\tilde{\chi}_1\nsim \tilde{\chi}_2$ be two Hayes characters of conductors $\leq N$. Then  
$$\max\{\mc{D}_{f\bar{\tilde{\chi}}_1}(N),\mc{D}_{f\bar{\tilde{\chi}}_2}(N)\} \geq \Big(\frac{1}{4}-o(1)\Big) \log N.$$
\end{lem}

\begin{proof}
For each $j = 1,2$, let $\theta_j$ be an angle for which $\mc{D}_{f\bar{\tilde{\chi}}_j}(N) = \mb{D}(f,\tilde{\chi}_je_{\theta_j};N)^2$. Suppose first that $f$ is unimodular. Then, by the triangle inequality, we have
\begin{equation}\label{MaxChi}
2\max\{\mc{D}_{f\bar{\tilde{\chi}}_1}(N)^{1/2},\mc{D}_{f\bar{\tilde{\chi}}_2}(N)^{1/2}\} \geq \mb{D}(f\bar{\tilde{\chi}}_1,e_{\theta_1};N) + \mb{D}(f\bar{\tilde{\chi}}_2,e_{\theta_2};N) \geq \mb{D}(\tilde{\chi}_1,\tilde{\chi}_2e_{\theta_2-\theta_1};N),
\end{equation}
where the unimodularity of $f$ was used in the second inequality to write $(f\bar{\tilde{\chi}}_1)\overline{(f\bar{\tilde{\chi}}_2)} = \bar{\tilde{\chi}}_1\tilde{\chi}_2$. Now, by definition we have
$$
\mb{D}(\tilde{\chi}_1,\tilde{\chi}_2e_{\theta_2-\theta_1};N)^2 = \log N - \text{Re}\Big(\sum_{P \in \mc{P}_{\leq N}} \tilde{\chi}_1\bar{\tilde{\chi}_2}(P)e((\theta_1-\theta_2)\deg{P})q^{-\deg{P}}\Big) + O(1).
$$

Since $\tilde{\chi}_1\bar{\tilde{\chi}}_2$ has conductor $\leq N^2$ and it is non-principal, Lemma~\ref{lem_hayesdist} (with $\theta := \theta_1-\theta_2$ and $\tilde{\chi} := \tilde{\chi}_1\bar{\tilde{\chi}}_2$) yields
$$
\text{Re}\Big(\sum_{P \in \mc{P}_{\leq N}} \tilde{\chi}_1\bar{\tilde{\chi}_2}(P)e((\theta_1-\theta_2)\deg{P})q^{-\deg{P}}\Big) \ll \log\log N,
$$
and so it follows that
\begin{align}\label{eqq10}
\mb{D}(\tilde{\chi}_1,\tilde{\chi}_2e_{\theta_2-\theta_1};N)^2 \geq \log N - O(\log \log N).
\end{align}
Squaring both sides of~\eqref{MaxChi}, then inserting this last estimate into the result yields
$$
\max\{\mc{D}_{f\bar{\tilde{\chi}}_1}(N),\mc{D}_{f\bar{\tilde{\chi}}_2}(N)\} \geq \Big(\frac{1}{4}-o(1)\Big) \log N.
$$
Suppose then that $f$ is not unimodular. Define a random completely multiplicative function $\mbf{f} \colon  \mc{M} \to S^1$ (on some associated probability space) at irreducibles $P$ in such a way that $f(P) = \mb{E} \mbf{f}(P)$ for every irreducible $P$. By linearity of expectation it follows that for any multiplicative function $g$, we have
\begin{equation}\label{ExpStoch}
\mb{D}(f,g;N)^2 = \mb{E}\mb{D}(\mbf{f},g;N)^2.
\end{equation}
It follows from this and~\eqref{eqq10} that for any $\theta \in [0,1]$, we have
\begin{align*}
2\max\{\mathcal{D}_{f\bar{\tilde{\chi}}_1}(N),\mathcal{D}_{f\bar{\tilde{\chi}}_2}(N)\}&\geq 
\mb{D}(f\bar{\tilde{\chi}}_1,e_{\theta_1};N)^2 + \mb{D}(f\bar{\tilde{\chi}}_2,e_{\theta_2};N)^2\\
&\geq \frac{1}{2}(\mb{D}(f\bar{\tilde{\chi}}_1,e_{\theta_1};N) + \mb{D}(f\bar{\tilde{\chi}}_2,e_{\theta_2};N))^2\\
&=\frac{1}{2}(\mathbb{E}(\mb{D}(\mbf{f}\bar{\tilde{\chi}}_1,e_{\theta_1};N) + \mb{D}(\mbf{f}\bar{\tilde{\chi}}_2,e_{\theta_2};N)))^2\\
&\geq \Big(\frac{1}{2}-o(1)\Big)\log N, 
\end{align*}
and the claim follows.
\end{proof}

Combining Theorem~\ref{HalThmFF} with Lemma~\ref{RepulsionFF} immediately produces the following.
\begin{cor}[Sup norm estimate for weighted character sums]\label{cor_halrep}
Let $N \geq 3$. Let $f\colon\mc{M} \to \mb{U}$ be multiplicative. Let $\tilde{\chi}_1$ be the Hayes character of conductor $\leq N$ that minimizes\footnote{If there are several minimizers, we choose one of them arbitrarily.} the map $\tilde{\chi} \mapsto \mc{D}_{f\bar{\tilde{\chi}}}(N)$. Then
\begin{equation}\label{SmallCorrels}
\max_{\substack{\textnormal{cond}_H(\tilde{\chi})\leq N\\\tilde{\chi}\nsim \tilde{\chi_1}}} \Big|\frac{1}{q^N}\sum_{G \in \mc{M}_N} f(G)\bar{\tilde{\chi}}(G)\Big| \ll N^{-1/4+o(1)}.
\end{equation}
\end{cor}

Lastly, we will need the following simple upper bound estimate for non-negative multiplicative functions later in this paper (for a corresponding result about functions over the integers, see~\cite{HalRic}).

\begin{lem}[Halberstam-Richert bound in function fields]\label{HalRicFF}
Let $g \colon \mc{M} \to [0,\infty)$ be multiplicative, and let $N \geq 1$. Let $\kappa > 0$, and assume that for all $P \in \mc{P}$ and $k \geq 1$ we have $g(P) \leq \kappa$ and $g(P^k) \ll_{\e} q^{k\e\deg{P}}$ for any $\e > 0$. Then
$$\frac{1}{|\mc{M}_N|} \sum_{G \in \mc{M}_N} g(G) \ll \frac{(\kappa+1)}{N}\exp\Big(\sum_{P \in \mc{P}_{\leq N}}g(P)q^{-\deg{P}}\Big).$$
\end{lem}
\begin{proof}
Observe that for any $G \in \mc{M}_N$ we have $N = \sum_{\substack{P^k||G \\ P \in \mc{P}}} k\deg{P}$ (where $P^k \mid \mid B$ means $P^k\mid B$ and $P^{k+1}\nmid B$), and thus
\begin{align*}
\sum_{G \in \mc{M}_N} g(G) &= \frac{1}{N}\sum_{\substack{P^kB \in \mc{M}_N \\ (P,B) = 1 \\ P \in \mc{P}}} g(P^k)g(B)k \deg{P} \\
&\leq \frac{1}{N}\sum_{B \in \mc{M}_{\leq N}} g(B) \sum_{P \in \mc{P}_{N-\deg{B}}} g(P)\deg{P} + \frac{1}{N}\sum_{\substack{P^kB \in \mc{M}_N \\ P \in \mc{P} \\ k \geq 2}} g(P^k)g(B)k\deg{P} \\
&=: \mf{S}_1 + \mf{S}_2.
\end{align*}
Consider $\mf{S}_1$ first. Bounding $g(P) \leq \kappa$ for each $P \in \mc{P}_{N-\deg{B}}$ and then using the prime polynomial theorem, we have
$$
\sum_{P \in \mc{P}_{N-\deg{B}}} g(P)\deg{P} \leq \kappa \sum_{G \in \mc{M}_{N-\deg{B}}} \Lambda(G) \ll \kappa q^{N-\deg{B}},
$$
for every $B \in \mc{M}_{\leq N}$. Summing over such $B$ now gives
$$
\mf{S}_1 \ll \kappa \frac{q^N}{N}\sum_{B \in \mc{M}_{\leq N}} g(B)q^{-\deg{B}} \leq \kappa \frac{q^N}{N} \prod_{P \in \mc{P}_{\leq N}} \Big(\sum_{k \geq 0} g(P^k)q^{-k\deg{P}}\Big).
$$
Using the condition $g(P^k) \ll q^{\frac{1}{4}k\deg{P}}$ for $k \geq 2$, we get
$$
\sum_{P \in \mc{P}_{\leq N}}\sum_{k \geq 2} g(P^k)q^{-k\deg{P}} \ll \sum_{P \in \mc{P}_{\leq N}} q^{-\frac{3}{2}\deg{P}} \ll \sum_{d \leq N} q^{-d/2} \ll 1.
$$
Thus, rewriting the product over $P \in \mc{P}_{\leq N}$ as an exponential, we get
\begin{align}
&\prod_{P \in \mc{P}_{\leq N}} \Big(1+g(P)q^{-\deg{P}} + \sum_{k \geq 2} g(P^k)q^{-k\deg{P}}\Big)\nonumber\\
&\leq \prod_{P \in \mc{P}_{\leq N}} \Big(1+g(P)q^{-\deg{P}}\Big)\Big(1+\sum_{k \geq 2} g(P^k)q^{-k\deg{P}}\Big) \nonumber\\
&\ll \exp\Big(\sum_{P \in \mc{P}_{\leq N}} g(P)q^{-\deg{P}}\Big) \label{eq_prodtoexp}
\end{align}
Inserting this into our bound for $\mf{S}_1$ yields
$$
\mf{S}_1 \ll \kappa \frac{q^N}{N}\exp\Big(\sum_{P \in \mc{P}_{\leq N}}g(P)q^{-\deg{P}}\Big).
$$
To bound $\mf{S}_2$, we use the identity $1 = q^N/q^{\deg{B}+k\deg{P}}$ and the upper bound $k\deg{P} g(P^k) \ll q^{k\deg{P}/3}$ to get
\begin{align*}
\mf{S}_2 &= \frac{q^N}{N}\sum_{B \in \mc{M}_{\leq N}} g(B)q^{-\deg{B}}\sum_{\substack{P^k \in \mc{P}_{N-\deg{B}} \\ k \geq 2}} k\deg{P}g(P^k)q^{-k\deg{P}} \\
&\ll \frac{q^N}{N}\Big(\prod_{P_1\in \mc{P}_{\leq N}}\sum_{\ell\geq 0} g(P_1^{\ell}) q^{-\ell\deg{P_1}}\Big) \sum_{k\geq 2}\sum_{P_2\in \mathcal{P}} q^{-2k\deg{P_2}/3}.
\end{align*}
The sum over $P_2$ can be bounded by
$$ 
\sum_{d \geq 1} |\mc{P}_d| \sum_{k \geq 2} q^{-2kd/3} \ll \sum_{d \geq 1} q^{-d/3} \ll 1.
$$
Bounding the product in $P_1$ as in~\eqref{eq_prodtoexp}, we obtain
$$
\mf{S}_2 \ll \frac{q^N}{N}\exp\Big(\sum_{P \in \mc{P}_{\leq N}} g(P)q^{-\deg{P}}\Big).
$$
Combining this with the bound for $\mf{S}_1$ proves the claim.
\end{proof}
\subsection{An Involution for Monic Polynomials} \label{sec_involution}
Let $G\in \mc{M}$, and assume that $(G,t) = 1$. Following Keating and Rudnick (see~\cite[Section 5]{kea_rud}), we define\footnote{We could extend this definition to other polynomials by writing $G^{\ast}(t) = t^{\nu(G)} (G/t^{\nu(G)})^{\ast}$, where $\nu(G)$ denotes the order of vanishing of $G$ at $t = 0$. We could also modify the definition here when $G(0) \neq 0$ to give $G^{\ast} = G(0)^{-1}t^{\deg{G}} G(1/t)$, thus ensuring that $G^{\ast}$ is monic whenever $G$ is; however, we will not need this variant of the involution here.} 
$$
G^{\ast}(t) := t^{\deg{G}}G(1/t).
$$
The coefficients of $G^{\ast}$ are the same as those of $G$, but in reverse order. One can easily check that when $(G,t) = 1$ and $G(0) = 1$, $G^{\ast}$ is monic and $(G^{\ast})^{\ast} = G$. Since $\deg{G^{\ast}} = \deg{G}$, the $\ast$-map is an involution on the set of monic degree $N$ polynomials with $G(0) = 1$, for each $N \geq 1$. 

We observe, furthermore, that this involution is a multiplicative homomorphism on $\mc{M}$. Indeed, if $(FG,t) = 1$ then
$$
(FG)^{\ast}(t) = t^{\deg{FG}}FG(1/t) = t^{\deg{F}} F(1/t) \cdot t^{\deg{G}} G(1/t) = F^{\ast}(t) G^{\ast}(t).
$$
In light of this, we can define a corresponding involution on the space of multiplicative functions. That is, suppose that $f \colon \mc{M}\to \mb{U}$ is multiplicative. We define a map $f \mapsto f^{\ast}$ via $f^{\ast}(G) := f(G^{\ast})$ for all $(G,t) = 1$ with $G(0) = 1$, and $f^{\ast}(t^k) = 0$ for $k \geq 1$. Under a suitable extension of $f$ to $\mb{F}_q[t]$ (which we are free to choose, given that $f$ is only defined on $\mc{M}$ by assumption), we may define $f^\ast(G)$ at all monic $G$ irrespective of the condition $G(0) = 1$.  Then $f^{\ast}$ acts as a multiplicative function on $\mc{M}$, and if $g \colon  \mc{M} \to \mb{U}$ is a second such multiplicative function then $(fg)^{\ast} = f^{\ast}g^{\ast}$.

The next result, which is essentially contained in~\cite{kea_rud}, shows that the $\ast$-operation maps short intervals to arithmetic progressions modulo a power of $t$. 
\begin{lem}\label{ShortIntToAP}
Let $1 \leq H \leq N$ and $G_0 \in \mc{M}_N$. There is a reduced residue class $A$ modulo $t^{N-H+1}$ for which we have a bijection
$$
\{G \in \mc{M}_N \colon G \in I_H(G_0), (G,t) = 1\} \leftrightarrow \{\deg{F} = N \colon F \equiv A \bmod{ t^{N-H+1}}, F(0) = 1\};
$$
the bijection is furnished by the map $G \mapsto G^{\ast}$. Moreover, the class $A = A(G_0)$ depends at most on the first $N-H$ coefficients of $G_0$ after the leading coefficient.
\end{lem}
\begin{proof}
This is implied by~\cite[Lemma 5.1]{kea_rud}, using the fact that $I_H(G_0) = I_H(t^HG_0')$ whenever $\deg{G_0 - t^HG_0'} < H$. 
\end{proof}

The following lemma shows how the pretentious distance is affected by replacing a multiplicative function $f$ (whose behavior on $\mb{F}_q^{\times}$ is fixed) by its involution $f^{\ast}$. In the following, we fix a generator $\rho$ for $\mb{F}_q^{\times}$ and write $\nu_c$ to be the minimal non-negative integer such that $\rho^{\nu_c} = c$ whenever $c \in \mb{F}_q^{\times}$.

\begin{lem} \label{lem_dist_inv}
Let $\zeta \in \mu_{q-1}$ and let $f \colon  \mb{F}_q[t] \ra \mb{U}$ be a multiplicative function. Extend $f$ to $\mb{F}_q[t]$ so that $f(cF) = \zeta^{\nu_c}f(F)$ for all $c \in \mb{F}_q^{\times}$. Let $\chi$ be a Dirichlet character modulo $t^M$, for $M \geq 1$. Then there is a character $\xi = \xi(\zeta,\chi)$ modulo $t$ such that for any $N \geq 1$ we have 
$$
\mc{D}_{f\xi \bar{\chi^{\ast}}}(N) = \mc{D}_{f^{\ast}\bar{\chi}}(N) + O(1).
$$  
Moreover, if $\zeta = \chi^\ast(\rho)$ then $\xi \equiv 1$.
\end{lem}

\begin{rem} Note that even though we are only concerned with the values of $f$ on $\mc{M}$, in order to define $f^{\ast}$ we need to choose an extension of $f$ to $\mb{F}_q^{\times}$.  
\end{rem}

\begin{proof}
We claim first that there is a unique character $\xi$ modulo $t$ such that $\bar{\chi}^{\ast}(c)\zeta^{\nu_c}\xi(c) = 1$. To see this, note first that $\chi^{\ast}(1) = \chi(1) = 1$, so that $\chi^{\ast}(\rho) \in \mu_{q-1}$. The group of characters $\bmod{t}$ may be identified with that of $\mb{F}_q^{\times}$ via the isomorphism $(\mb{F}_q[t]/(t\mb{F}_q[t]))^{\times} \cong \mb{F}_q^{\times}$, so there is a $\xi$ modulo $t$ such that $\xi(\rho) = \bar{\zeta} \chi^{\ast}(\rho)$. Extending by complete multiplicativity, we obtain $\xi(c) = \chi^{\ast}(c)\zeta^{-\nu_c}$ for all $c \in \mb{F}_q^{\times}$. Moreover, if there is a second such character $\xi'$ modulo $t$ then we must have $\xi'(\rho) = \xi(\rho)$, and thus $\xi' = \xi$, as required.

We select $\xi$ to be the character modulo $t$ determined above. Let $\theta_0 \in [0,1]$ be fixed. We will show that $\mb{D}(f^{\ast},\chi e_{\theta_0}; N) = \mb{D}(f\xi,\chi^{\ast}e_{\theta_0};N) + O(1)$. By minimizing over $\theta_0$, we deduce the claimed estimate. 

First, note that if $R \in \mc{P}$, $R \neq t$, then $R^{\ast}/R(0) \in \mc{P}$. For if $R^{\ast} = AB$ with $\deg{A}\deg{B} > 0$ then as $(R^{\ast},t) = 1$ we have $R = A^{\ast}B^{\ast}$, with $\deg{A^{\ast}}\deg{B^{\ast}} = \deg{A}\deg{B} > 0$, a contradiction to irreduciblity. In particular, for each $c \in \mb{F}_q^{\times}$ and $d \geq 2$ we have a bijection
$$
\{R \in \mc{P}_d \colon R(0) = c\} \leftrightarrow \{R' \in \mc{P}_d  \colon  R'(0) = c^{-1}\},
$$
implied by the map $R \mapsto R' := R^{\ast}/R(0)$. Thus, we have
\begin{align*}
\mb{D}(f^{\ast},\chi e_{\theta_0};N)^2 &= \log N - \text{Re}\left(\sum_{2 \leq d \leq N} q^{-d}e(-\theta_0 d) \sum_{c \in \mb{F}_q^{\times}} \sum_{R \in \mc{P}_d \atop R(0) = c} f^{\ast}(R)\bar{\chi}(R)\right) + O(1) \\
&= \log N - \text{Re}\left(\sum_{2 \leq d \leq N} q^{-d}e(-\theta_0 d) \sum_{c \in \mb{F}_q^{\times}} \sum_{R' \in \mc{P}_d \atop R'(0) = c^{-1}} f(cR') \bar{\chi}^{\ast}(cR')\right) + O(1).
\end{align*}
Since $R'(0) = c^{-1}$ iff $R' \equiv c^{-1} \pmod{t}$, we get
$$
1_{R'(0) = c^{-1}} = \frac{1}{\phi(t)} \sum_{\xi' \pmod{t}} \xi'(c)\xi'(R'),
$$
and thus for each $2 \leq d \leq N$ we obtain
\begin{align*}
\sum_{c \in \mb{F}_q^{\times}} \sum_{R' \in \mc{P}_d \atop R'(0) = c^{-1}} f(cR') \bar{\chi}^{\ast}(cR') &= \frac{1}{\phi(t)} \sum_{\xi' \pmod{t}} \left(\sum_{c \in \mb{F}_q^{\times}} \bar{\chi}^{\ast}(c) \zeta^{\nu_c} \xi'(c)\right) \sum_{R' \in \mc{P}_d} f(R') \xi'(R') \bar{\chi}^{\ast}(R') \\
&= \sum_{R' \in \mc{P}_d} f(R')\xi(R')\bar{\chi}^{\ast}(R'),
\end{align*}
where $\xi$ is the character modulo $t$ constructed earlier. It follows then that
\begin{align*}
\mb{D}(f^{\ast},\chi e_{\theta_0};N)^2 &= \log N - \text{Re}\left(\sum_{2 \leq d \leq N} q^{-d}e(-\theta_0 d) \sum_{R' \in \mc{P}_d} f(R') \xi(R') \bar{\chi}^{\ast}(R')\right) + O(1) \\
&= \mb{D}(f\xi,\chi^{\ast}e_{\theta_0};N)^2 +O(1), 
\end{align*}
proving the first claim. \\
For the second, note that if $\chi^\ast(\rho) = \zeta$ then by construction we have $\xi(\rho) = 1$, and thus $\xi$ is trivial modulo $t$, as required.
\end{proof}

\subsection{Hayes Characters} \label{HayesSubSec}
We introduce here the following notation. Let $F,G \in \mb{F}_q[t]$ with $G \neq 0$, and consider $F/G \in \mb{F}_q(t)$. When $G$ is a power of $t$ this rational function admits a finite Laurent polynomial representation (in $1/t$)
$$
(F/G)(t) = \sum_{j = m_1}^{m_2} a_j t^{-j},
$$
where $m_1 \leq m_2$ are integers and $a_{m_1} \neq 0$. We then set $\lla F/G\rra := q^{-m_1}$. We note that the map $\lla \cdot \rra$ satisfies the ultrametric inequality $\lla f_1-f_2\rra \leq \max\{\lla f_1\rra, \lla f_2\rra\}$, with equality if $\lla f_1 \rra \neq \lla f_2\rra$, whenever $f_1,f_2 \in \mb{F}_q(t)$ have finite Laurent polynomial representations (in Section~\ref{sec_shortexp} we will extend this notation to all of $\mb{F}_q(t)$).

Let $\nu \geq 1$ and $M \in \mc{M}$. We define a relation $\mc{R}_{M,\nu}$ on $\mc{M}$ as follows: 
if $A,B \in \mc{M}$ then we say that
$$
A \equiv B \bmod{\mc{R}_{M,\nu}} \text{ if, and only if, } A \equiv B \bmod{M} \text{ and } \lla At^{-\deg{A}}-Bt^{-\deg{B}}\rra < q^{-\nu}. 
$$
This latter condition says that the leading $\nu+1$ coefficients of $A$ and $B$ are the same; in the particular case where $A,B \in \mc{M}_N$ for some $N$, it is equivalent to $\deg{A-B} < N-\nu$.  

It turns out that this defines an equivalence relation, and quotienting $\mc{M}$ by this relation yields a monoid whose multiplicative group of invertible elements is abelian. It thus admits a set of characters, which we call \emph{Hayes characters}. We will denote by $\mc{X}_{M,\nu}$ the collection of all Hayes characters associated with the pair $(M,\nu)$. A Hayes character $\tilde{\chi}$ is characterized by the property that it is constant on sets of the form 
$$
\{G \in \mc{M} \colon G \equiv C \bmod{M}\} \cap \{G \in \mc{M}  \colon  \lla Gt^{-\deg{G}}-Dt^{-\deg{D}}\rra < q^{-\nu}\},
$$
where $C$ is a reduced residue class modulo $M$, and $D\in \mc{M}_{\leq \nu}$. Any Hayes character in $\mc{X}_{M,\nu}$ can be uniquely decomposed as a product $\psi_M\xi_{\nu}$, where $\psi_M$ is a Dirichlet character modulo $M$, and $\xi_{\nu}$ is a \emph{short interval} character of \emph{length} $\text{len}(\xi_{\nu}) := \nu$, i.e., for $\ell=\nu$ the multiplicative function $\xi_{\nu}$ fixes the set $\{G \in \mc{M}_N \colon \lla Gt^{-\deg{G}}-Dt^{-\deg{D}}\rra < q^{-\ell}\}$ for all $D$, and the same does not hold for any $\ell<\nu$  (see, e.g.,~\cite[Theorem 8.6]{Hayes}). Thus this definition agrees with Definition~\ref{defn3}. We say that $\tilde{\chi} \in \mc{X}_{M,\nu}$ is \emph{primitive} if $\psi_M$ is primitive and $\nu > 0$, and imprimitive otherwise. Likewise, a Hayes character is \emph{non-principal} if it is either non-principal in the Dirichlet character aspect or if the length of its short interval character is non-zero. We define the \emph{Hayes conductor} of $\chi = \psi \xi \in \mc{X}_{M,\nu}$ by $\condH{\chi} := \cond{\psi} + \text{len}(\xi) := \deg{M} + \nu$. 

The group $\mathcal{X}_{M,\nu}$ has size $\phi(M)q^{\nu}$, and the orthogonality relations are given by
\begin{align}\label{ortho1}
\frac{1}{\phi(M)q^{\nu}}\sum_{A\bmod {R_{M,\nu}}}\widetilde{\chi_1}(A)\overline{\widetilde{\chi_2}(A)}=1_{\widetilde{\chi_1}=\widetilde{\chi_2}}    
\end{align}
and
\begin{align}\label{ortho2}
\frac{1}{\phi(M)q^{\nu}}\sum_{\widetilde{\chi}\in \mc{X}_{M,\nu}}\widetilde{\chi}(A)\overline{\widetilde{\chi}(B)}=1_{A\equiv B\bmod{R_{M,\nu}}};      
\end{align}
these are proved in~\cite{Hayes}.

An important fact about the relationship between Hayes characters and the $\ast$-involution from the previous subsection is the following.
\begin{lem} \label{lem_short_inv}
Let $n \geq 2$ and $k \geq 2$. Let $\chi$ be a Dirichlet character modulo $t^{k}$. Then there is a short interval character $\psi$ of length $k-1$ such that $\chi^{\ast}(G) = \psi(G)$ for all $G$ coprime to $t$. Moreover, if $\chi$ is non-principal then $\psi$ is also non-principal.
\end{lem}

\begin{proof}
It is enough to show that if $G_1,G_2 \in \mb{F}_q[t]$ satisfy $(G_1G_2,t) = 1$ and are close to each other in the sense that $\lla G_1t^{-\deg{G_1}} - G_2t^{-\deg{G_2}} \rra \leq q^{-k}$, then $\chi^{\ast}(G_1) = \chi^{\ast}(G_2)$. 

Without loss of generality suppose that $m_1 := \deg{G_1} \geq \deg{G_2} =: m_2$. Then we can write $G_1 = t^{m_1-m_2}G_2 + M$, where $r:= \deg{M} \leq m_1-k$. Writing $G_2(t) = \sum_{0 \leq j \leq m_2} b_jt^j$ and $M(t) = \sum_{0 \leq j \leq r} a_j t^j$ (with $a_0b_{m_2} \neq 0$ by assumption) we find
\begin{align*}
G_1^{\ast} &= \Big(\sum_{m_1-m_2 \leq j \leq m_1} b_{m_2-(m_1-j)} t^{j} + \sum_{0 \leq j \leq r} a_jt^j\Big)^{\ast}\\
&= t^{m_1}\Big(\sum_{m_1 - m_2 \leq j \leq m_1} b_{m_2-(m_1-j)} t^{-j} + \sum_{0 \leq j \leq r} a_jt^{-j}\Big) \\
&= t^{m_1-r}\sum_{0 \leq j \leq r} a_{r-j}t^{j} + \sum_{0 \leq l \leq m_2} b_{m_2-l}t^l \equiv \sum_{0 \leq l \leq m_2} b_{m_2-l}t^l  \bmod{t^{m_1-r}} \\
&\equiv G_2^{\ast} \bmod{t^{m_1-r}} \equiv G_2^{\ast} \bmod{t^k},
\end{align*}
since $k \leq m_1-r$. Thus, $\chi^{\ast}(G_1) - \chi^{\ast}(G_2) = \chi(G_1^{\ast}) - \chi(G_2^{\ast}) = 0$, as claimed.

For the second claim, if $\psi$ were principal then $\chi(G^{\ast}) = 1$ for all $(G^{\ast},t) = 1$. The set $\{G \in \mb{F}_q[t] \colon G(0) \neq 0\}$ is invariant under the involution, so this would imply that $\chi(G) = 1$ whenever $G(0) \neq 0$; but since $\chi(G) = 0$ whenever $G(0) = 0$, this implies that $\chi(G) = 1_{(G,t^k) = 1}$, which implies that $\chi$ is principal, and the claim follows. 
\end{proof}

\begin{rem}\label{rem_chiast_att}
Note that if $\chi$ is a character modulo $t^k$ then the previous lemma does not prescribe a value for $\chi^{\ast}(t)$. However, in keeping with our convention $f^{\ast}(t) = 0$ for multiplicative functions $f$ we shall set $\chi^{\ast}(t) = 0$. In any case, this particular definition will play no significant role in the sequel.
%since $\chi^{\ast}$ is equal to a short interval character $\psi$ for $(G,t) = 1$, we genuinely have $\chi^{\ast} = \psi$ if, and only if, $\chi^{\ast}(t) = \psi(t) = \psi(1) = 1$ (if $\psi$ is principal this is simply because $\psi$ is equal to 1 everywhere, and if $\psi$ is non-principal then its length is $> 0$, and $t$ and $1$ have the same string of coefficients). Our convention throughout the paper is thus that $\chi^{\ast}(t) = 1$ for any character $\chi$ modulo a power of $t$.
\end{rem}

We shall distinguish between the following notions of non-pretentiousness.
\begin{def1} \label{def_hayesnp}
Let $N \geq 1$. Let $f\colon \mc{M} \to \mb{U}$ be multiplicative. We say that $f$ is \emph{Hayes non-pretentious to level $W=W(N)$} if, as $N \ra \infty$,
$$\min_{w \leq W} \min_{\substack{\psi \bmod{M} \\ M \in \mc{M}_{w}}}\min_{\substack{\xi \text{ short} \\ \text{len}(\xi) \leq N}} \mc{D}_{f\bar{\chi \xi}}(N) \to \infty.$$
We say that $f$ is \emph{Dirichlet non-pretentious} to level $W=W(N)$ if, as $N \ra \infty$,
$$\min_{w \leq W} \min_{\substack{\psi \bmod{M} \\ M \in \mc{M}_{w}}} \mc{D}_{f\bar{\chi}}(N) \to \infty.$$
\end{def1}

An immediate corollary of Lemma~\ref{lem_short_inv} relating to Hayes non-pretentiousness (and utilized in Section~\ref{sec_shortexp}) is the following.
\begin{cor}[Hayes non-pretentiousness implies Dirichlet non-pretentiousness of dual]\label{StarNonPret}
Let $N \geq 1$, and let $W = W(N) \leq N$. Let $f\colon \mb{F}_q[t] \to \mb{U}$ be multiplicative and even, i.e., $f(cG) = f(G)$ for all $c \in \mb{F}_q^{\times}$. Then
$$
\min_{\substack{\psi \bmod{M} \\ M \in \mc{M}_{\leq W(N)+1}}} \min_{\substack{\xi \textnormal{ short} \\ \textnormal{len}(\xi) \leq N}} \mc{D}_{f\psi\bar{\xi}}(N) \leq \min_{M \in \mc{M}_{\leq W(N)}}\,\, \min_{\psi \bmod{M}} \min_{\substack{\chi \bmod{t^{\nu}} \\ 1 \leq \nu \leq N}} \mc{D}_{(f\psi)^{\ast}\bar{\chi}}(N) + O(1).
$$
In particular, if $f$ is Hayes non-pretentious to level $W' := W + 1$ then 
$$
\lim_{N \to \infty} \min_{\substack{\psi \bmod{M} \\ M \in \mc{M}_{\leq W(N)}}} \min_{\substack{\chi \bmod{t^{\nu}} \\ 1 \leq \nu \leq N}} \mc{D}_{(f\psi)^{\ast}\bar{\chi}}(N) = \infty.
$$
\end{cor}

\begin{proof}
Let $N$ be large and let $\psi \bmod{M}$ with $\deg{M} \leq W(N)$ and $\chi \bmod{t^{\nu}}$ with $1 \leq \nu \leq N$ be chosen such that
$$
\mc{D}_{(f\psi)^{\ast}\bar{\chi}}(N) = \min_{M'\in \mathcal{M}_{\leq W(N)}} \min_{\psi' \bmod{M'}} \min_{\substack{\chi' \bmod {t^{\nu'}} \\ 1 \leq \nu' \leq N}} \mc{D}_{(f\psi')^{\ast}\bar{\chi'}}(N).
$$
Since $f$ is even and $\psi(c) \in \mu_{q-1}$ for all $c \in \mb{F}_{q}^\times$ we may apply Lemma~\ref{lem_dist_inv} to conclude that there is a character $\xi \bmod{t}$, depending on $\psi$ and $\chi$, such that
$$
\mc{D}_{f\psi \xi \bar{\chi^{\ast}}}(N) = \mc{D}_{(f\psi)^{\ast}\bar{\chi}}(N) + O(1).
$$
By Lemma~\ref{lem_short_inv}, $\chi^{\ast}$ coincides with a short interval character of length $\nu-1$ at all primes $P\neq t$, so that $\psi\xi \bar{\chi^{\ast}}$ coincides at all $P \in \mc{P} \bk \{t\}$ with a Hayes character whose Dirichlet part has conductor $\leq \deg{M t} \leq W(N)+1$ and whose short interval character part has conductor at most $N$. It follows then that 
$$
\min_{M \in \mc{M}_{\leq W(N)+1}} \min_{\psi \bmod{M}} \min_{\substack{\xi \text{ short} \\ \text{len}(\xi) \leq N}}\mc{D}_{f\bar{\chi \xi}}(N) \leq \mc{D}_{f\psi \xi \bar{\chi^{\ast}}}(N) + O(1) \leq \mc{D}_{(f\psi)^{\ast}\bar{\chi}}(N) + O(1).
$$
This implies the first claim. The second claim follows upon taking $N \to \infty$ and using the definition of Hayes non-pretentiousness.
\end{proof}

\section{Preliminaries II:  Character Sums and Sieve Estimates} \label{sec_prelim2}

Beginning in this section we set out to prove (a generalization of) Theorem~\ref{thm_mr_real}, as well as Theorem~\ref{LogEllFF1}. We collect together the main general results we shall use for this purpose. Most of these are simple translations of the corresponding result in the number field setting, but we have not managed to locate such translations in the literature.

\begin{rem}\label{remshort}
For brevity and to simplify notation, all of the lemmas below are stated for sums of Dirichlet characters, but  as we will note in Section~\ref{MRTHMSec}, all of them work equally well if $\chi\bmod Q$ is replaced with $\chi\in \mathcal{X}_{1,\nu}$ (that is, we are summing over short interval characters of length $\nu$), and $\deg{Q}$ is replaced with $\nu$ and $\phi(Q)$ is replaced with $q^{\nu}$.
\end{rem}

\subsection{Large Sieve Estimates in Function Fields}
\begin{lem}[$L^2$ mean value theorem] \label{L2MVT}
Let $N \geq 1$. Let $\{a_G\}_{G \in \mc{M}_N} \subset \mb{C}$, and let $Q \in \mc{M}$. Then
$$\sum_{\chi \bmod{Q}} \Big|\sum_{G \in \mc{M}_N} a_G \chi(G)\Big|^2 \leq 2\Big(\phi(Q)q^{N-\deg{Q}}+\phi(Q)\Big)\sum_{\substack{G \in \mc{M}_N \\ (G,Q) = 1}}|a_G|^2.$$
\end{lem}
\begin{rem}
The short interval analogue of this lemma reads as 
\begin{align*}
 \sum_{\xi\in \mathcal{X}_{1,\nu}} \Big|\sum_{G \in \mc{M}_N} a_G \xi(G)\Big|^2 \leq 2\Big(q^{\nu}q^{N-\nu}+q^{\nu}\Big)\sum_{G \in \mc{M}_N}|a_G|^2.   
\end{align*}
All the lemmas that follow in this section have short interval formulations in a completely analogous fashion.
\end{rem}

\begin{proof}
Denote the left-hand side by $\Sigma$. Expanding the square and swapping orders of summation yields
$$\Sigma = \sum_{G,G' \in \mc{M}_N} a_G\bar{a_{G'}}\sum_{\chi \bmod{Q}} \chi(G)\bar{\chi}(G') = \phi(Q)\Big(\sum_{\substack{G \in \mc{M}_N \\ (G,Q) = 1}} |a_G|^2 + \sum_{\substack{G,G' \in \mc{M}_N \\ G\equiv G' \bmod{Q}\\ G \neq G' , (GG',Q) = 1}} a_G\bar{a_{G'}}\Big).$$
Bounding the second sum trivially, using the AM-GM inequality in the form $|a_Ga_{G'}| \leq \frac{1}{2}\Big(|a_G|^2 + |a_{G'}|^2\Big)$ and invoking symmetry in $G$ and $G'$, we get
$$\Sigma \leq \phi(Q)\sum_{\substack{G \in \mc{M}_N \\ (G,Q) = 1}} |a_G|^2\Big(1 + \sum_{\substack{G' \in \mc{M}_N \\ Q|(G'-G)}} 1\Big).$$
Since $\deg{G'-G} \leq N$ for each $G \in \mc{M}_N$, and the number of polynomials in $\mc{M}_{\leq N}$ divisible by $Q$ is precisely $|\mc{M}_{\leq N-\deg{Q}}| \leq 2q^{N-\deg{Q}}$, it follows that
$$\Sigma \leq \Big(\phi(Q) + \phi(Q)|\mc{M}_{\leq N-\deg{Q}}|\Big)\sum_{\substack{G \in \mc{M}_N \\ (G,Q) = 1}} |a_G|^2 \leq 2\Big(\phi(Q) + \phi(Q)q^{N-\deg{Q}}\Big) \sum_{\substack{G \in \mc{M}_N \\ (G,Q) = 1}} |a_G|^2,$$
as claimed.
\end{proof}

\begin{lem} [Hal\'{a}sz--Montgomery lemma] \label{HalMonInt}
Let $N \geq 1$. Let $\{a_G\}_{G \in \mc{M}_N} \subset \mb{C}$, and let $Q \in \mc{M}$, $\deg{Q}\leq (1+o(1))N$. Let $\Xi \subseteq \mc{X}_Q$. Then
$$\sum_{\chi \in \Xi} \Big|\sum_{G \in \mc{M}_N} a_G\chi(G)\Big|^2 \ll \Big(\phi(Q)q^{N-\deg{Q}} + |\Xi|q^{(1/2+o(1))N}\Big)\sum_{\substack{G \in \mc{M}_N \\ (G,Q) = 1}} |a_G|^2.$$
\end{lem}
\begin{proof}
We may obviously assume that $\Xi \neq \emptyset$, since otherwise the claim is trivial. Moreover, by duality (see e.g.,~\cite[Lemma 10]{mr-annals}), it suffices to show that for any set of coefficients $\{c_{\chi}\}_{\chi \in \Xi} \subset \mb{C}$ we have
$$
\sum_{G \in \mc{M}_N} \Big|\sum_{\chi \in \Xi} c_{\chi}\chi(G)\Big|^2 \ll \Big(\phi(Q) q^{N-\deg{Q}} + |\Xi|q^{\deg{Q}/2}\Big) \sum_{\chi \in \Xi} |c_{\chi}|^2.
$$
Expanding the square in the left-hand side and swapping the order of summations, we get
$$
\sum_{\chi_1,\chi_2 \in \Xi} c_{\chi_1}\bar{c_{\chi_2}} \sum_{G \in \mc{M}_N} \chi_1\bar{\chi}_2(G).
$$
The diagonal contribution with $\chi_1 = \chi_2$ yields 
$$
|\{G \in \mc{M}_N  \colon  (G,Q) = 1\}| \sum_{\chi \in \Xi}|c_{\chi}|^2 \ll \phi(Q)q^{N-\deg{Q}} \sum_{\chi \in \Xi}|c_{\chi}|^2.
$$ 
When $\chi_1 \neq \chi_2$, $\chi_1\bar{\chi}_2$ is non-principal, so by Lemma~\ref{hayesbound} we have
$$\sum_{\substack{\chi_1,\chi_2 \in \Xi \\ \chi_1 \neq \chi_2}} |c_{\chi_1}c_{\chi_2}| \Big|\sum_{G \in \mc{M}_N} \chi_1\bar{\chi_2}(G)\Big| \ll q^{(1/2+o(1))N} \sum_{\substack{\chi_1,\chi_2 \in \Xi \\ \chi_1\neq \chi_2}}|c_{\chi_1}||c_{\chi_2}|.$$
Applying AM-GM as in the proof of the previous lemma, the sum above is bounded by $|\Xi| \sum_{\chi \in \Xi}|c_{\chi}|^2$. Putting everything together, this proves the claim.
\end{proof}
\begin{lem} [Hal\'{a}sz--Montgomery lemma for primes] \label{HalMonPrim}
Let $N \geq 1$. Let $\{a_P\}_{P \in \mc{P}_N} \subset \mb{C}$, and let $Q \in \mc{M}$. For any $\Xi \subseteq \mc{X}_Q$ we have
$$\sum_{\chi \in \Xi} \Big|\sum_{P \in \mc{P}_N} a_P\chi(P)\Big|^2 \ll \Big(\frac{q^N}{N} + \deg{Q}\frac{q^{N/2}}{N}|\Xi|\Big)\sum_{P \in \mc{P}_N} |a_P|^2.$$
\end{lem}
\begin{proof}
We apply duality, as in the proof of Lemma~\ref{HalMonInt}. Given a sequence $\{c_{\chi}\}_{\chi \in \Xi} \subset \mb{C}$, we bound $1_{G \in \mc{P}_N} \leq N^{-1}\Lambda(G)$ to obtain
$$
\sum_{P \in \mc{P}_N} \Big|\sum_{\chi \in \Xi} c_{\chi}\chi(P)\Big|^2 \leq \sum_{G \in \mc{M}_N} \frac{\Lambda(G)}{N} \Big|\sum_{\chi \in \Xi} c_{\chi}\chi(G)\Big|^2 = \frac{1}{N}\sum_{\chi_1,\chi_2 \in \Xi} c_{\chi_1}\bar{c_{\chi_2}} \sum_{G \in \mc{M}_N} \Lambda(G) \chi_1\bar{\chi_2}(G).
$$
When $\chi_1 = \chi_2$, the prime polynomial theorem gives $\sum_{G \in \mc{M}_N} \Lambda(G) \ll q^N$, whence the diagonal contribution to the sum becomes $(q^N/N)\sum_{\chi \in \Xi} |c_{\chi}|^2$. 

When $\chi_1 \neq \chi_2$, we may apply Lemma~\ref{thm_GRH_hayes} to give
$$
\sum_{G \in \mc{M}_N} \Lambda(G)\chi_1\bar{\chi_2}(G) \ll \deg{Q}q^{N/2}.
$$
It follows that
$$
\sum_{\substack{\chi_1,\chi_2 \in \Xi \\ \chi_1 \neq \chi_2}} |c_{\chi_1}||c_{\chi_2}|\Big|\sum_{G \in \mc{M}_N} \Lambda(G)\chi_1\bar{\chi_2}(G)\Big| \ll \deg{Q}q^{N/2} |\Xi|\sum_{\chi \in \Xi} |c_{\chi}|^2,
$$
upon applying AM-GM and using symmetry, as before.

Combined with the diagonal contribution, we get
$$\sum_{P \in \mc{P}_N} \Big|\sum_{\chi \in \Xi} c_{\chi}\chi(P)\Big|^2 \ll \Big(\frac{q^N}{N} + \deg{Q}\frac{q^{N/2}}{N}|\Xi|\Big) \sum_{\chi \in \Xi}|c_{\chi}|^2.$$
Invoking duality as discussed above, the claim follows.
\end{proof}
\begin{lem}[A large values estimate]\label{LargeValues}
Let $N,Z \geq 1$. Let $\{a_P\}_{P \in \mc{P}_N} \subset \mb{U}$, and let $Q \in \mc{M}$, with $\phi(Q) \geq q^N$. Then
$$\Big|\Big\{\chi \bmod{Q}  \colon  \frac{1}{q^N}\Big|\sum_{P \in \mc{P}_N} a_P\chi(P)\Big| \geq \frac{1}{Z}\Big\}\Big| \ll \exp\Big(\frac{\log (q^N\phi(Q))}{N\log q} \Big(2\log\Big(\frac{2\log \phi(Q)}{N\log q}\Big) + \log\Big(\frac{2Z^2}{N}\Big)\Big)\Big).$$
\end{lem}
\begin{proof}
The proof is essentially the same as in the number fields case~\cite[Lemma 8]{mr-annals}. Let $k := \llf \frac{\log \phi(Q)}{N\log q}\rrf + 1$. Let $\mc{N}$ denote the cardinality of the set of characters on the left-hand side. By Chebyshev's inequality, we have
\begin{align}
\mc{N} &\leq \Big(\frac{Z}{q^N}\Big)^{2k}\sum_{\chi \bmod{Q}}\Big|\sum_{P \in \mc{P}_N} a_P\chi(P)\Big|^{2k} = \Big(\frac{Z}{q^N}\Big)^{2k}\sum_{\chi \bmod{Q}}\Big|\Big(\sum_{P \in \mc{P}_N} a_P\chi(P)\Big)^k\Big|^{2} \nonumber\\
&= \Big(\frac{Z}{q^N}\Big)^{2k} \sum_{\chi \bmod{Q}} \Big|\sum_{G \in \mc{M}_{kN}} b_G \chi(G)\Big|^2, \label{2kPow}
\end{align}
where we have defined
$$
b_G := \sum_{\substack{P_1\cdots P_k = G \\ P_j \in \mc{P}_N \ \forall j}} a_{P_1} \cdots a_{P_k}.
$$
Applying Lemma~\ref{L2MVT}, we get
\begin{align*}
\sum_{\chi \bmod{Q}} \Big|\sum_{G\in \mc{M}_{kN}} b_G\chi(G)\Big|^2&\ll \Big(\phi(Q) + q^{kN}\Big)\sum_{G \in \mc{M}_{kN}} |b_G|^2\\
&\ll q^{kN}\sum_{\substack{P_1 \cdots P_k = Q_1\cdots Q_k \\ P_i,Q_j \in \mc{P}_N}} a_{P_1} \cdots a_{P_k}\bar{a_{Q_1}\cdots a_{Q_k}},
\end{align*}
according to our choice of $k$. Since the $P_i$ and $Q_j$ are irreducible, up to permutation we have $P_i = Q_i$ for all $1 \leq i \leq k$, and thus by the prime polynomial theorem
$$
\sum_{\substack{P_1 \cdots P_k = Q_1 \cdots Q_k \\ P_i,Q_j \in \mc{P}_N}} a_{P_1} \cdots a_{P_k} \bar{a_{Q_1} \cdots a_{Q_k}} \leq (k!)^2\Big(\sum_{P \in \mc{P}_N} |a_P|^2\Big)^k \ll (k!)^2(1.1q^N/N)^k.
$$
Inserting this into our mean value estimate, we get that
$$
\sum_{\chi \bmod{Q}} \Big|\sum_{P \in \mc{P}_N} a_P\chi(P)\Big|^{2k} \ll \Big(\frac{q^{2N}}{N}\Big)^k1.1^{k}(k!)^2.
$$
Combining this with~\eqref{2kPow} and using $\log k! \leq k\log k$ for $k \geq 2$, we find that
$$
\mc{N} \ll 1.1^k(k!)^2(Z^2/N)^k \ll \exp\Big(\Big(1+\frac{\log \phi(Q)}{N\log q}\Big) \Big(2\log\Big(\frac{2\log \phi(Q)}{N\log q}\Big) + \log(2Z^2/N)\Big)\Big).
$$
This implies the claim.
\end{proof}

\begin{lem}[A moment computation]\label{Lengthen}
Let $1 \leq d \leq m \leq N$. Let $\{a_P\}_{P \in \mc{P}_{d}}, \{b_G\}_{G \in \mc{M}_{N-m}} \subset \mb{U}$. Set
\begin{align*}
U(\chi) &:= \frac{1}{d|\mc{P}_d|}\sum_{P \in \mc{P}_d} a_P\chi(P),\\
V(\chi) &:= \frac{1}{|\mc{M}_{N-m}|} \sum_{G \in \mc{M}_{N-m}} b_G\chi(G).
\end{align*}
Set $\ell := \lceil m/d\rceil$. Then for any $Q \in \mc{M}$, we have
$$
\sum_{\chi \bmod{Q}} |U(\chi)^{\ell}V(\chi)|^2 \ll \Big(\phi(Q)q^{-N} + \phi(Q)q^{-\deg{Q}}\Big)\ell ^{2\ell}.
$$
\end{lem}
\begin{proof}
This is similar to~\cite[Lemma 13]{mr-annals}. Expanding out the product for each $\chi$, we have
\begin{align*}
U(\chi)^{\ell}V(\chi) = \frac{1}{d^{\ell}|\mc{P}_d|^{\ell}|\mc{M}_{N-m}|} \sum_{M \in \mc{M}_{N-m + \ell d}} \chi(M) \Big(\sum_{\substack{GP_1 \cdots P_{\ell} = M \\ P_j \in \mc{P}_d \ \forall j}} a_{P_1}\cdots a_{P_{\ell}}b_G\Big).
\end{align*}
We denote by $g(M)$ the bracketed sum on the right-hand side. Taking squares, summing over $\chi \bmod{Q}$ and then applying Lemma~\ref{L2MVT} (and the prime polynomial theorem) yields
\begin{align*}
\sum_{\chi \bmod{Q}}|U(\chi)^{\ell}V(\chi)|^2 &\ll \phi(Q)\Big(1 + q^{N-\deg{Q}-m+\ell d}\Big) \frac{1}{d^{2\ell}|\mc{P}_d|^{2\ell}|\mc{M}_{N-m}|^2} \sum_{M \in \mc{M}_{N-m+\ell d}} |g(M)|^2 \\
&\ll  \Big(\phi(Q)q^{-N} + \phi(Q)q^{-\deg{Q}}\Big)\frac{1.1^{\ell}}{|\mc{M}_{N-m+\ell d}|}\sum_{M \in \mc{M}_{N-m+\ell d}} |g(M)|^2.
\end{align*}
Now, by the triangle inequality we can bound $g$ as
$$
|g(M)| \leq \sum_{\substack{G P_1 \cdots P_{\ell} = M \\ P_j \in \mc{P}_d \forall j}} 1 \leq (\ell !) 1\ast \gamma (M) =: (\ell !)\widetilde{g}(M),
$$
where $\gamma$ is the indicator function of monic polynomials all of whose prime factors belong to $\mc{P}_d$; note that on prime powers, $\widetilde{g}(P^k) = 1+k 1_{\mc{P}_d}(P)$, which is $\ll_{\e} q^{\e k \deg{R}}$ for any $\e > 0$ and $k \geq 1$, and $\widetilde{g}(P) \leq 2$ for all irreducibles $P$. We may thus apply Lemma~\ref{HalRicFF} to get that
\begin{align*}
\frac{1}{|\mc{M}_{N-m+d\ell}|} \sum_{M \in \mc{M}_{N-m+\ell d}} \widetilde{g}(M)^2 &\ll \frac{1}{N-m+\ell d}\exp\Big(\sum_{P \in \mc{P}_{\leq N-m+\ell d}}\widetilde{g}(P)^2q^{-\deg{P}}\Big) \\
&\ll \exp\Big(\sum_{P \in \mc{P}_{d}} (2^2-1)q^{-\deg{P}}\Big)  \ll 1.
\end{align*}
Inserting this into the above estimate, we get
$$
\sum_{\chi \bmod{Q}}|U(\chi)^{\ell}V(\chi)|^2 \ll \Big(\phi(Q)q^{-N}+\phi(Q)q^{-\deg{Q}}\Big)\ell^{2\ell},
$$
as claimed.
\end{proof}

\subsection{Sieve Bounds in Function Fields}
Our next result shows that most monics have irreducible factors whose degrees belong to prescribed ranges, provided these ranges are large enough.
\begin{lem} \label{SieveErat}
Let $P < Q$. Then
$$
|\{ G \in \mc{M}_N  \colon  R \in \mc{P} \text{ such that } R|G \Rightarrow \deg{R} \notin [P,Q]\}| \ll \frac{P}{Q}q^N.
$$
\end{lem}
\begin{proof}
Let $g$ denote the indicator function for the set on the left-hand side. Then $0 \leq g \leq 1$ and $g$ is multiplicative. By Lemma~\ref{HalRicFF}, the left-hand side is
\begin{align*}
\sum_{G \in \mc{M}_N} g(G) 
&\ll \frac{q^N}{N} \exp\Big(\sum_{\substack{R \in \mc{P}_{\leq N} \\ \deg{R} \notin [P,Q]}} q^{-d}\Big) \ll q^N\exp\Big(-\sum_{P \leq d \leq Q} q^{-d} |\mc{P}_d| \Big) \ll \frac{P}{Q} q^N,
\end{align*}
as claimed.
\end{proof}

\begin{def1} \label{def_SPQ}
Let $J \geq 1$, and let $\mbf{P} := \{P_j\}_{1 \leq j \leq J}$ and $\mbf{Q} := \{Q_j\}_{1 \leq j \leq J}$ be collections of parameters satisfying $P_j < P_{j+1}$, $Q_j < Q_{j+1}$ and $P_j < Q_j$ for all $j$. We define the set $\mc{S}_{\mbf{P},\mbf{Q}}(N)$ by
$$\mc{S}_{\mbf{P},\mbf{Q}} := \{G \in \mc{M}  \colon  \forall \ 1 \leq j \leq J \ \exists \ d \in [P_j,Q_j], R \in \mc{P}_d \text{ such that } R|G\}.$$
If $J = 1$ then, for convenience, we write $S_{P_1,Q_1} = S_{\mbf{P},\mbf{Q}}$.
\end{def1}

We will be able to restrict character-twisted sums over monic polynomials to monics belonging to sets of the form $\mc{S}_{\mbf{P},\mbf{Q}}(N)$, on average. 
\begin{lem}\label{RestricToS}
Let $N \geq 1$, and let $Q \in \mc{M}$ with $\deg{Q} \leq N$. Let $\Xi \subseteq \mc{X}_Q$ be a set of characters modulo $Q$, and let $f\colon \mc{M} \to \mb{U}$ be multiplicative. Then
$$
\sum_{\chi \in \Xi} \Big|\frac{1}{|\mc{M}_N|}\sum_{G \in \mc{M}_N} f(G)\bar{\chi}(G) \Big|^2 \ll \sum_{\chi \in \Xi}\Big|\frac{1}{|\mc{M}_N|}\sum_{\substack{G \in \mc{M}_N \\ G \in \mc{S}_{\mbf{P},\mbf{Q}}}} f(G)\bar{\chi}(G)\Big|^2 + \phi(Q)q^{-\deg{Q}}\sum_{1 \leq j \leq J} \frac{P_j}{Q_j}.
$$
\end{lem}
\begin{proof}
Given a map $g  \colon  \mb{F}_q[t] \to \mb{C}$, set $M_g(N) := \frac{1}{|\mc{M}_N|}\sum_{G \in \mc{M}_N} g(G)$. For each $\chi \bmod M$ we have
$$
|M_{f\bar{\chi}}(N)|^2 \leq 2|M_{f\bar{\chi} 1_{\mathcal{S}_{\mbf{P},\mbf{Q}}}}(N)|^2 + 2|M_{f\bar{\chi} 1_{\mathcal{S}_{\mbf{P},\mbf{Q}}^c}}(N)|^2. 
$$
Summing the first of these terms over $\chi \in \Xi$ gives the first term in the estimate. Summing the second term over $\chi$ and applying Lemma~\ref{L2MVT} gives
$$
\sum_{\chi \in \Xi}|M_{f\chi1_{\mathcal{S}_{\mbf{P},\mbf{Q}}^c}}(N)|^2 \leq \sum_{\chi \bmod{Q}} |M_{f\chi1_{\mc{S}_{\mbf{P},\mbf{Q}}^c}}(N)|^2 \ll \phi(Q) \Big(q^{N-\deg{Q}} + 1\Big)\frac{1}{|\mc{M}_N|^2}\sum_{\substack{G \in \mc{M}_N \\ G \notin \mathcal{S}_{\mbf{P},\mbf{Q}}}} 1.
$$
By the union bound and Lemma~\ref{SieveErat}, we have
\begin{align*}
\frac{1}{|\mc{M}_N|}\sum_{\substack{G \in \mc{M}_N \\ G \notin \mathcal{S}_{\mbf{P},\mbf{Q}}}} 1 &\leq \sum_{1 \leq j \leq J} \frac{1}{|\mc{M}_N|}|\{G \in \mc{M}_N  \colon  R \in \mc{P}, R|G \Rightarrow \deg{R} \notin [P_j,Q_j]\}| \\
&\ll \sum_{1 \leq j \leq J} \frac{P_j}{Q_j}.
\end{align*}
This implies the claim.
\end{proof}
We will also need the following estimate for smooth (otherwise known as friable) polynomials, i.e., polynomials with no irreducible factors of large degree. For $1 \leq M \leq N$, we write
$$\mc{S}(N,M) := \{G \in \mc{M}_N  \colon  R \in \mc{P} \text{ and } R|G \Rightarrow \deg{R} \leq M\}.$$
\begin{lem} \label{SmoothsFF}
Let $1 \leq M \leq N$. Then for some absolute constant $c>0$ we have
$$|\mc{S}(N,M)| \ll q^N\exp(-cN/M).$$
\end{lem}

\begin{proof}
This follows from~\cite{warlimont}.
\end{proof}

\begin{lem}[Selberg upper bound sieve in function fields] \label{Selberg}
Let $1 \leq y \leq z, H \leq N$ and let $\mc{A} \subseteq \mc{M}_N$. Put 
$$
\mathfrak{P}_{y,z} = \prod_{Q \in \mc{P}_{\leq z} \bk \mc{P}_{\leq y}}Q.
$$ 
Suppose $g$ is a multiplicative function supported on squarefree monic polynomials such that for each $D \in \mc{M}$ squarefree with $D \in \mc{M}_{\leq H}$,
\begin{equation}\label{DIVBYD}
\sum_{\substack{G \in \mc{A} \\ D|G}} 1 = g(D)|\mc{A}| + r_D(\mc{A}).
\end{equation}
Put $J = J(H) = \sum_{\substack{D|\mf{P}_{y,z} \\ \text{deg}(D) \leq H}} \prod_{\ss{R \in \mc{P} \\ R|D}} g(R)/(1-g(R))$. Then
$$
\sum_{\substack{G \in \mc{A} \\ (G,\mathfrak{P}_{y,z}) = 1}} 1 \leq |\mc{A}|J^{-1} + \sum_{\text{deg}(D) \leq H} \tau_3(D)|r_D(\mc{A})|,
$$
where $\tau_3(D) = \sum_{\substack{A,B,C \in \mc{M} \\ ABC = D}} 1$.
\end{lem}
\begin{proof}
This follows from~\cite[Theorem 1]{webb} (take $\mc{P} := \mc{P}_{\leq z} \bk \mc{P}_{\leq y}$ and $\mc{D} := \{D \in \mc{D}  \colon  \deg{D} \leq H\}$, which is divisor closed, as needed according to the hypotheses there).
\end{proof}

We have the following useful corollary.
\begin{cor}[Additive energy of irreducible polynomials] \label{PRIM4TUP}
Let $H \geq 1$. If $M\in \mb{F}_q[t]$ has $\deg{M} < H$ then
$$
|\{(P_1,P_2,P_3,P_4) \in \mc{P}_H^4  \colon  P_1 + P_2 - P_3 - P_4 = M\}| \ll q^{3H}/H^4.
$$
\end{cor}
\begin{proof}
We begin by considering the case $M = 0$. 
Given $G \in \mb{F}_q[t]$ of degree $H$, let $r(G)$ denote the number of representations of $G$ as a sum of two irreducible polynomials of degree $ \leq H$. Note that if $q = 2$ then $r(G) > 0$ only when $\text{deg}(G) < H$; otherwise, if $q > 2$ then $r(G) > 0$ only when $\text{deg}(G) = H$. 
We thus have
$$
\sum_{\substack{P_1,P_2,P_3,P_4 \in \mc{P}_H \\ P_1 + P_2 = P_3 + P_4}} 1 = \sum_{\deg{G} \leq H} r(G)^2.
$$
Let $z := H/2$, $y = 1$ and $\mf{P}_z := \mf{P}_{y,z} = \prod_{P \in \mc{P}_{\leq z} \bk \mc{P}_{\leq y}} P$ as in the previous lemma. We then have
$$
r(G) \leq \sum_{\substack{M \in \mc{M}_H \\ (M(G-M),\mf{P}_z) = 1 \\ G-M \in \mc{M}_H}} 1 = \sum_{\substack{F \in \mc{A} \\ (F,\mf{P}_z) = 1}} 1,
$$
where $\mc{A} = \{B(G-B)  \colon  B \in \mc{M}_H\} \cap \mc{M}_{2H}$; as $\mc{A}$ and $\mc{M}_H$ are in bijection with one another, we have $|\mc{A}| = |\mc{M}_H| \asymp q^H$. \\
Note that for $D|\mf{P}_z$ with $\deg{D} \leq H$ 
$$\sum_{\substack{F \in \mc{A} \\ D|F}} 1 = g(D)|\mc{M}_{H}|,$$
where $g$ is the multiplicative function supported on squarefree polynomials and defined at irreducibles via $g(P) = 2q^{-\deg{P}}$ if $P\nmid G$ and $g(P) = q^{-\deg{P}}$ otherwise; note that $g(P) \leq 1/2$ for all $P | \mf{P}_z$ and all $q \geq 2$, since such $P$ must have $\text{deg}(P) \geq 2$. By Lemma~\ref{Selberg}, we deduce
$$
\sum_{\substack{F \in \mc{A} \\ (F,\mf{P}_z) = 1}} 1 \ll q^{H}\Big(\sum_{\substack{D|\mf{P}_z \\ \deg{D} \leq H}} \prod_{P|D} \frac{g(P)}{1-g(P)} \Big)^{-1} \ll q^H \Big(\sum_{D|\mf{P}_z} g(D) - \sum_{\substack{D \mid \mf{P}_z \\ \deg{D} > H}} g(D)\Big)^{-1}.
$$
Note that the full bracketed sum over $D|\mf{P}_z$ has order of magnitude
$$= \prod_{R\in \mc{P}_{\leq z}} \Big(1+g(R)\Big) \asymp \exp\Big(2\sum_{R \in \mc{P}_{ \leq z}} q^{-\deg{R}} - \sum_{\substack{R \in \mc{P} \\ R|G}} q^{-\deg{R}}\Big) \asymp \frac{\phi(G)}{q^{\deg{G}}} z^2.$$
The remaining sum over $D|\mf{P}_z$ with $\deg{D} > H$ can be bounded above as
\begin{align*}
\leq \sum_{k > H} q^{-k} \sum_{\substack{D|P_z \\ \deg{D} = k}} 2^{\omega(D)} &= \sum_{k > H} q^{-k} \sum_{\substack{a_1 + 2a_2 + \cdots + za_z = k \\ 0 \leq a_j \leq |\mc{P}_j|}} \prod_{1 \leq j \leq z} 2^{a_j} \\
&\leq \sum_{k > H} q^{-k} \sum_{\substack{a_1 + 2a_2 + \cdots + za_z = k \\ 0 \leq a_j \leq |\mc{P}_j|}} 2^{q + \frac{1}{2}(2a_2 + \cdots + za_z)} \\
&\ll \sum_{k > H} (\sqrt{2}/q)^k |\{\mbf{a} \in \mb{N} \cup \{0\}  \colon  a_1 + 2a_2 + \cdots + za_z = k\}|.
\end{align*}
Using standard results on partitions (see e.g.,~\cite{erd_part}), the cardinality above is $\ll e^{c\sqrt{k}}$, for some $c > 0$ absolute. Thus, as $q \geq 2$ the series over $k$ converges, and in fact
$$
\sum_{\substack{D \mid \mf{P}_z \\ \deg{D} > H}} g(D) \ll e^{-c'H},
$$
for a suitable absolute $c' > 0$. 

It follows that for large $H$,
$$r(G) \ll \frac{q^{\deg{G}}}{\phi(G)} \frac{q^H}{z^2} \ll \frac{q^{\deg{G}}}{\phi(G)} \frac{q^H}{H^2}.$$ 
Squaring this bound and summing over $G \in \mb{F}_q[t]$ of degree $\leq H$ for which $r(G) \neq 0$, we get that
$$
|\{(P_1,P_2,P_3,P_4) \in \mc{P}_H  \colon  P_1 + P_2 = P_3 + P_4\}| \leq \sum_{\deg{G} \leq H} r(G)^2 \ll \frac{q^{2H}}{H^4} \sum_{\deg{G}\leq H} \Big(\frac{q^{\deg{G}}}{\phi(G)}\Big)^2.
$$
We claim that the sum over $G$ is $\ll q^H$, which will then imply the claim for $M = 0$. To see this, write $\psi(G) := (q^{\deg{G}}/\phi(G))^2$; note that $\psi$ is independent of the leading coefficient of $G$, and so we may replace $G$ by $G/G(0)$ and assume $G$ is monic (this changes the sum by at most a factor depending only on $q$). Note that for any $k \geq 1$, $\psi(R^k) = \psi(R) = (1-q^{-\deg{R}})^{-2}\leq 4$ uniformly over $R \in \mc{P}$. Hence, we may apply Lemma~\ref{HalRicFF} to get
$$
\sum_{G \in \mc{M}_{\leq H}} \psi(G) \ll \sum_{0 \leq h \leq H} \frac{q^h}{h}\exp\Big(\sum_{R \in \mc{P}_{\leq h}} \psi(R)q^{-\deg{R}}\Big).
$$
We may directly evaluate the sum over $R$ here for each $h \leq H$ by the prime polynomial theorem, getting
$$
\sum_{R \in \mc{P}_{\leq h}} \psi(R)q^{-\deg{R}} = \sum_{k \leq h}\frac{q^{-k}}{(1-q^{-k})^{2}} |\mc{P}_k| = \sum_{k \leq h} \Big(\frac{1}{k} (1-q^{-k})^{-2} + O(q^{-k/2})\Big) = \log h +O(1),
$$
which leads to
$$
\sum_{G \in \mc{M}_{\leq H}} \psi(G) \ll \sum_{h \leq H} q^h \ll q^H,
$$
as required. 
Next, let $M \in \mc{M}_{< H}$. Then
\begin{align*}
|\{(P_1,P_2,P_3,P_4) \in \mc{P}_H^4  \colon  P_1+P_2-P_3-P_4 = M\}| = 
\sum_{\deg{G} \leq H} r(G)r(M+G) \leq \sum_{\deg{G} \leq H} r(G)^2,
\end{align*}
by the AM-GM inequality and the fact that $\deg{G+M} \leq H$ whenever $\max\{\deg{G},\deg{M}\} \leq H$. The second claim now follows from the first.
\end{proof}

\subsection{Dirichlet Polynomial Decompositions}
Let $Q> P \geq 1$. Recall that $\mc{S}_{P,Q}$ denotes the set of monic $G$ that have an irreducible factor $R$ satisfying $\deg{R} \in [P,Q]$.
\begin{lem}[Ramar\'e's identity]\label{Ramare}
Let $P < Q$. Let $f\colon \mc{M} \to \mb{U}$ be multiplicative. Then for any $G \in \mc{S}_{P,Q}$,
$$f(G) = \sum_{\substack{RM = G \\ R \in \mc{P} \\ \deg{R} \in [P,Q]}} \frac{f(RM)}{1_{(R,M) = 1} + \omega_{[P,Q]}(M)},$$
where $\omega_{[P,Q]}(M):= |\{R \in \mc{P}  \colon  \deg{R} \in [P,Q], R|M\}|$.
\end{lem}
\begin{proof}
Since $\omega_{[P,Q]}(G) \geq 1$ by assumption we have
$$1 = \sum_{\substack{R|G \\ R \in \mc{P} \\ \deg{R} \in [P,Q]}} \frac{1}{\omega_{[P,Q]}(G)}  = \sum_{\substack{RM = G \\ R \in \mc{P} \\ \deg{R} \in [P,Q]}} \frac{1}{\omega_{[P,Q]}(RM)} = \sum_{\substack{RM = G \\ R \in \mc{P} \\ \deg{R} \in [P,Q]}} \frac{1}{1_{(R,M) = 1} + \omega_{[P,Q]}(M)}.$$
This implies the claim.
\end{proof}
We will use Ramar\'{e}'s identity to decompose Dirichlet polynomials supported on $\mc{S}_{P,Q}$, as in the following lemma.
\begin{lem}\label{DirDecomp}
Let $N \geq 1$. Let $L \in \mc{M}_{\leq N}$ and suppose $\Xi \subseteq \mc{X}_L$. Lastly, let $f\colon \mc{M} \to \mb{U}$ be multiplicative. Then for any $1 \leq P< Q \leq N$,
\begin{align*}
\sum_{\chi \in \Xi} \Big|q^{-N} \sum_{G \in \mc{M}_N} f(G)\bar{\chi}(G)1_{\mc{S}_{P,Q}}(G)\Big|^2 &\ll (Q-P+1)\sum_{P \leq d \leq Q} \sum_{\chi \in \Xi}|A_d(\chi)|^2|B_{N-d}(\chi)|^2 \\
&+ \phi(L)(q^{-N} + q^{-\deg{L}}) q^{-P},
\end{align*}
where for $d \geq 1$ and a character $\chi$ modulo $L$, we set
\begin{align*}
A_{d}(\chi) &:= q^{-d}\sum_{R \in \mc{P}_d} f(R)\bar{\chi}(R).\\
B_{N-d}(\chi) &:= q^{-N+d}\sum_{D \in \mc{M}_{N-d}} \frac{f(D)\bar{\chi}(D)}{1+\omega_{[P,Q]}(D)}.
\end{align*}
\end{lem}
\begin{proof}
This is analogous to~\cite[Lemma 12]{mr-annals}. By Lemma~\ref{Ramare}, for any $\chi \in \Xi$ we have
\begin{align}\label{AdBd}
&q^{-N}\sum_{G \in \mc{M}_N}f(G) \bar{\chi}(G)1_{\mc{S}_{P,Q}}(G)\nonumber\\ 
&= q^{-N}\sum_{RM \in \mc{M}_N} \frac{f\bar{\chi}(R)f\bar{\chi}(M)}{1+\omega_{[P,Q]}(M)} + q^{-N}\sum_{RM \in \mc{M}_N} \frac{(f(RM) - f(R)f(M))\bar{\chi}(RM)}{1+ \omega_{[P,Q]}(M)}\nonumber \\
&+ q^{-N}\sum_{RM \in \mc{M}_N} f(RM)\bar{\chi}(RM) \Big(\frac{1}{1_{(R,M) = 1} + \omega_{[P,Q]}(M)}
- \frac{1}{1+\omega_{[P,Q]}(M)}\Big)\nonumber \\
&= \sum_{P \leq d \leq Q} \Big(q^{-d}\sum_{R \in \mc{P}_d} f(R)\bar{\chi}(R)\Big)\Big(q^{-N+d}\sum_{M \in \mc{M}_{N-d}} \frac{f(M)\bar{\chi}(M)}{1+\omega_{[P,Q]}(M)}\Big) + \mc{R}_{1,\chi} + \mc{R}_{2,\chi}\nonumber \\
&= \sum_{P \leq d \leq Q} A_{d}(\chi)B_{N-d}(\chi) + \mc{R}_{1,\chi} + \mc{R}_{2,\chi}.
\end{align}
Note that for each $\chi \in \Xi$,  both of $\mc{R}_{1,\chi}$ and $\mc{R}_{2,\chi}$ are supported on polynomials $M$ such that $R|M$ for some $R \in \mc{P}$, $\deg{R} \in [P,Q]$. We now take squares and sum the whole expression over all $\chi \in \Xi$ to see that the mean square of~\eqref{AdBd} is 
$$
\ll \sum_{\chi \in \Xi} \Big|\sum_{P \leq d \leq Q} A_d(\chi)B_{N-d}(\chi)\Big|^2 + \sum_{\chi \in \Xi} |\mc{R}_{1,\chi}|^2 + \sum_{\chi \in \Xi} |\mc{R}_{2,\chi}|^2.
$$
To treat the first term, we use the Cauchy--Schwarz inequality in the inner sum to get
$$
\sum_{\chi \in \Xi} \Big|\sum_{P \leq d \leq Q} A_{d}(\chi)B_{N-d}(\chi)\Big|^2 \leq (Q-P+1)\sum_{P \leq d \leq Q} \sum_{\chi \in \Xi}|A_{d}(\chi)|^2|B_{N-d}(\chi)|^2.
$$
To treat $\sum_{\chi \in \Xi}|\mc{R}_{j,\chi}|^2$ for $j = 1,2$ we use Lemma~\ref{L2MVT}; since the arguments are similar we shall restrict ourselves to proving the bound for $\mc{R}_{1,\chi}$. By Lemma~\ref{L2MVT},
$$
\sum_{\chi \in \Xi}|R_{1,\chi}|^2 \ll \phi(L)\Big(q^{-N} + q^{-\deg{L}}\Big) q^{-N}\sum_{G \in \mc{M}_N} \Big|\sum_{\substack{\substack{RM = G \\ R \in \mc{P}, \,  R|M \\ \deg{R} \in [P,Q]}}} \frac{(f(RM)-f(R)f(M))}{1+\omega_{[P,Q]}(M)}\Big|^2.
$$
Expanding the square and bounding the summands trivially, we bound the sum on the right-hand side as
\begin{align*}
&\sum_{\substack{R_1,R_2 \in \mc{P} \\ \deg{R_j} \in [P,Q]}} \sum_{\substack{G \in \mc{M}_N \\ [R_1,R_2]^2 | G}} 1 \ll q^N \Big(\sum_{P \leq d \leq Q} |\mc{P}_d|q^{-2d} + \sum_{P \leq d_1,d_2 \leq Q} |\mc{P}_{d_1}||\mc{P}_{d_2}| q^{-2(d_1+d_2)}\Big) \\
&\ll q^N \sum_{P \leq d \leq Q} q^{-d} \ll q^{N-P},
\end{align*}
which implies the claim.
\end{proof}
\begin{lem}[Pointwise bound with Ramar\'e weight]\label{QuasiMultFF}
Let $1 \leq P < Q < N^{0.9}$. Let  $f\colon \mc{M} \to \mb{U}$ be multiplicative. There is a Hayes character $\widetilde{\chi_1}$ of conductor $\leq N$ such that
$$\max_{\substack{\textnormal{cond}_H(\widetilde{\chi})\leq N \\ \widetilde{\chi} \not\sim \widetilde{\chi_1}}} \frac{1}{|\mc{M}_N|}\Big|\sum_{G \in \mc{M}_N} \frac{f(G)\bar{\widetilde{\chi}}(G)}{1+\omega_{[P,Q]}(G)}\Big| \ll (Q/P)^3N^{-1/4+o(1)}.$$
Moreover, we can take $\widetilde{\chi_1}$ to be the Hayes character of conductor $\leq N$ that minimizes $\chi\mapsto \mathcal{D}_{f\bar{\widetilde{\chi}}}(N)$.
\end{lem}

\begin{proof}
Let $\widetilde{\chi_{1}}$ be the character that minimizes $\mc{D}_{f\bar{\widetilde{\chi}}}(N)$ among all $\widetilde{\chi}$ of conductor $\leq N$, and let $\widetilde{\chi} \nsim \widetilde{\chi_{1}}$. Write $\mc{I} := [P,Q]$ and $\mc{I}^c := \mb{N} \bk \mc{I}$. We can express $f = f_{\mc{I}} \ast f_{\mc{I}^c}$, where for $\mc{J} \in \{\mc{I},\mc{I}^c\}$ we define the multiplicative function $f_{\mc{J}}$ at powers of irreducibles via
$$f_{\mc{J}}(P^k) := \begin{cases} f(P^k) &\text{ if $\deg{P} \in \mc{J}$} \\ 0 &\text{ otherwise.}\end{cases}$$
Let $N' := \llf N/2\rrf$. By the hyperbola method,
\begin{align*}
&\sum_{G \in \mc{M}_N} \frac{f(G)\bar{\widetilde{\chi}}(G)}{1+\omega_{\mc{I}}(G)} = \sum_{AB \in \mc{M}_N} \frac{f_{\mc{I}}(A)f_{\mc{I}^c}(B)\bar{\widetilde{\chi}}(AB)}{1+\omega_{\mc{I}}(A)} \\
&= \sum_{A \in \mc{M}_{\leq N'}} \frac{f_{\mc{I}}(A)\bar{\widetilde{\chi}}(A)}{1+\omega_{\mc{I}}(A)}\sum_{B \in \mc{M}_{N-\deg{A}}}f_{\mc{I}^c}(B)\bar{\widetilde{\chi}}(B) + \sum_{B \in \mc{M}_{\leq N-N'}} f_{\mc{I}^c}(B)\bar{\widetilde{\chi}}(B) \sum_{A \in \mc{M}_{N-\deg{B}}} \frac{f_{\mc{I}}(A)\bar{\widetilde{\chi}}(A)}{1+\omega_{\mc{I}}(A)} \\
&- \Big(\sum_{A \in \mc{M}_{N'}}\frac{f_{\mc{I}}(A)\bar{\widetilde{\chi}}(A)}{1+\omega_{\mc{I}}(A)} \Big)\Big(\sum_{B \in \mc{M}_{N-N'}}f_{\mc{I}^c}(B)\bar{\widetilde{\chi}}(B)\Big) =: T_1+T_2-T_3.
\end{align*}
We first treat $T_1$. Let $0 \leq K \leq N'$. Since $\widetilde{\chi} \not\sim \widetilde{\chi_{1}}$, Lemma~\ref{RepulsionFF} implies that
\begin{align*}
\mc{D}_{f_{\mc{I}^c}\bar{\widetilde{\chi}}}(N-K) &= \mc{D}_{f_{\mc{I}^c}\bar{\widetilde{\chi}}}(N) - O(1) \geq \mc{D}_{f\bar{\widetilde{\chi}}}(N) - 2\sum_{P \leq d \leq Q} q^{-d}|\mc{P}_d|\\
&\geq \Big(\frac{1}{4}-o(1)\Big)\log N - 2\log(Q/P).
\end{align*}
Combining this with Theorem~\ref{HalThmFF}, we obtain
\begin{equation} \label{eq:DirConvNoWt}
\sum_{G \in \mc{M}_{N-K}} f_{\mc{I}^c}(G)\bar{\widetilde{\chi}}(G) \ll q^{N-K}\mc{D}_{f_{\mc{I}^c}\bar{\widetilde{\chi}}}(N-K)\exp\Big(-\mc{D}_{f_{\mc{I}^c}\bar{\widetilde{\chi}}}(N-K)\Big) \ll q^{N-K} (Q/P)^2N^{-1/4+o(1)}.
\end{equation}
Applying this with $K = \deg{A}$ in $T_1$ and summing over $A \in \mc{M}_{\leq N'}$ yields
\begin{align*}
T_1 &\ll q^N(Q/P)^2N^{-1/4+o(1)}\sum_{\substack{A \in \mc{M}_{\leq N'} \\ R |A \Rightarrow \deg{R} \in [P,Q]}} q^{-\deg{A}}\\
&\ll q^N(Q/P)^2N^{-1/4+o(1)} \exp\Big(\sum_{\substack{R \in \mc{P} \\ P \leq \deg{R} \leq Q}} q^{-\deg{R}}\Big) \\
&\ll q^N(Q/P)^3N^{-1/4+o(1)}.
\end{align*}
We next consider $T_2$. Using Lemma~\ref{SmoothsFF}, for every $0 \leq K \leq N-N' \leq N/2+1$ we have
\begin{align} \label{eq:DirConvSmooth}
\Big|\sum_{A \in \mc{M}_{N-K}} \frac{f_{\mc{I}}(A)}{1+\omega_{\mc{I}}(A)} \Big| &\leq |\{A \in \mc{M}_{N-K}  \colon  R|A, R \in \mc{P} \Rightarrow \deg{R}\leq Q\}| \ll q^{N-K}\exp(-c(N-K)/Q)\\
&\ll q^{N-K}\exp\Big(-c\frac{N}{3Q}\Big)
\end{align}
for some $c>0$.
Applying this with $K = \deg{B}$, then summing over $B$ in $T_2$ yields
\begin{align*}
T_2 &\ll \sum_{B \in \mc{M}_{\leq N-N'}} \Big|\sum_{A \in \mc{M}_{N-\deg{B}}} \frac{f_{\mc{I}}(A)}{1+\omega_{\mc{I}}(A)}\Big| \ll q^N\exp\Big(-c\frac{N}{3Q}\Big)\sum_{B \in \mc{M}_{\leq N-N'}} q^{-\deg{B}} \\
&\ll q^N N\exp\Big(-c\frac{N}{3Q}\Big).
\end{align*}
Finally, consider $T_3$. Using the estimates~\eqref{eq:DirConvNoWt} and~\eqref{eq:DirConvSmooth} (with $N-K$ replaced by $N'$ and $N-N'$, respectively) yields 
\begin{align*}
T_3 &\ll q^NN\exp\Big(-c\frac{N}{6Q}\Big)(Q/P)^2N^{-1/4+o(1)}\\
&\ll q^N N^3 \exp\Big(-c\frac{N}{6Q}\Big).
\end{align*}
Combining the estimates for $T_1,T_2$ and $T_3$ and noting that $N^3\exp(-c\frac{N}{6Q})\ll N^{-100}$ establishes the claim.
\end{proof}

\section{Variance of Multiplicative Functions in Progressions to Large Degree Moduli} \label{VarSec}
In this section, we will prove Theorem~\ref{VarAPs}. In the next section, we will apply a very similar argument to deduce the Matom\"{a}ki--Radziwi\l\l \ type theorem that we shall need. 

Let $1 \leq H \leq N-N^{3/4}$ with $H= H(N) \ra \infty$ as $N\to \infty$. Let $f\colon\mc{M} \ra \mb{U}$ be multiplicative, and let $Q \in \mc{M}_{N-H}$. Let $\chi_1 \in \mc{X}_Q$ be the Dirichlet character $\bmod{Q}$ that minimizes $\chi \mapsto \mc{D}_{f\bar{\chi}}(N)$. By orthogonality, 
\begin{equation}\label{CharOrtho}
\asum_{A \bmod{Q}} \Big|\sum_{\substack{G \in \mc{M}_N \\ G \equiv A \bmod{Q}}} f(G)-\frac{\chi_1(A)}{\phi(Q)} \sum_{G \in \mc{M}_N} f(G)\bar{\chi_1}(G)\Big|^2 =\frac{1}{\phi(Q)} \sum_{\chi \neq \chi_1} \Big|\sum_{G \in \mc{M}_N} f(G)\bar{\chi}(G)\Big|^2.
\end{equation}
Let $\eta \in (0,1/6)$ be fixed, and set 
\begin{align*}
Q_1 := \min\{H, N^{1/5}\}, \quad \textnormal{and}\quad P_1 := \frac{400}{\eta \log q} \log Q_1.
\end{align*} 
Fix $J \geq 1$ to be the least integer such that $J^{4J+2}Q_1^J \geq N^{1/2}$, and if $J \geq 2$ set \begin{align*}
P_j := j^{4j}Q_1^{j-1}P_1\quad \textnormal{and}\quad Q_j := j^{4j+2}Q_1^j   
\end{align*} 
for each $2 \leq j \leq J$. We define $\mc{S}_{\mbf{P},\mbf{Q}}$ as in Definition~\ref{def_SPQ} with these collections of parameters $P_j$ and $Q_j$, and for $1 \leq j \leq J$, we let $\mc{S}^{(j)}_{\mbf{P},\mbf{Q}}$ denote the set of $G \in \mc{M}$ with an irreducible factor $R$ with $\deg{R} \in [P_i,Q_i]$ for all $i \neq j$.

For each $j$, $1 \leq d \leq N$ and a character $\chi$ modulo $Q$, set
\begin{align*}
A_{j,d}(\chi) &:= \frac{1_{[P_j,Q_j]}(d)}{d|\mc{P}_d|}\sum_{R \in \mc{P}_d} f(R)\bar{\chi}(R).\\
B_{j,d}(\chi) &:= \frac{1}{|\mc{M}_d|}\sum_{\substack{D \in \mc{M}_d \\ D \in \mc{S}^{(j)}_{\mbf{P},\mbf{Q}}}} \frac{f(D)\bar{\chi}(D)}{1+\omega_{[P_j,Q_j]}(D)}.
\end{align*}
Thus, $A_{j,d}(\chi) = 0$ except when $d \in [P_j,Q_j]$. Following~\cite{mr-annals} (see also~\cite[Section 5.2]{klu_man_ter}), we split the set $\Xi := \mc{X}_Q \bk \{\chi_1\}$ into the following sets.
\begin{def1}
For $j \geq 1$ put $\beta_j := \frac{1}{4}-\frac{\eta}{2}(1+1/j)$. Define
\begin{align*}
\mc{X}_1 &:= \{\chi \in \Xi  \colon  |A_{1,d}(\chi)| \leq q^{-\beta_1d} \ \forall P_1 \leq d \leq Q_1\}  \\
\mc{X}_j &:= \{\chi \in \Xi  \colon  |A_{j,d}(\chi)| \leq q^{-\beta_jd} \ \forall P_j \leq d \leq Q_j\} \bk \bigcup_{1 \leq i \leq j-1} \mc{X}_i \ \text{ ($2 \leq j \leq J$)} \\
\mc{U} &:= \Xi \bk \bigcup_{1 \leq j \leq J} \mc{X}_j.
\end{align*}
\end{def1}
We shall bound the contribution of the characters from each of $\mc{X}_j$ and $\mc{U}$ using the lemmata from the previous sections. 

We begin by estimating the contribution from $\chi \notin \mc{U}$. By Lemma~\ref{RestricToS}, we have
\begin{align}
\frac{1}{\phi(Q)} \sum_{\ss{\chi \neq \chi_1 \\ \chi \notin \mc{U}}} \Big|\sum_{G \in \mc{M}_N} f(G)\bar{\chi}(G)\Big|^2 &\ll \frac{1}{\phi(Q)}\sum_{\ss{\chi \neq \chi_1 \\ \chi \notin \mc{U}}} \Big|\sum_{\substack{G \in \mc{M}_N \\ G \in \mc{S}_{\mbf{P},\mbf{Q}}}} f(G)\bar{\chi}(G)\Big|^2 + q^{2N-\deg{Q}}\sum_{1 \leq j \leq J} \frac{P_j}{Q_j} \nonumber\\
&\ll \frac{1}{\phi(Q)}\sum_{\ss{\chi \neq \chi_1 \\ \chi \notin \mc{U}}} \Big|\sum_{\substack{G \in \mc{M}_N \\ G \in \mc{S}_{\mbf{P},\mbf{Q}}}} f(G)\bar{\chi}(G)\Big|^2 + q^{2N-\deg{Q}}\frac{P_1}{Q_1}. \label{CharOrtho2}
\end{align}
For each $\chi \in \Xi \bk \mc{U}$, write
$$F(\chi) := \frac{1}{|\mc{M}_N|}\sum_{\substack{G \in \mc{M}_N \\ G \in \mc{S}_{\mbf{P},\mbf{Q}}}}f(G)\bar{\chi}(G).$$
We apply Lemma~\ref{DirDecomp} for each $1 \leq j \leq J$ (with $P = P_j$ and $Q = Q_j$ in the notation there) to get
\begin{align*}
\sum_{\chi \in \mc{X}_j} |F(\chi)|^2 &\ll (Q_j-P_j+1) \sum_{P_j \leq d \leq Q_j}\sum_{\chi \in \mc{X}_j} |A_{j,d}(\chi)|^2|B_{j,N-d}(\chi)|^2 + \phi(Q)q^{-\deg{Q}}q^{-P_j} \\
&=: \mf{M}_j + \mf{R}_j.
\end{align*}
Summing the error terms arising from the terms $1 \leq j \leq J$ yields 
\begin{align}\label{Rj}
\sum_{1 \leq j \leq J} \mf{R}_j \ll \phi(Q)q^{-\deg{Q}} \cdot q^{-P_1}\ll Q_1^{-100}\phi(Q)/q^{\deg{Q}},
\end{align}
using the definitions of $P_j$ and $Q_j$ above. We thus focus on the main terms arising in the above estimate.

\subsection*{Case 1: \texorpdfstring{$j = 1$}{j=1}}
In this case we bound $|A_{1,d}(\chi)| \leq q^{-\beta_1d}$ for each $\chi \in \mc{X}_1$ and then apply Lemma~\ref{L2MVT} to get
\begin{align*}
\mf{M}_1 &\ll (Q_1-P_1+1) \sum_{P_1 \leq d \leq Q_1} q^{-2\beta_1d} \sum_{\chi \bmod{Q}} |B_{1,N-d}(\chi)|^2 \\
&\leq Q_1\Big(\phi(Q)q^{Q_1-N}+\phi(Q)q^{-\deg{Q}}\Big)\sum_{P_1 \leq d \leq Q_1}q^{-2\beta_1d} \\
&\ll \phi(Q)q^{-\deg{Q}} \cdot Q_1 q^{-2\beta_1P_1} \ll \phi(Q)q^{-\deg{Q}} \cdot Q_1q^{-P_1/6},
\end{align*}
since $\beta_1 = 1/4-\eta \geq 1/12$. Thus,
\begin{align}\label{M1}
\mf{M}_1\ll Q_1^{-100} \phi(Q)/q^{\deg{Q}}.   
\end{align}
 
\subsection*{Case 2: \texorpdfstring{$2\leq j\leq J$}{2 <=j<=J}}
We know that for each $\chi \in \mc{X}_j$ we can find a $d_{\chi} \in [P_{j-1},Q_{j-1}]$ for which $|A_{j-1,d_{\chi}}(\chi)| \geq q^{-\beta_{j-1}d_{\chi}}$. Thus, similarly as in~\cite{mr-annals}, we can estimate
\begin{align*}
\mf{M}_j &\ll (Q_j-P_j+1) \sum_{P_{j-1} \leq r \leq Q_{j-1}} \sum_{\substack{\chi \in \mc{X}_j\\ d_{\chi} = r}}\sum_{P_j \leq d \leq Q_j} |A_{j,d}(\chi)|^2|B_{j,N-d}(\chi)|^2 \\
&\leq (Q_j-P_j+1)(Q_{j-1}-P_{j-1}+1) \max_{P_{j-1} \leq r \leq Q_{j-1}} \sum_{\substack{\chi \in \mc{X}_j \\ d_{\chi} = r}}\sum_{P_j \leq d \leq Q_j} q^{-2\beta_jd} |B_{j,N-d}(\chi)|^2 \\
&\leq Q_j^2 \sum_{P_j \leq d \leq Q_j} q^{-2\beta_j d} q^{2\ell_d r_0\beta_{j-1}} \sum_{\chi \bmod{Q}}|A_{j-1,r_0}(\chi)^{\ell_d}B_{j,N-d}(\chi)|^2,
\end{align*}
for some $r_0 \in [P_{j-1},Q_{j-1}]$, with $\ell_d:= \lceil d/r_0\rceil$. Applying Lemma~\ref{Lengthen}, we have
$$
\sum_{\chi \bmod{Q}}|A_{j-1,r_0}(\chi)^{\ell_d}B_{j,N-d}(\chi)|^2 \ll \phi(Q)q^{-\deg{Q}}\ell_d^{2\ell_d}.
$$ 
Combining this with the estimates from the previous line, we get
$$
\mf{M}_j \ll \phi(Q)q^{-\deg{Q}} \cdot Q_j^2 \sum_{P_j \leq d \leq Q_j} q^{2(\ell_dr_0\beta_{j-1}-d\beta_j)} \ell_d^{2\ell_d}.
$$
By definition, $\ell_d \leq d/r_0+1$, so that since $r_0 \leq Q_{j-1}$,
$$
\ell_d r_0 \beta_{j-1} - d\beta_j \leq d(\beta_{j-1}-\beta_j) + r_0\beta_{j-1} \leq -\frac{\eta d}{2j^2} + Q_{j-1}\beta_{j-1}.
$$
Furthermore, we have
$$
\ell_d \log \ell_d \leq \frac{d\log d}{r_0} + \log d \leq (\log Q_j) (d/P_{j-1} + 1).
$$
We thus may bound $\mf{M}_j$ as
$$
\mf{M}_j \ll \phi(Q)q^{-\deg{Q}} \cdot Q_j^4 q^{2Q_{j-1}\beta_{j-1}} \sum_{P_j \leq d \leq Q_j} q^{-2d(\eta/(2j^2) - (\log Q_j)/(P_{j-1}\log q))}.
$$
We record the following easy-to-check bounds, contingent on $Q_1$ being sufficiently large and $j \geq 2$:
\begin{enumerate}[label=(\roman*)]
 
\item $\frac{\log Q_j}{P_{j-1}} \leq \frac{j (\log Q_1)(1 + 5\log j)}{(j-1)^{4j-4}Q_1^{j-2}P_1} \leq \frac{50\log Q_1}{P_1} \cdot \frac{1}{j^2} \leq \frac{\eta\log q}{8j^2}$ 
\item $Q_j^4q^{2\beta_{j-1}Q_{j-1}} \leq q^{Q_{j-1}/2}$.
\item $Q_{j-1} \leq j^{4j-2}Q_1^{j-1} \leq P_j/(j^2P_1).$
\end{enumerate}
Using these bounds, we get
\begin{align*}
\mf{M}_j &\ll \phi(Q)q^{-\deg{Q}} \cdot q^{Q_{j-1}/2} \sum_{P_j \leq d \leq Q_j} q^{-\eta d/(2j^2)} \ll \phi(Q)q^{-\deg{Q}} \cdot j^2\eta^{-1} q^{P_j/(2j^2 P_1)} q^{-\eta P_j/(2j^2)} \\
&\ll_{\eta} \phi(Q)q^{-\deg{Q}} \cdot j^2 q^{-\eta P_j/(4j^2)} \leq \phi(Q)q^{-\deg{Q}} \cdot j^{-2}q^{-\eta Q_1P_1/4}.
\end{align*}
Summing over $2 \leq j \leq J$, we get 
\begin{align}\label{M2}
\sum_{2 \leq j \leq J} \mf{M}_j \ll_{\eta} \phi(Q)q^{-\deg{Q}} \cdot q^{-\eta Q_1P_1/4}\ll Q_1^{-100}\phi(Q)q^{-\deg{Q}}.
\end{align}
\subsection*{Case 3:  \texorpdfstring{$\mc{U}$}{U}}
We now treat the remaining characters $\chi \in \mc{U}$. We make an additional choice of parameters $\tilde{P} := N^{2/3}$, $\tilde{Q} := N^{13/18}$. Combining Lemma~\ref{RestricToS} with Lemma~\ref{DirDecomp}, we find $\tilde{P} \leq d_0 \leq \tilde{Q}$ such that
\begin{align*}
&\sum_{\chi \in \mc{U}}|\sum_{G \in \mc{M}_N} f(G) \bar{\chi}(G)|^2 \\
&\ll \sum_{\chi \in \mc{U}} |\sum_{\ss{G \in \mc{M}_N \\ G \in \mc{S}_{\tilde{P},\tilde{Q}}}} f(G) \bar{\chi}(G)|^2 + \phi(Q) q^{-\deg{Q}} \frac{\tilde{P}}{\tilde{Q}} \\ 
&\ll (\tilde{Q} - \tilde{P}+1)\sum_{\tilde{P} \leq d \leq \tilde{Q}} \sum_{\chi \in \mc{U}}|A_d(\chi)|^2|B_{N-d}(\chi)|^2 + \phi(Q)q^{-\deg{Q}} \Big(q^{-\tilde{P}} + \frac{\tilde{P}}{\tilde{Q}}\Big)\\
&\ll \tilde{Q}^2\sum_{\chi \in \mc{U}}|A_{d_0}(\chi)|^2|B_{N-d_0}(\chi)|^2 + \phi(Q)q^{-\deg{Q}} \Big(q^{-\tilde{P}} + \frac{\tilde{P}}{\tilde{Q}}\Big),
\end{align*}
where $A_{d}$ and $B_{N-d}$, $\tilde{P} \leq d \leq \tilde{Q}$, are defined as in Lemma~\ref{DirDecomp} with respect to the parameters $\tilde{P}$ and $\tilde{Q}$. We now split the set $\mc{U}$ further.
Following~\cite{mr-annals}, we define the subsets
\begin{align*}
\mc{U}_S &:= \{\chi \in \mc{U}  \colon  |A_{d_0}(\chi)| \leq N^{-10}\} \\
\mc{U}_L &:= \{\chi \in \mc{U}  \colon  |A_{d_0}(\chi)| > N^{-10}\}.
\end{align*}
We begin by treating the contribution from $\mc{U}_S$. By Lemma~\ref{HalMonInt} and the fact that $d_0 \leq \tilde{Q} =o(N)$, we have
\begin{align}\label{chisum}
\sum_{\chi \in \mc{U}_S} |A_{d_0}(\chi)|^2|B_{N-d_0}(\chi)|^2 &\leq N^{-20}\sum_{\chi \in \mc{U}_S} |B_{N-d_0}(\chi)|^2\nonumber\\
&\ll N^{-20}\phi(Q)q^{-\deg{Q}}\Big(1 + |\mc{U}_S| q^{(1/2+o(1))N-N+d_0}\Big)\nonumber \\
&\ll N^{-20}\phi(Q)q^{-\deg{Q}}\Big(1 + |\mc{U}|q^{(-1/2+o(1))N}\Big).
\end{align}
To estimate the size of $\mc{U}$, we note that whenever $\chi \in \mc{U}$ there is $d_J \in [P_J,Q_J]$ such that $|A_{J,d_J}(\chi)| \geq q^{-\beta_Jd_J}$, and thus 
$$
|\mc{U}| \leq \left|\bigcup_{P_J \leq d_J \leq Q_J} \{\chi \bmod{Q}  \colon  |A_{J,d_J}(\chi)| \geq q^{-\beta_Jd_J}\}\right| \leq Q_J \max_{P_J \leq d_J \leq Q_J} |\{\chi \bmod{Q}  \colon  |A_{J,d_J}(\chi)| \geq q^{-\beta_Jd_J}\}|.
$$
By choice, we have that $N^{1/2} \leq Q_J \ll N^{1/2} J^4 Q_1 \ll N^{7/10+o(1)}$, so that from (i) above we have $P_J \geq \tfrac{8J^2}{\eta \log q} \log Q_{J+1} \geq \tfrac{4}{\eta \log q}\log N$ and $\phi(Q) \geq q^{Q_J} \geq q^{d_J}$ for all $d_J \in [P_J,Q_J]$. Hence, $(\log \log \phi(Q))/P_J < \frac{1}{2}\eta\log q$ for $N$ large enough, and Lemma~\ref{LargeValues} may be applied to give 
$$
|\mc{U}| \ll Q_J\exp\Big(\frac{\log(q^{d_J} \phi(Q))}{d_J\log q}\Big(2\log\Big(\frac{\log (2\phi(Q))}{d_J\log q}\Big) + \log(2q^{2\beta_J d_J}/d_J)\Big)\Big) \leq \phi(Q)^{\frac{1}{2}-\eta/2}.
$$
Inserting this into~\eqref{chisum}, the off-diagonal term becomes $O(q^{(-\eta/2+o(1))N}) = o(1)$. Thus, we find that
$$\sum_{\chi \in \mc{U}_S} |A_{d_0}(\chi)|^2|B_{N-d_0}(\chi)|^2 \ll \phi(Q)q^{-\deg{Q}} \cdot N^{-20}.$$
We now consider the contribution from $\mc{U}_L$. Since $\chi_1 \notin \mc{U}$, and since $2(N-d_0)>N$, we may apply Lemma~\ref{QuasiMultFF} to obtain
$$
B_{N-d_0}(\chi) \ll (\tilde{Q}/\tilde{P})^3(N-d_0)^{-1/4+o(1)} \ll \phi(Q)q^{-\deg{Q}} \cdot (\tilde{Q}/\tilde{P})^3N^{-1/4+o(1)},
$$
since $d_0 \leq \tilde{Q} =o(N)$ and $\phi(Q)q^{-\deg{Q}} \gg (\log N)^{-1}$. It follows that
$$
\sum_{\chi \in \mc{U}_L} |A_{d_0}(\chi)|^2|B_{N-d_0}(\chi)|^2 \ll (\phi(Q)q^{-\deg{Q}})^2 (\tilde{Q}/\tilde{P})^6N^{-1/2+o(1)}\sum_{\chi \in \mc{U}_L}|A_{d_0}(\chi)|^2.
$$
Applying Lemma~\ref{HalMonPrim}, we deduce that
$$
\sum_{\chi \in \mc{U}_L} |A_{d_0}(\chi)|^2 \ll \frac{1}{d_0^2}\Big(1 + \deg{Q}q^{-d_0/2}|\mc{U}_L|\Big) \ll \tilde{P}^{-2}\Big(1 + \deg{Q}|\mc{U}_L|q^{-\tilde{P}/2}\Big).
$$
Since $\phi(Q) \geq q^{(1-o(1))N^{3/4}} \geq q^{\tilde{Q}} \geq q^{d_0}$ we may appeal once again to Lemma~\ref{LargeValues}, this time with $Z = N^{10}$, getting
$$
|\mc{U}_L| \ll \exp\Big(\frac{\log (q^{d_0}\phi(Q))}{d_0\log q}\Big(2\log\Big(\frac{\log(2\phi(Q))}{d_0\log q}\Big) + \log (2N^{19})\Big)\Big) \leq e^{N^{1+o(1)}/\tilde{P}}.
$$
Inserting this into the previous bound and using the fact that $\tilde{P} = N^{2/3} \geq 2N^{1.01}/\tilde{P}$ yields
$$
\sum_{\chi \in \mc{U}_L} |A_{d_0}(\chi)|^2 \ll \tilde{P}^{-2}+\deg{Q}q^{-\tilde{P}/2} e^{N^{1+o(1)}/\tilde{P}} \ll \tilde{P}^{-2}.
$$
It follows that
$$
\sum_{\chi \in \mc{U}_L} |A_{d_0}(\chi)|^2|B_{N-d_0}(\chi)|^2 \ll \Big(\phi(Q)q^{-\deg{Q}}\Big)^2 \cdot N^{-1/2+o(1)}\tilde{Q}^6\tilde{P}^{-8}.
$$
Combined with the bounds for $\mc{U}_S$, we get
\begin{align}\label{US}
\sum_{\chi \in \mc{U}} |F(\chi)|^2 &\ll (\phi(Q)q^{-\deg{Q}})^2 \cdot \Big(\tilde{Q}^2\Big(N^{-18} + \tilde{Q}^6\tilde{P}^{-8}N^{-1/2+o(1)}\Big)+\tilde{P}/\tilde{Q}+  q^{-\tilde{P}}\Big)\nonumber \\
&\ll \phi(Q)q^{-\deg{Q}} \cdot \Big(N^{-1/2+o(1)}\Big(\tilde{Q}/\tilde{P}\Big)^8 + \tilde{P}/\tilde{Q}\Big)\nonumber\\
&\ll N^{-1/18+o(1)}\phi(Q)q^{-\deg{Q}}.
\end{align}

Lastly, putting~\eqref{Rj},~\eqref{M1},~\eqref{M2} and~\eqref{US} together with~\eqref{CharOrtho} and~\eqref{CharOrtho2}, Theorem~\ref{VarAPs} follows.

\begin{rem} \label{RemOnS1}
Note that if we began by assuming that the sums in the variable $G$ in~\eqref{CharOrtho} are supported on $\mc{S}_{\mbf{P},\mbf{Q}}$ then the same proof would give the sharper estimate
\begin{align}
&\asum_{A \bmod{Q}} \Big|\sum_{\substack{G \in \mc{M}_N \\ G \equiv A\bmod{Q}}} f1_{\mc{S}_{\mbf{P},\mbf{Q}}}(G)-\frac{\chi_1(A)}{\phi(Q)}\sum_{G \in \mc{M}_N} f1_{\mc{S}_{\mbf{P},\mbf{Q}}}\bar{\chi_1}(G)\Big|^2 \nonumber \\
&\ll q^{2N-\deg{Q}} \Big(Q_1q^{-P_1/6} + N^{-1/18+o(1)}\Big). \label{MRAPonS}
\end{align}
We will use this sharper version of the theorem in Section~\ref{sec_shortexp}.
\end{rem}

\section{Matom\"{a}ki--Radziwi\l \l \ Theorem in Function Fields} \label{MRTHMSec}

In this section, we prove the following analogue of the main result in~\cite{mr-annals} (see Section~\ref{sec_involution} for the definition of $g^{\ast}$, for $g$ multiplicative).
\begin{thm}[Matom\"{a}ki--Radziwi\l\l \ theorem in function fields]\label{MRFF}
Let $f \colon  \mc{M} \to \mb{U}$ be a multiplicative function and let $1 \leq H \leq N-N^{3/4}$, with $H = H(N)$ tending to infinity with $N$. Then
\begin{align}\label{mrsum}
&\frac{1}{|\mc{M}_N|}\sum_{G_0 \in \mc{M}_N} \Big|\frac{1}{|\mc{M}_{<H}|}\sum_{\substack{G \in \mc{M}_N \\ G \in I_H(G_0)}} f(G)- \frac{1}{|\mc{M}_N|} \sum_{G \in \mc{M}_N} f(G) \overline{\chi_1^{\ast}}(G) \Big|^2 \\
&\ll (\log H)/H+N^{-1/18+o(1)},\nonumber
\end{align}
where $\chi_1$ is the Dirichlet character modulo $t^{N-H+1}$ that minimizes the map $\chi \mapsto \mc{D}_{f\bar{\chi^{*}}}(N)$. 
\end{thm}
Theorem~\ref{thm_mr_real} will follow as a special case, as we will see later in this section. 

\begin{rem} \label{RemOnS2}
In light of Remark~\ref{RemOnS1}, if we replace $f$ by $f1_{\mc{S}_{\mbf{P},\mbf{Q}}}$, with the choice of parameters $P_1,Q_1$, $P_j=j^{4j}P_1Q_1^{j-1}$ and $Q_j = j^{4j+2}Q_1^j$ for $j \geq 2$ then the bound in Theorem~\ref{MRFF} improves to $\ll Q_1^{1/2}q^{-P_1/12} + N^{-1/18+o(1)}$. 
The additional flexibility in choosing $P_1$ and $Q_1$ will be used in the next section.
\end{rem}

\begin{rem}\label{mrremark}
We can obtain the same estimate as in Theorem~\ref{MRFF} if~\eqref{mrsum} is replaced with
\begin{align*}
&\frac{1}{|\mc{M}_N|}\sum_{G_0 \in \mc{M}_N} \Big|\frac{1}{|\mc{M}_{<H}|}\sum_{\substack{G \in \mc{M}_N \\ G \in I_H(G_0)}} f(G)-\frac{1}{|\mc{M}_N|} \sum_{G \in \mc{M}_N} f(G) - \frac{1_{\chi_1 \neq \chi_0}}{|\mc{M}_N|} \sum_{G \in \mc{M}_N} f(G) \overline{\chi_1^{\ast}}(G) \Big|^2.
\end{align*}
The proof is the same, and in fact it will be clear from the application of the orthogonality relations in the proof that this quantity is never larger than~\eqref{mrsum}.
\end{rem}

\begin{proof}[Proof of Theorem~\ref{MRFF}]
The proof is very similar to the proof of Theorem~\ref{VarAPs}, just with a different set of characters. 

By the orthogonality relation~\eqref{ortho2}, we can write~\eqref{mrsum} as 
\begin{align*}
\frac{1}{q^{2N}}\sum_{\substack{\xi\in \mathcal{X}_{1,N-H}\\\xi\neq \chi_1^{*}}}\Big|\sum_{G\in \mc{M}_N}f(G)\overline{\xi}(G)\Big|^2.
\end{align*}
This is analogous to~\eqref{CharOrtho}, just with a different group of characters (see also~\cite[(2.12)]{Gor}). Now the rest of the proof follows precisely as the proof of Theorem~\ref{VarAPs} up to notation. Indeed, the only properties of the Dirichlet characters used in the proof of Theorem~\ref{VarAPs} were the lemmas from Sections~\ref{sec_prelim1} and~\ref{sec_prelim2}. In Section~\ref{sec_prelim1}, all the lemmas  are readily stated for Hayes characters, which includes both short interval characters and Dirichlet characters as special cases. Also in Section~\ref{sec_prelim2}, all the mean value estimates have perfect analogues for short interval characters, and the proofs are identical, as noted in Remark~\ref{remshort}. Moreover, the pointwise bound  offered by Lemma~\ref{QuasiMultFF} is written for more general Hayes characters, and we can take $\widetilde{\chi_1}$ there to be the short interval character of length $N-H$ that minimizes $\tilde{\chi} \mapsto \mc{D}_{f\bar{\tilde{\chi}}}(N)$. Hence, all the steps in the proof of Theorem~\ref{VarAPs} work in the same way.
\end{proof}

To deduce the real-valued case of the Matom\"aki--Radziwi\l{}\l{} theorem from this, we will need the following variant of Corollary~\ref{cor_halrep}, applicable to real-valued multiplicative functions twisted by Dirichlet characters modulo powers $t^m$, $m \geq 2$. 

\begin{lem}[Sup norm estimate for weighted Dirichlet character sums] \label{PrincPret}
Assume $\textnormal{char}(\mb{F}_q) \neq 2$. Let $N \geq 1$ and $2\leq k\leq N$. Let $f \colon  \mc{M} \to [-1,1]$ be multiplicative. Let $2 \leq k \leq N$. Then
$$
\max_{\substack{\chi \bmod{t^k} \\ \chi \neq \chi_0}} \frac{1}{q^N}\Big|\sum_{G \in \mc{M}_N} f(G)\chi(G) \Big| \ll N^{-1/4+o(1)}.
$$
\end{lem}

\begin{proof}
In light of Theorem~\ref{HalThmFF}, it suffices to show that
$$
\min_{\substack{\chi \bmod{t^k}\\\chi\neq \chi_0}} \mc{D}_{f\bar{\chi}}(N) \geq \Big(1/4-o(1)\Big) \log N.
$$
Note that if $\chi$ is a character modulo $t^k$ and $\chi'$ is the primitive character inducing $\chi$ then 
$$
\mc{D}_{f\bar{\chi}}(N) = \mc{D}_{f\bar{\chi'}}(N) + O(1),
$$ 
so it suffices to consider primitive characters modulo $t^k$. Now, suppose $\chi$ is a primitive character that is \emph{not} real. Since $\chi^2$ is not principal, arguing precisely as in the proof of Lemma~\ref{RepulsionFF} we obtain
\begin{align*}
\mb{D}(f,\chi e_{\theta};N)^2\geq (1/4-o(1))\log N
\end{align*}
(for instance, when $f$ takes values in $S^1$ the triangle inequality immediately yields $\mb{D}(f,\chi e_{\theta};N) \geq \frac{1}{2} \mb{D}(1,\chi^2e_{2\theta};N)$, and the general case follows from this as in the proof of Lemma~\ref{RepulsionFF}).
Thus we may conclude that for any \emph{primitive non-quadratic} Dirichlet character modulo $t^k$ we have 
$$
\mc{D}_{f\bar{\chi}}(N) \geq \Big(\frac{1}{4}-o(1)\Big) \log N.
$$
Furthermore, it is easy to see that there are \emph{no} primitive quadratic (non-principal) characters modulo $t^k$ for any $k \geq 2$. Indeed, suppose $\chi$ is real and primitive modulo $t^k$. Then $\chi$ cannot be periodic modulo $t^j$, for any $j < k$. To deduce a contradiction from this, set now $m := \left \lceil k/2\right \rceil < k$. Since $q\geq 3$ is odd and $\chi$ is real we have $\chi^q = \chi$, and also $mq>k$. Thus, for any $A,B \in \mb{F}_q[t]$ we have
$$\chi(B+At^m) = \chi^q(B+At^m) = \chi\Big((B+At^m)^q\Big) = \chi(B^q+qAB^{q-1}t^m) = \chi(B)^q = \chi(B).$$
Thus, in fact, $\chi$ is periodic modulo $t^m$, contradicting the fact that $\chi$ is primitive modulo $t^k$. 

Therefore, we obtain
$$
\min_{\substack{\chi \bmod{t^k} \\ \chi \neq \chi_0}} \mc{D}_{f\bar{\chi}}(N) = \min_{\substack{\chi \bmod{t^k} \\ \chi^2 \neq \chi_0 \\ \chi \text{ primitive}}} \mc{D}_{f\bar{\chi}}(N) + O(1) \geq (1/4-o(1))\log N,
$$
as claimed.
\end{proof}

\begin{proof}[Proof of Theorem~\ref{thm_mr_real}]
Assume that $f \colon  \mc{M} \ra [-1,1]$. We extend $f$ to a map on $\mb{F}_q[t]$ by requiring that $f(c) = \chi_1^\ast(c)$ for all $c \in \mb{F}_q^{\times}$, where $\chi_1$ is given by Theorem~\ref{MRFF}. 

Theorem~\ref{thm_mr_real} follows immediately from Theorem~\ref{MRFF} when $q$ is even, aside from the claim that $\chi_1$ may be assumed to be real. To see this, note that by Lemma~\ref{PrincPret}, if $\chi_1$ were not real since then the mean value of $f\bar{\chi_1}$ would contribute $\ll N^{-1/4+o(1)}$, which is anyway dwarfed by the error term in the statement of the theorem. 

When $q$ is odd, it suffices to show that if $\chi_1 \neq \chi_0$ then $\mc{D}_{f\overline{\chi_1^{\ast}}}(N) \geq \Big(1/4-o(1)\Big)\log N$, so that once again the sum in $f\overline{{\chi}_1^{\ast}}$ contributes negligibly. 

Since $f(c) = \chi_1^{\ast}(c)$ on $\mb{F}_q^\times$ and $f^\ast$ is necessarily real, Lemma~\ref{lem_dist_inv} combines with Lemma~\ref{PrincPret} to show that
$$
\mc{D}_{f\overline{{\chi_1}^{\ast}}}(N) = \mc{D}_{f^{\ast}\bar{\chi_1}}(N) + O(1) \geq \Big(\frac{1}{4}-o(1)\Big) \log N.
$$
This completes the proof.
\end{proof}

\section{Short Exponential Sums of Non-Pretentious Functions} \label{sec_shortexp}

In this section, we apply the results of the previous section to derive two function field analogues of estimates for short exponential sums weighted by a multiplicative function, due to Matom\"{a}ki, Radziwi\l\l \ and Tao~\cite{mrt}. To explain the formulation of our results, we begin by recording some of the relevant definitions. See~\cite{wooley-liu} for an excellent reference to the definitions given here.

We write $\mb{F}_q(t)$ to denote the field of rational functions of $t$ over $\mb{F}_q$. This comes equipped with the non-archimedean valuation $\lla \cdot \rra$ such that if $G = \sum_{j = N}^{\infty} a_{-j}t^{-j}$ for an integer $N$ with $a_N\neq 0$, then $\lla G\rra = q^{-N}$. The completion of $\mb{F}_q(t) = \mb{F}_q(1/t)$ with respect to this valuation is the set $\mb{K}_{\infty}(t) := \mb{F}_q((1/t))$ of formal Laurent series in $1/t$ with a finite number of non-negative power terms. We define $\mb{T}$ to be the unit ball of $\mb{K}_{\infty}(t)$ with respect to $\lla \cdot \rra$, i.e.,
$$
\mb{T} := \{\alpha \in \mb{K}_{\infty}(t)  \colon  \lla \alpha \rra < 1\} \cong \mb{K}_{\infty}(t)/\mb{F}_q[t].
$$
That is, $\mb{T}$ is the set of formal power series in $1/t$. This set forms a compact abelian group under addition, and thus comes equipped with a normalized Haar measure, which we shall denote by $d\alpha$. The Pontryagin dual group consists of the characters $\{\alpha \mapsto e_{\mb{F}}(G\alpha)\}_{G \in \mb{F}_q[t]}$, where, given $\alpha \in \mb{K}_{\infty}(t)$, we have written
$$
e_{\mb{F}}(\alpha) := e\Big(\frac{\text{tr}_{\mb{F}_q/\mb{F}_p}(a_{-1}(\alpha))}{p}\Big),
$$
writing $a_{-1}(\alpha)$ to denote the coefficient of the term $t^{-1}$ in the expansion of $\alpha$. An important feature of these characters is that 
$$
\int_{\mb{T}} e_{\mb{F}}(G\alpha) d\alpha = \begin{cases} 1 &\text{ if $G = 0$}\\ 0 &\text{ otherwise,}\end{cases}
$$
in analogy to the orthogonality of additive characters on $\mb{R}/\mb{Z}$. 

Our goal in this section will be to prove the following two results. The first is an estimate for exponential sums with multiplicative coefficients over short intervals that applies to complex-valued $f$, provided that $f$ is \emph{Hayes} non-pretentious (see Definition~\ref{def_hayesnp}). The second result concerns such exponential sums with \emph{real-valued} functions $f$, for which only the usual notion of (Dirichlet) non-pretentiousness needs to be assumed. The first of these theorems will be of relevance in proving the logarithmically averaged binary Chowla conjecture in this context. 

In the theorems below, given $1 \leq H \leq N-N^{3/4}$ and $f \colon  \mc{M} \to \mb{U}$ a multiplicative function, set 
\begin{align*}
M_{\text{Hayes}}(f;N,H) &:= \min_{M \in \mc{M}_{\leq H}} \min_{\psi \bmod{M}} \min_{\substack{\xi \text{ short} \\ \text{len}(\xi) \leq N}} \mc{D}_{f\bar{\psi}\bar{\xi}}(N) \\
M_{\textnormal{Dir}}(f;N,H) &:= \min_{M \in \mc{M}_{\leq H}}\min_{\psi \bmod{M}} \mc{D}_{f\bar{\psi}}(N).
\end{align*}
It is clear from the definitions that $M_{\textnormal{Hayes}}(f;N,H) \leq M_{\textnormal{Dir}}(f;N,H)$.
\begin{thm}\label{ShortExpSumFFComplex}
Let $1 \leq H \leq N-N^{3/4}$. Let $f \colon  \mc{M} \to \mb{U}$ be multiplicative. Then
$$
\sup_{\alpha \in \mb{T}}\frac{1}{|\mc{M}_N|} \sum_{G_0 \in \mc{M}_N} \frac{1}{|\mc{M}_{<H}|} \Big|\sum_{\substack{G \in \mc{M}_N \\ G \in I_H(G_0)}} f(G)e_{\mb{F}}(G\alpha)\Big| \ll \frac{\log H}{H} + N^{-1/(2000\log q)} + Me^{-M/100},
$$
where $M:=1+M_{\textnormal{Hayes}}(f;N,H)$.
\end{thm}

\begin{thm}\label{ShortExpSumFF}
Assume $q$ is odd. Let $1 \leq H \leq N-N^{3/4}$. Let $f \colon  \mc{M} \to [-1,1]$ be a multiplicative function. Then
\begin{align}\label{EXPSUMFF}
\sup_{\alpha \in \mb{T}}\frac{1}{|\mc{M}_N|} \sum_{G_0 \in \mc{M}_N} \frac{1}{|\mc{M}_{<H}|} \Big|\sum_{\substack{G \in \mc{M}_N \\ G \in I_H(G_0)}} f(G)e_{\mb{F}}(G\alpha)\Big| \ll \frac{\log H}{H} + N^{-1/(2000\log q)} +Me^{-M/100}, 
\end{align}
where $M:=1+M_{\textnormal{Dir}}(f;N,H)$
\end{thm}

We will deduce both of these results from the following result about completely multiplicative functions.

\begin{thm}\label{ShortExpSumFFCM}
Let $1 \leq H \leq N-N^{3/4}$. Let $f  \colon  \mc{M} \ra \mb{U}$ be completely multiplicative. Then Theorem~\ref{ShortExpSumFFComplex} holds for $f$. Moreover, if $f$ is real-valued and $q$ is odd then Theorem~\ref{ShortExpSumFF} holds for $f$.
\end{thm}
We will begin by proving Theorem~\ref{ShortExpSumFFCM}; we will prove the deduction of Theorems~\ref{ShortExpSumFFComplex} and~\ref{ShortExpSumFF} for general 1-bounded multiplicative functions at the end of this section. The proofs of the complex and real cases begin the same way. We shall thus begin both simultaneously, then highlight where the differences arise below. 

We proceed using the circle method, as in~\cite{mrt}, splitting into cases according to whether $\alpha$ lies in a major or minor arc (to be defined momentarily). In the function field setting, arcs can be determined via the following form of Dirichlet's theorem.

\begin{lem}[Dirichlet's theorem in function fields]
Suppose $\alpha \in \mb{T}$. Given $M \geq 1$ we can find $g \in \mc{M}_{\leq M}$ and $a \in \mb{F}_q[t]$ coprime to $g$ with $\deg{a} < \deg{g}$ such that $\lla g\alpha - a \rra \leq q^{-M}.$
\end{lem}

\begin{proof}
This follows from the pigeonhole principle, just as in the integer setting.
\end{proof}

We are now ready to embark on the proof of Theorem~\ref{ShortExpSumFFCM}. We will first prove the following closely related statement.

\begin{prop}\label{expprop}\ Let $1 \leq H \leq N-N^{3/4}$. Let $f \colon  \mc{M} \to \mb{U}$ be completely multiplicative. Suppose that
\begin{align}\label{Wbound}
10\log H\leq \min\{(\log N)/(100\log q),H/10,M_{*}(f;N,H)/100\},    
\end{align}
where $M_{\ast}(f;N,H) :=  1+M_{\text{Dir}}(f;N,H)$ if $f$ is real-valued and $q$ is odd, and $M_{\ast}(f;N,H) :=1+ M_{\text{\textnormal{Hayes}}}(f;N,H)$ otherwise.
Then
\begin{align*}
\sup_{\alpha \in \mb{T}}\frac{1}{|\mc{M}_N|} \sum_{G_0 \in \mc{M}_N} \frac{1}{|\mc{M}_{<H}|} \Big|\sum_{\substack{G \in \mc{M}_N \\ G \in I_H(G_0)}} f(G)e_{\mb{F}}(G\alpha)\Big| \ll  (\log H)/H+ N^{-1/40} + e^{-M_{\text{Hayes}}(f;N,H)/20}.   
\end{align*}
Moreover, if $f$ is real-valued and $q$ is odd, we can replace $M_{\text{Hayes}}$ with $M_{\text{Dir}}$.
\end{prop}

 By~\eqref{Wbound}, we can choose $1 \leq W \leq X \leq H$ such that
$$
10\log H \leq W \leq \min\{(\log N)/(100\log q), H/10, M_{\ast}(f; N,H)/10\},
$$ 
and set $X = H-W$ (so that $W \leq X/2$). In general, we define arcs of the form 
$$
\mf{M}_{a,g}(X) := \{\alpha \in \mb{T}  \colon  \lla g\alpha - a\rra \leq q^{-X} \}.
$$
The major arcs of length $X$ and degree $W$ are defined by
$$
\mf{M}(X,W) := \bigcup_{\deg{g} \leq W} \bigcup_{\substack{a \bmod{g} \\ (a,g) = 1}} \mf{M}_{a,g}(X),
$$
and the minor arcs are then defined by 
$$
\mathfrak{m} = \mf{m}(X,W) := \mb{T} \bk \mc{M}(X,W).
$$ 

Let $P_1 := 100W$ and $Q_1 := H/3$, and let $\mc{S} = \mc{S}_{\mbf{P},\mbf{Q}}$, with $P_j,Q_j$ defined in terms of $P_1$ and $Q_1$ as in Section~\ref{VarSec}. For the same reason as in Section~\ref{VarSec}, it will be advantageous to replace the expression on the left-hand side in~\eqref{EXPSUMFF} by 
\begin{equation}\label{EXPSUMFFwithS}
\sup_{\alpha \in \mb{T}}\frac{1}{|\mc{M}_N|} \sum_{G_0 \in \mc{M}_N} \frac{1}{|\mc{M}_{<H}|} \Big|\sum_{\substack{G \in \mc{M}_N \\ G \in I_H(G_0)}} f1_{\mc{S}}(G)e_{\mb{F}}(G\alpha)\Big|.
\end{equation}
By Lemma~\ref{SieveErat} and the triangle inequality, the difference between this latter expression and the one in~\eqref{EXPSUMFF} is $O(P_1/Q_1)$. We will thus focus our attention mostly on the estimation of~\eqref{EXPSUMFFwithS}.

As mentioned, the expression in~\eqref{EXPSUMFFwithS} will be treated differently according to whether $\alpha$ lies in a major arc or a minor arc. We start with the minor arc case, where the argument has some resemblance to the derivation of the orthogonality criterion for multiplicative functions~\cite{bourgain-sarnak-ziegler}, and which can be derived independently of the results of the last two sections.

\subsection{The Minor Arcs}

We fix $\alpha \in \mf{m}$. In order to proceed in estimating~\eqref{EXPSUMFFwithS}, we shall need the following basic result.
\begin{lem} \label{TYPEI}
Let $\alpha \in \mb{K}_{\infty}(t)$ and $H \geq 1$. Then
$$\sum_{\text{deg}(F) < H} e_{\mb{F}}(F\alpha) = q^H1_{\lla \alpha \bmod{1} \rra \leq q^{-H-1}}.$$
\end{lem}

\begin{proof}
This is standard, see e.g.~\cite[Lemma 7]{kub}.
\end{proof}

We will also need the following estimate, connected with Lemma~\ref{TYPEI}.

\begin{lem}\label{lem_rightcount}
Let $\alpha \in \mf{M}_{a,g}(X)$, where $W < \deg{g} \leq X$ and $(a,g) = 1$. Let $100W \leq k \leq H/3$. Then
$$
|\{\deg{F} < k  \colon  \lla F\alpha \bmod{1}\rra < q^{-H+k-1}\}| \ll q^{k-W}.
$$
\end{lem}

\begin{proof}
Write $\beta := \alpha-a/g$. By assumption, we have $\lla \beta \rra \leq q^{-X-\deg{g}}$. Since $(A,B) \mapsto \lla A-B\rra$ is an ultrametric, for any $F \in \mb{F}_q[t]$ we have
$$
\lla F\alpha \bmod{1}\rra \leq \max\{ \lla Fa/g \bmod{1}\rra, \lla F\beta \bmod{1}\rra\},
$$
with equality whenever the two valuations on the right-hand side differ. 

Note that if $g \nmid F$ then as $(a,g) = 1$, we can write $Fa = Mg + L$ with $L \not \equiv 0 \bmod{g}$. Hence, if $\deg{F} \leq k$ then
$$
\lla Fa/g \bmod{1} \rra = q^{\deg{L}-\deg{g}} \geq q^{-\deg{g}} > q^{\deg{F}-X - \deg{g}} \geq \lla F\beta\bmod{1}\rra,
$$
using $X \geq H/2 > k$. On the other hand, if $g|F$ then $\lla Fa/g \bmod{1} \rra = 0 \leq \lla F\beta \bmod{1} \rra$. In particular, we have
$$
\lla F\alpha \bmod{1} \rra \begin{cases} = \lla Fa/g \bmod{1} \rra &\text{ if $g\nmid F$} \\ \leq \lla F\beta \bmod{1} \rra &\text{ if $g|F$.}\end{cases}
$$
Let $\mc{E} := \{\deg{F} < k  \colon  \lla F\alpha \bmod{1}\rra < q^{-H+k-1}\}$. Consider separately the number of $F\in \mc{E}$ with $g\mid F$ and $g\nmid F$. Note that
$$
|\{F \in \mc{E}  \colon  g|F\}| \leq |\{\deg{F} < k  \colon  g|F\}| \leq 1 + q^{k-\deg{g}}\ll q^{k-W}.
$$
Next, consider the contribution to $\mc{E}$ from $F$ that are not divisible by $g$. We first observe that there are no $F \in \mc{E}$ with $\deg{F} \geq \deg{g}$. Indeed, if such an $F$ belonged to $\mc{E}$ then
$\lla F\alpha \bmod{1} \rra = \lla Fa/g \bmod{1}\rra \geq q^{-\deg{g}}$. This implies the chain of inequalities
\begin{align*}
q^{-k} < q^{-\deg{F}} \leq q^{-\deg{g}} \leq \lla F\alpha \bmod{1}\rra < q^{-H+k-1},
\end{align*}
which are conflicting since $k \leq H/3$. 

We may therefore assume that $\deg{F} < \deg{g}$. Suppose next that $\deg{Fa} < \deg{g}$ as well. Then 
$$
\lla F\alpha \bmod{1}\rra = \lla Fa/g\bmod{1}\rra = q^{\deg{Fa}-\deg{g}} \geq q^{\deg{F}-\deg{g}}.
$$ 
Thus, if $F \in \mc{E}$ then we must have 
$$
\deg{F} \leq \deg{g} + k-H-1 \leq k-H+X -1 < k-W,
$$ 
since $W=H-X$. Hence,
$$
|\{F \in \mc{E}  \colon  g\nmid F, \text{ and } \deg{F} \geq \deg{g} \text{ or } \deg{Fa} < \deg{g}\}| \leq |\{\deg{F} < k-W\}| \ll q^{k-W}.
$$
It remains to consider those $F$ with $\deg{F} < \deg{g} \leq \deg{Fa}$. Observe that
\begin{align*}
&|\{F \in \mc{E}  \colon  g\nmid F, \deg{F} < \deg{g} \leq \deg{Fa}\}| \\
&= \sum_{0 \leq m < \deg{g}} \sum_{\deg{B} = m} |\{F  \colon  \deg{g}-\deg{a} \leq \deg{F} < \min\{k,\deg{g}\}  \colon  Fa \equiv B \bmod{g}\\
&\quad \quad \quad \quad \quad \quad \quad \quad \quad \quad \quad \textnormal{and}\quad\lla Fa/g\bmod{1} \rra <  q^{-H+k-1}\}|.
\end{align*}
Note that $\lla Fa/g \bmod{1} \rra = \lla B/g\rra = q^{m-\deg{g}}$ whenever $\deg{B} = m$ and $Fa \equiv B \bmod{g}$, and so $F \in \mc{E}$ under these conditions only if $0 \leq m < \deg{g} - H + k -1$. This condition is empty if $\deg{g} \leq H-k+1$ so we assume otherwise (and hence $\deg{g} > k$). We may thus bound the above by
$$
\leq \sum_{0 \leq m < \deg{g}-H+k-1} \sum_{\deg{B} = m} |\{\deg{g}-\deg{a} \leq \deg{F} < k  \colon  F\equiv \bar{a}B \bmod{g}\}|,
$$
where $\bar{a}$ is the inverse of $a\bmod {g}$. Since $\deg{F} < \deg{g}$, the cardinality above is $\leq 1$, and thus
\begin{align*}
|\{F \in \mc{E}  \colon  g\nmid F, \deg{F} < \deg{g} \leq \deg{Fa}\}| &\leq 1 + \sum_{0 \leq m < \deg{g}-H+k-1} \sum_{\deg{B} = m} 1 \\
&\leq 1+ q\sum_{0 \leq m < \deg{g}-H+k-1} q^m \ll 1+ q^{\deg{g}-H+k}\\
&\ll q^{X-H+k} = q^{k-W}.
\end{align*}
It follows that
\begin{align*}
|\mc{E}| &\leq |\{F \in \mc{E}  \colon  g|F\}| + |\{F \in \mc{E}  \colon  g\nmid F, \deg{F} \geq \deg{g} \text{ or } \deg{Fa} < \deg{g}\}|\\
&+ |\{F \in \mc{E}  \colon  \deg{F} < \deg{g} \leq \deg{Fa}\}|\ll q^{k-W},
\end{align*}
as claimed.
\end{proof}
Let $\alpha \in \mf{m}.$ For each $G_0 \in \mc{M}_N$ let $\theta(G_0) \in S^1$ be chosen so as to write~\eqref{EXPSUMFFwithS} as
\begin{align}\label{eq100}
\Sigma_{\mc{S}}(\alpha):=\frac{1}{|\mc{M}_N|} \sum_{G_0 \in \mc{M}_N} \frac{\theta(G_0)}{|\mc{M}_{<H}|} \sum_{\substack{G \in \mc{M}_N \\ G \in I_H(G_0)}} f1_{\mc{S}}(G)e_{\mb{F}}(\alpha G).
\end{align}
Since $f$ is completely multiplicative, upon applying Lemma~\ref{Ramare} we obtain
\begin{align*}
&\Sigma_{\mc{S}}(\alpha)\\
&= \frac{1}{|\mc{M}_N|} \sum_{G_0 \in \mc{M}_N} \frac{\theta(G_0)}{|\mc{M}_{<H}|}\sum_{G'\in \mc{M}_{\leq N}} \frac{f(G')}{1+\omega_{[P_1,Q_1]}(G')} \sum_{\substack{R \in \mc{P} \\ P_1 \leq \deg{R} \leq \min\{Q_1,N-\deg{G'}\} \\ RG'\in I_H(G_0) \\ \deg{RG'} = N}} f(R)e_{\mb{F}}(RG'\alpha) \\
&\quad \quad \quad \quad \quad \quad \quad \quad + O\Big(q^{-N-H} \sum_{G_0 \in \mc{M}_N} \sum_{P_1 \leq d \leq Q_1} \sum_{R \in \mc{P}_d} |\{G \in I_H(G_0)  \colon  R^2|G\}|\Big) \\
&= \frac{1}{|\mc{M}_N|} \sum_{G_0 \in \mc{M}_N} \frac{\theta(G_0)}{|\mc{M}_{<H}|}\sum_{G'\in \mc{M}_{\leq N}} \frac{f(G')}{1+\omega_{[P_1,Q_1]}(G')} \sum_{\substack{R \in \mc{P} \\ P_1 \leq \deg{R} \leq \min\{Q_1,N-\deg{G'}\} \\ RG'\in I_H(G_0) \\ \deg{RG'} = N}} f(R)e_{\mb{F}}(RG'\alpha)\\
&\quad \quad \quad \quad \quad \quad \quad \quad \quad \quad \quad \quad \quad \quad \quad \quad \quad \quad \quad \quad \quad \quad \quad \quad \quad \quad \quad \quad \quad \quad \quad \quad + O(q^{-P_1}).
\end{align*}
We pull the summation over $G'$ out, split the sum over $R$ according to degree and apply the triangle inequality to get
\begin{align*}
|\Sigma_{\mc{S}}(\alpha)| &\leq \sum_{P_1 \leq k \leq Q_1} \frac{1}{|\mc{M}_N|} \sum_{G' \in \mc{M}_{N-k}} \Big|\sum_{R \in \mc{P}_k} f(R)e(G'R\alpha) \frac{1}{|\mc{M}_{<H}|} \sum_{\substack{G_0 \in \mc{M}_N \\ G'R \in I_H(G_0)}} \theta(G_0)\Big| + O(H^{-50}),
\end{align*}
since $P_1 \geq 1000 \log H$. We apply H\"{o}lder's inequality to the sum over $G'$, getting
\begin{align}
|\Sigma_{\mc{S}}(\alpha)| &\ll q^{-N/4-H}\sum_{P_1 \leq k \leq Q_1} q^{-3k/4} \mc{T}_k^{1/4} + H^{-50}, \label{InsertTk}
\end{align}
where for each $P_1 \leq k \leq Q_1$ we define
\begin{align*}
\mc{T}_k &:= \sum_{R_1,R_2,R_3,R_4 \in \mc{P}_k } f(R_1)f(R_2)\bar{f(R_3)f(R_4)} \sum_{G_1,G_2,G_3,G_4\in \mc{M}_N} \theta(G_1)\theta(G_2)\bar{\theta(G_3)\theta(G_4)} \\
&\sum_{\substack{G' \in \mc{M}_{N-k}\\ G'R_j \in I_H(G_j) \,\,\forall j}}e_{\mb{F}}(G'\alpha(R_1+R_2-R_3-R_4)).
\end{align*}
Fix $P_1 \leq k \leq Q_1$ for the time being. Split the sums over $G_j$ according to their residue classes $A_j \bmod{R_j}$. Writing $G_j = D_jR_j+A_j$, we know that $I_H(G_j) = I_H(D_jR_j)$ since $\deg{A_j} < \deg{R_j} \leq Q_1 < H$. Thus, we can rewrite $\mc{T}_k$ as
\begin{align*}
\mc{T}_k &= \sum_{\substack{R_1,R_2,R_3,R_4 \in \mc{P}_k \\1 \leq j \leq 4}}f(R_1)f(R_2)\bar{f(R_3)f(R_4)} \sum_{\substack{A_1,A_2,A_3,A_4 \\ A_j \bmod{R_j}\,\, \forall 1 \leq j \leq 4}} \sum_{D_1 \in \mc{M}_{N-k}} \theta(D_1R_1+A_1) \\
&\cdot \sum_{\substack{D_2,D_3,D_4 \in \mc{M}_{N-k}}} \theta(D_2R_2+A_2)\bar{\theta(D_3R_3+A_3)\theta(D_4R_4+A_4)}\\
&\cdot \sum_{\substack{G' \in \mc{M}_{N-k} \\ G'\in I_{H -k}(D_j)\,\, \forall 1\leq j \leq 4}} e_{\mb{F}}(G'\alpha(R_1+R_2-R_3-R_4)).
\end{align*}
We observe now that 
$$
G' \in \bigcap_{1 \leq j \leq 4} I_{H-k}(D_j) \Longleftrightarrow G' \in I_{H-k}(D_1) \text{ and } \deg{D_j-D_1} < H-k \text{ for all $1 \leq j \leq 4$.}
$$
Hence, making the change of variables $L := G'-D_1$, we can recast the above expression for $\mc{T}_k$ as
\begin{align*}
&\sum_{\substack{R_j \in \mc{P}_k \\ 1 \leq j \leq 4}}f(R_1)f(R_2)\bar{f(R_3)f(R_4)} \sum_{D_1 \in \mc{M}_{N-k}}\sum_{\substack{A_j \bmod{R_j} \\ 1 \leq j \leq 4}} \theta(D_1R_1+A_1)e_{\mb{F}}(D_1\alpha(R_1+R_2-R_3-R_4)) \\
&\cdot \sum_{\substack{D_2,D_3,D_4 \in \mc{M}_{N-k} \\ \deg{D_j-D_1} < H-k\,\, \forall j}} \theta(D_2R_2+A_2) \bar{\theta(D_3R_3+A_3)} \bar{\theta(D_4R_4+A_4)}\\
&\cdot\sum_{\deg{L} <H-k} e_{\mb{F}}(L\alpha(R_1+R_2-R_3-R_4)).
\end{align*}
Note that now the inner sum over $L$ is decoupled from the sums over $A_j$ and $D_j$. Given $D_1 \in \mc{M}_{N-k}$ fixed, there are $\ll q^{H-k}$ choices of each of $D_2$, $D_3$ and $D_4$ to satisfy the condition $\deg{D_j-D_1} < H-k$. Furthermore, there are $\ll q^{4k}$ choices of 4-tuples of residue classes $A_1,A_2,A_3,A_4$ to their respective moduli $R_1,R_2,R_3$ and $R_4$. Recalling that $\theta(\cdot)$ is unimodular and bounding trivially in $D_1 \in \mc{M}_{N-k}$, it follows that
\begin{align*}
\mc{T}_k &\ll q^{4k} \cdot q^{3(H-k)} \cdot q^{N-k} \sum_{R_1,R_2,R_3,R_4 \in \mc{P}_k} \Big|\sum_{\deg{L} <H-k} e_{\mb{F}}(L\alpha(R_1+R_2-R_3-R_4))\Big| \\
&\ll q^{3H+N} \sum_{R_1,R_2,R_3,R_4 \in \mc{P}_k} \Big|\sum_{\deg{L} <H-k} e_{\mb{F}}(L\alpha(R_1+R_2-R_3-R_4))\Big|
\end{align*}
We arrange the 4-tuples $(R_1,R_2,R_3,R_4) \in \mc{P}_k^4$ according to the values of $F:= R_1+R_2 - R_3 - R_4 \in \mb{F}_q[t]$; note that since the $R_j$ are all monic, $\deg{F} < k$. By Corollary~\ref{PRIM4TUP}, there are $\ll q^{3k}/k^4$ such representations of $F$ in terms of irreducibles $R_j \in \mc{P}_k$. It follows that
$$\mc{T}_k \ll \frac{q^{3(H+k)+N}}{k^4} \sum_{\deg{F} < k}\Big|\sum_{\deg{L}<H-k} e_{\mb{F}}(LF\alpha)\Big|.$$
By Lemma~\ref{TYPEI}, we can evaluate the exponential sum to yield
\begin{align*}
\mc{T}_k &\ll \frac{q^{3(H+k)+N}}{k^4} \sum_{\deg{F} < k} q^{H-k} 1_{\lla F\alpha \bmod{1}\rra < q^{-H+k-1}} \\
&= \frac{q^{4H+2k+N}}{k^4} |\{\deg{F} < k \colon \lla F\alpha \bmod{1}\rra < q^{-H+k-1}\}|.
\end{align*}
Since $\alpha \in \mf{m}$ there must be a $g \in \mc{M}$ and a reduced residue class $a \bmod g$ such that $W < \deg{g} \leq X$, $(a,g) = 1$ and $\alpha \in \mf{M}_{a,q}(X)$. Since $P_1 \leq k \leq Q_1$ and given our choice of $P_1,Q_1$, Lemma~\ref{lem_rightcount} yields
$$
|\{\deg{F} < k  \colon  \lla F\alpha \bmod{1} \rra < q^{-H+k-1}\}| \ll q^{k-W},
$$
so that we finally obtain the estimate 
$$
\mc{T}_k \ll\frac{q^{4H+3k+N}}{k^4}q^{-W}.
$$
Taking fourth roots of both sides and inserting this into~\eqref{InsertTk}, we get
\begin{align*}
|\Sigma_{\mc{S}}(\alpha)| &\ll q^{-N/4-H} \sum_{P_1 \leq k \leq Q_1} q^{-3k/4} \Big(q^{3k+4H+N-W}/k^4\Big)^{1/4} + H^{-50} \ll q^{-W/4} \sum_{P_1 \leq k \leq Q_1}1/k + H^{-50}\\
&\ll \log(Q_1/P_1)q^{-W/4} + H^{-50}.
\end{align*}
In light of the choices $W \geq 10 \log H$, $P_1 = 100W$, $Q_1 = H/3$, this leads, finally, to the bound
\begin{align*}
&\max_{\alpha \in \mf{m}(X,W)} \frac{1}{|\mc{M}_N|} \sum_{G_0 \in \mc{M}_N}\Big|\frac{1}{|\mc{M}_{<H}|} \sum_{\substack{G \in \mc{M}_N \\ G\in I_H(G_0)}} f(G)e_{\mb{F}}(G\alpha)\Big| \ll \max_{\alpha \in \mf{m}(X,W)} |\Sigma_{\mc{S}}(\alpha)| + \frac{P_1}{Q_1}\\
&\ll\log(Q_1/P_1)q^{-W/4} + \frac{P_1}{Q_1} + H^{-50} \ll \frac{P_1}{Q_1}.
\end{align*}

\subsection{The Major Arcs}
Next, we turn to the estimation of the major arcs, where the Matom\"aki--Radziwi\l{}\l{} theorem in function fields will be put into use. Fix $g \in \mc{M}_{\leq W}$ and a reduced residue class $a$ modulo $g$ coprime to $g$. Suppose that $\alpha \in \mf{M}_{a,g}(X)$. We shall estimate $\Sigma_{\mc{S}}(\alpha)$ (given by~\eqref{EXPSUMFFwithS}) in this case as well. 

Write $e_{\mb{F}}(G\alpha) = e_{\mb{F}}(Ga/g)e_{\mb{F}}(G\beta)$, and set $\gamma := \deg{g}-1$. Since $X+\gamma < H$, for each $G_0 \in \mc{M}_N$ we may decompose
$$
I_H(G_0) = \bigsqcup_{\deg{G'} < H-X-\gamma} I_{X+\gamma}(G_0 + t^{X+\gamma}G').
$$
As $\beta = \sum_{j \geq X+\gamma+1} b_jt^{-j}$, it follows that $e_{\mb{F}}(\beta G)$ is constant on $I_{X+\gamma}(G_0+t^{X+\gamma}G')$, for each $G'$ in the union. Splitting the inner sum over $G$ in $\Sigma_{\mc{S}}(\alpha)$ into pieces supported on each of these shorter intervals and applying the triangle inequality, we obtain
\begin{align*}
|\Sigma_{\mc{S}}(\alpha)| &\leq \frac{1}{|\mc{M}_N|} \sum_{G_0 \in \mc{M}_N} \frac{1}{|\mc{M}_{<H}|}\sum_{\deg{G'} < H-X-\gamma} \Big|\sum_{\substack{G \in \mc{M}_N \\ G \in I_{X+\gamma}(G_0+t^{X+\gamma}G')}} f1_{\mc{S}}(G)e_{\mb{F}}(Ga/g)\Big|.\\
&=: \frac{1}{|\mc{M}_N|} \sum_{G_0 \in \mc{M}_N} \frac{1}{|\mc{M}_{<H}|}\sum_{\deg{G'}< H-X-\gamma}|\Sigma_{\mc{S}}(\alpha;G_0,G')|.
\end{align*}
For a Dirichlet character $\psi \bmod{Q}$ recall that the Gauss sum $\tau(\psi)$ of $\psi$ is defined as
$$
\tau(\psi) := \sum_{G \bmod{Q}} \psi(G) e_{\mb{F}}(G/M).
$$
It is well-known, as in the number field setting, that $|\tau(\psi)| \leq q^{\tfrac 12 \deg{Q}}$. Expanding $e_{\mathbb{F}}(Ga/g)$ in terms of Dirichlet characters $\bmod g$, separating $G$ according to the greatest common divisor $D = (G,g)$, we can rewrite
\begin{align*}
\Sigma_{\mc{S}}(\alpha; G_0,G') &= \sum_{\substack{G \in \mc{M}_N \\ G \in I_{X+\gamma}(G_0+t^{X+\gamma}G')}} f1_{\mc{S}}(G)e_{\mb{F}}(Ga/g) \\
&= \sum_{D|g}\frac{f(D)}{\phi(g/D)} \sum_{\psi \bmod{g/D}} \bar{\psi}(a)\tau(\psi) \sum_{\substack{G \in \mc{M}_{N-\deg{D}} \\ GD \in I_{X+\gamma}(G_0+t^{X+\gamma}G')}} f1_{\mc{S}}(G)\bar{\psi}(G)
\end{align*}
for each $G_0,G'$ in their respective ranges; here we have used the fact that as $\deg{D} \leq W < P_1$ we have $1_{\mc{S}}(DG) = 1_{\mc{S}}(G)$.

As in the minor arc case, we separate $G_0$ and $G'$ according to residue classes modulo $D$. Write $G_0 = DG_0' + A$, $G' = DG'' + B$ and $t^{X+\gamma} = DT_D+C$, so that
$$t^{X+\gamma}G' = D(DG''T+CG'' + BT_D)+BC.$$
Then as $\deg{D} \leq W < X/2$, we have $\deg{A}, \deg{BC} < X + \gamma$ and thus 
$$
I_{X+\gamma}(G_0+t^{X+\gamma}G') = I_{X+\gamma}(D(G_0'+DG''T_D + BT_D+CG'')) = I_{X+\gamma}(D(G_0' + t^{X+\gamma}G'' + BT_D)).
$$ 
Hence, we see that 
$$
GD \in I_{X+\gamma}(G_0+t^{X+\gamma}G') \text{ if and only if } G \in I_{X+\gamma-\deg{D}}(G_0'+t^{X+\gamma}G''+BT_D).
$$
It follows, using the triangle inequality and the bound $|\tau(\psi)| \leq q^{\deg{g/D}/2}$ that
\begin{align*}
&|\Sigma_{\mc{S}}(\alpha)|\\
&\ll \sum_{D|g} \frac{q^{\frac{1}{2}\deg{g/D}}}{\phi(g/D)} \sum_{\psi \bmod{g/D}} \sum_{A,B \bmod{D}} \frac{1}{|\mc{M}_N|} \sum_{DG_0'+A \in \mc{M}_{N}} \frac{|\mc{M}_{<X+\gamma-\deg{D}}||\mc{M}_{<H-X-\gamma-\deg{D}}|}{|\mc{M}_{<H}|} \\
&\cdot \frac{1}{|\mc{M}_{<H-X-\gamma-\deg{D}}|}\sum_{\deg{DG''+B} <H-X-\gamma} \Big|\frac{1}{|\mc{M}_{< X+\gamma-\deg{D}}|} \sum_{\substack{G \in \mc{M}_{N-\deg{D}} \\ G\in I_{X+\gamma-\deg{D}}(G_0'+t^{X+\gamma}G'' + BT_D)}} f1_{\mc{S}}(G)\bar{\psi}(G)\Big| \\
&\ll q^{\frac{1}{2}\deg{g}}\sum_{D|g} q^{-\frac{3}{2}\deg{D}}\max_{\psi \bmod{g/D}}\, \max_{\deg{G''} < H-X-\gamma-\deg{D}}\, \max_{B\bmod{D}} \\
&\cdot \frac{1}{|\mc{M}_{N-\deg{D}}|} \sum_{G_0' \in \mc{M}_{N-\deg{D}}} \Big|\frac{1}{|\mc{M}_{<X+\gamma-\deg{D}}|} \sum_{\substack{G \in \mc{M}_{N-\deg{D}}\\ G \in I_{X+\gamma-\deg{D}}(G_0'+t^{X+\gamma}G'' + BT_D)}} f1_{\mc{S}}(G)\bar{\psi}(G)\Big|.
\end{align*}
We observe that $\deg{BT_D} < X+\gamma \leq \deg{G''t^{X+\gamma}}$, so that $\deg{G''t^{X+\gamma}+BT_D} \leq H-\deg{D} < N-\deg{D}$. We can thus make the change of variables $G_0'' := G_0' + G''t^{X+\gamma}+BT_D$ to finally obtain
\begin{align*}
|\Sigma_{\mc{S}}(\alpha)| &\ll q^{\frac{1}{2}\deg{g}} \max_{\substack{D|g \\ \psi \bmod{g/D}}} \frac{1}{|\mc{M}_{N-\deg{D}}|}\\
&\cdot\sum_{G_0'' \in \mc{M}_{N-\deg{D}}} \Big|\frac{1}{|\mc{M}_{<X+\gamma-\deg{D}}|} \sum_{\substack{G \in \mc{M}_{N-\deg{D}} \\ G \in I_{X+\gamma-\deg{D}}(G_0'')}} f1_{\mc{S}}(G)\bar{\psi}(G)\Big|.
\end{align*}
Let $D_1$ be a divisor of $g$ such that some character $\psi_1\bmod{D_1}$ yields the maximal contribution among the characters whose modulus divides $g$. Put $d := \deg{D_1}$, so that $d \leq W < N/2$. By Theorem~\ref{HalThmFF} and Lemma~\ref{SieveErat}, we get
\begin{align*}
&\frac{1}{|\mc{M}_{N-d}|} \Big|\sum_{G \in \mc{M}_{N-d}} f1_{\mc{S}}(G)\bar{\psi}_1(G)\Big| \\
&\leq \frac{1}{|\mc{M}_{N-d}|} \Big|\sum_{G \in \mc{M}_{N-d}} f(G)\bar{\psi}_1(G)\Big| + \frac{1}{|\mc{M}_{N-d}|}\Big|\sum_{G \in \mc{M}_{N-d}} f(G)\bar{\psi}_1(G)1_{\mc{S}^c}(G)\Big| \\
&\ll (1+\mc{D}_{f\bar{\psi}_1}(N))e^{-\mc{D}_{f\bar{\psi}_1}(N)} + \min\left\{\frac{P_1}{Q_1},(1+\mc{D}_{f\bar{\psi}_1 1_{\mc{S}^c}}(N))e^{-\mc{D}_{f\bar{\psi}_1 1_{\mc{S}^c}}(N)}\right\} \\
&\ll (1+\mc{D}_{f\bar{\psi}_1}(N))e^{-\mc{D}_{f\bar{\psi}_1}(N)} + \frac{P_1}{Q_1}\min\left\{1,(Q_1/P_1)^3(1+\mc{D}_{f\bar{\psi}_1}(N))e^{-\mc{D}_{f\bar{\psi}_1}(N)}\right\},
\end{align*}
where in the last step we used the fact that
\begin{align*}
\sum_{R \in \mc{P}_{\leq N}} \frac{1-\text{Re}(f(R)\bar{\psi}_1(R)1_{\mc{S}^c}(R)e_{-\theta}(R))}{q^{\deg{R}}} &\geq \sum_{R \in \mc{P}_{\leq N}} \frac{1-\text{Re}(f(R)\bar{\psi}_1(R)e_{-\theta}(R))}{q^{\deg{R}}} - 2\sum_{\ss{R \in \mc{P} \\ P_1 \leq \deg{R} \leq Q_1}} q^{-\deg{R}} \\
&= \sum_{R \in \mc{P}_{\leq N}} \frac{1-\text{Re}(f(R)\bar{\psi}_1(R)e_{-\theta}(R))}{q^{\deg{R}}} - 2\log(Q_1/P_1) + O(1).
\end{align*}
By the triangle inequality and the assumption $\deg{g} \leq W$, we thus have
\begin{align*}
&|\Sigma_{\mc{S}}(\alpha)|\\
&\ll \frac{q^{\frac{W}{2}} }{|\mc{M}_{N-d}|} \sum_{G_0'' \in \mc{M}_{N-d}} \Big|\frac{1}{|\mc{M}_{<X+\gamma-d}|}\sum_{\substack{G \in \mc{M}_{N-d} \\ G \in I_{X+\gamma-d}(G_0'')}} f1_{\mc{S}}\bar{\psi}_1(G)  - \frac{1}{|\mc{M}_{N-d}|} \sum_{G \in \mc{M}_{N-d}} f1_{\mc{S}}(G)\bar{\psi}_1(G) \Big| \\
&+ q^{\frac{W}{2}}\Big(\mc{D}_{f\bar{\psi}_1}(N)e^{-\mc{D}_{f\bar{\psi}_1}(N)} + \frac{P_1}{Q_1}\min\left\{1,(Q_1/P_1)^3(1+\mc{D}_{f\bar{\psi}_1}(N))e^{-\mc{D}_{f\bar{\psi}_1}(N)}\right\}\Big).
\end{align*}
Applying Theorem~\ref{MRFF} (in the form given in Remarks~\ref{RemOnS2} and~\ref{mrremark}, bounding the long sum using Theorem~\ref{HalThmFF}), the first expression above is 
$$
\ll q^{W/2}\Big(Q_1^{\frac{1}{2}}q^{-\frac{1}{12}P_1} + N^{-1/36+o(1)} + 1_{\substack{f\bar{\psi}_1 \text{ not real} \\ \text{or $2|q$}}} 1_{\chi_1 \neq \chi_0} (1+\mc{D}_{(f\bar{\psi}_1)^{\ast}\bar{\chi}_1}(N))e^{-\mc{D}_{(f\bar{\psi}_1)^{\ast}\bar{\chi}_1}(N)}\Big),
$$
where $\chi_1$ denotes the character modulo $t^{N-X-\gamma+1}$ such that $\chi \mapsto \mc{D}_{(f\bar{\psi}_1)^{\ast}\bar{\chi}}(N)$ is minimal (and $\chi_0$ is the principal character to the same modulus). Recalling that $P_1 = 100W$ and $Q_1 = H/3$, it follows that
\begin{align}
&\max_{\alpha \in \mc{M}(X,W)} |\Sigma_{\mc{S}}(\alpha)|\ll q^{W/2}(H^{1/2}q^{-8W}+N^{-1/36+o(1)} \nonumber \\
&+ \max_{\substack{M \in \mc{M}_{\leq W} \\ \psi \bmod{M}}}\Big((1+\mc{D}_{f\bar{\psi}}(N))e^{-\mc{D}_{f\bar{\psi}}(N)}
+(P_1/Q_1) \min\{1,(Q_1/P_1)^3 (1+\mc{D}_{f\bar{\psi}}(N))e^{-\mc{D}_{f\bar{\psi}}(N)}\})\Big)\label{MajArcNoHayes}\\
&+ q^{W/2}\max_{M \in \mc{M}_{\leq W}}\max_{\substack{\psi \bmod{M} \\ f\bar{\psi} \text{ not real} \\ \text{ if $2 \nmid q$}}}\max_{\substack{X < j \leq H \\ \chi \neq \chi_0 \bmod{t^{N-j+1}}}}(1+\mc{D}_{(f\bar{\psi})^{\ast}\bar{\chi}}(N))e^{-\mc{D}_{(f\bar{\psi})^{\ast}\bar{\chi}}(N)} \label{MajArc}.
\end{align}

In order to estimate this quantity further and to prove Proposition~\ref{expprop}, we split the remainder of the analysis into two cases.

\subsection*{Case 1: \texorpdfstring{$f$}{f} is not real-valued or \texorpdfstring{$q$}{q} is even}
By Corollary~\ref{StarNonPret}, we have
$$
M_{\text{Hayes}}(f;N,H) \leq \min_{M \in \mc{M}_{\leq X}} \min_{\substack{\psi \bmod{M} \\ f\bar{\psi} \text{ not real}}} \min_{X < j \leq H} \min_{\substack{\chi \bmod{t^{N-j+1}} \\ \chi \neq \chi_0}} \mc{D}_{(f\bar{\psi})^{\ast} \bar{\chi}}(N) + O(1).
$$
Of course, we also have
$$
M_{\text{Hayes}}(f;N,H) \leq M_{\text{Dir}}(f;N,H) \leq \min_{M \in \mc{M}_{\leq X}} \min_{\psi \bmod{M}} \mc{D}_{f\bar{\psi}}(N).
$$
Inserting these bounds into~\eqref{MajArc} and using $10 \log H \leq W \leq (\log N)/(100 \log q)$, we get
\begin{align*}
&\max_{\alpha \in \mc{M}(X,W)} |\Sigma_{\mc{S}}(\alpha)|\\
&\ll H^{-50} + N^{-1/40} \\
&+ q^{W/2}\Big(e^{-M_{\text{Dir}}(f;N,H)/2}+(P_1/Q_1) \min\{1,(Q_1/P_1)^3 e^{-M_{\text{Dir}}(f;N,H)/2}\} + e^{-M_{\text{Hayes}}(f;N,H)/2}\Big).
\end{align*}
By assumption, we also have $W \leq M_{\text{Hayes}}(f;N,H)/10 \leq M_{\text{Dir}}(f;N,H)/10$. Furthermore, if $Q_1/P_1 \geq e^{M_{\text{Dir}}(f;N,H)/5}$ then
$$
q^{W/2}P_1/Q_1 \leq e^{M_{\text{Dir}}(f;N,H)/20 - M_{\text{Dir}}(f;N,H)/5} \leq e^{-M_{\text{Dir}}(f;N,H)/10},
$$
whereas if $Q_1/P_1 < e^{M_{\text{Dir}}(f;N,H)/5}$ then
$$
q^{W/2} (Q_1/P_1)^2e^{-M_{\text{Dir}}(f;N,H)/2} < e^{M_{\text{Dir}}(f;N,H)(1/20 + 2/5 - 1/2)} = e^{-M_{\text{Dir}}(f;N,H)/20}.
$$
Thus, we deduce the bound
\begin{align*}
\max_{\alpha \in \mc{M}(X,W)} |\Sigma_{\mc{S}}(\alpha)| &\ll (\log H)H^{-1} + N^{-1/40} + e^{-M_{\text{Dir}}(f;N,H)/20} + e^{-M_{\text{Hayes}}(f;N,H)/2}q^{W/2} \\
&\ll H^{-50} + N^{-1/40} + e^{-M_{\text{Hayes}}(f;N,H)/20},
\end{align*}
since $10 \log H \leq W \leq \min\{H/10,(\log N)/(100\log q), M_{\text{Hayes}}(f;N,H)/10\}$.
\subsection*{Case 2: \texorpdfstring{$f$}{f} is real-valued and \texorpdfstring{$q$}{q} odd}
We claim that
\begin{equation} \label{eq_claim_rv}
\max_{M \in \mc{M}_{\leq W}}\max_{\substack{\psi \bmod{M} \\ f\bar{\psi} \text{ not real}}}\max_{\substack{X < j \leq H \\ \chi \neq \chi_0 \bmod{t^{N-j}}}} \mc{D}_{(f\bar{\psi})^{\ast}\bar{\chi}}(N)e^{-\mc{D}_{(f\psi)^{\ast}\bar{\chi}}(N)} \ll N^{-1/4+o(1)}.
\end{equation}
Inserting this into~\eqref{MajArc} and then repeating the arguments in Case 1 to simplify the terms in~\eqref{MajArcNoHayes}, we obtain
$$
\max_{\alpha \in \mc{M}(X,W)} |\Sigma_{\mc{S}}(\alpha)| \ll H^{-50} + N^{-1/40} + e^{-M_{\text{Dir}}(f;N,H)/20}.
$$
Let $\psi$ be a character of modulus $M$ with $\deg{M} \leq W \leq \log N$ for which $f\bar{\psi}$ is not real-valued. Since $f$ is real-valued it follows that $\psi$ is not, nor is $(f\bar{\psi})^{\ast}$. Put $M = \tilde{M} t^r$, where $(\tilde{M},t) = 1$, and write $\psi = \psi_{\tilde{M}} \psi_{t^r}$. We consider two subcases, depending on whether or not $\psi_{\tilde{M}}$ is real-valued. 

\textbf{Case 2.1.} Suppose first that $\psi_{\tilde{M}}$ is real. Since $\psi^2$ is non-principal, it follows that $\psi_{t^r}^2$ is non-principal. Applying the triangle inequality as in the proof of Lemma~\ref{PrincPret}, we can show that
$$
\mc{D}_{(f\bar{\psi})^{\ast}\bar{\chi}}(N) \geq \frac{1}{4}\mc{D}_{(\psi_{t^r}^2)^{\ast}\chi^2}(N).
$$
By Lemma~\ref{lem_short_inv}, $(\psi_{t^r}^2)^{\ast}$ is a non-principal short interval character of length $\leq r$, so that $(\psi_{t^r}^2)^{\ast}\chi^2$ is a non-principal Hayes character of conductor $\leq r + \cond{\chi^2} \leq W + N-H + 1 < N$. 
Lemma~\ref{lem_hayesdist} now implies that, for some $\theta_0 \in [0,1]$,
\begin{align*}
\mc{D}_{(\psi_{t^r}^2)^{\ast} \chi^2}(N) &= \log N - \text{Re}\Big(\sum_{d \leq N} \frac{e(-\theta_0d)}{dq^{d}} \sum_{G \in \mc{M}_d} (\psi_{t^r}^2)^{\ast}\chi^2(G) \Lambda(G)\Big) + O(1) \\
&= (1-o(1)) \log N.
\end{align*}
In particular, we find that
$$
\mc{D}_{(f\bar{\psi})^{\ast}\bar{\chi}}(N) \geq (1/4-o(1))\log N,
$$
which implies~\eqref{eq_claim_rv} in this case. 

\textbf{Case 2.2.} Next, suppose $\psi_{\tilde{M}}$ is not real, so that $\psi_{\tilde{M}}^2$ is non-principal. Without loss of generality, we may assume that $f$ is extended to $\mb{F}_q[t]$ by $f(c)\bar{\psi}(c)=\chi(c)$. By Lemma~\ref{lem_dist_inv}, we see that 
$$
\mc{D}_{(f\bar{\psi})^{\ast}\bar{\chi}}(N) \geq \mc{D}_{f\bar{\psi}_{\tilde{M}} \bar{\psi}_{t^r}  \bar{\chi}^{\ast}}(N) + O(1).
$$
Applying a similar argument as in the previous subcase, we have then that
$$
\mc{D}_{(f\bar{\psi})^{\ast}\bar{\chi}}(N) \geq \frac{1}{4} \mc{D}_{\psi_{\tilde{M}}^2 (\psi_{t^r})^2 (\chi^2)^{\ast}}(N) + O(1).
$$
Since $(\tilde{M},t) = 1$, $\psi_{\tilde{M}}^2 (\psi_{t^r})^2(\chi^2)^{\ast}$ is a non-principal Hayes character. Similarly as in the previous subcase, we obtain
$$
\mc{D}_{(f\bar{\psi})^{\ast}\chi}(N) \geq (1/4-o(1)) \log N.
$$
Thus,~\eqref{eq_claim_rv} is valid in this case as well, and thus in all cases in which $f$ is real-valued. This completes the proof of Proposition~\ref{expprop}.

\begin{proof}[Proof of Theorem~\ref{ShortExpSumFFCM}]
Let $1 \leq W, H' \leq N$, and put $P_1 := 100W$ and $Q_1 := H'/3$. If we assume the condition
\begin{equation} \label{eq_adm_W}
10\log H' \leq W \leq \min\{M_{\ast}(f; N, H')/10, (\log N)/(100 \log q), H'/10\},
\end{equation}
where we recall that $M_{\ast} = M_{\text{Hayes}}$ unless $f$ is real and $q$ is odd in which case $M_{\ast} = M_{\text{Dir}}$, then we have
\begin{align*}
&\max_{\alpha \in \mb{T}} \frac{1}{|\mc{M}_N|} \sum_{G_0 \in \mc{M}_N} \Big|\sum_{\substack{G \in \mc{M}_N \\ G \in I_{H'}(G_0)}} f(G) e_{\mb{F}}(G\alpha)\Big|\\
&\ll q^{H'}\Big((\log H')(H')^{-1} + P_1/Q_1 + N^{-1/40} + e^{-M_{\ast}(f; N,H')/20}\Big).
\end{align*}
Suppose now that $1 \leq H \leq N-N^{3/4}$, and define $1 \leq H_0 \leq N$ by
$$
\log H_0 := \min\{M_{\ast}(f;N,H)/100, (\log N)/(1000\log q), H/100\}.
$$
We will make a choice of $W$ that suits our current choice of $H$. \\
If $H \leq H_0$ then $W := 10 \log H$ is admissible in~\eqref{eq_adm_W} with $H' = H$, and Theorem~\ref{ShortExpSumFFCM} is verified in this case (here $P_1 \ll \log H$, so $P_1/Q_1 \ll (\log H)H^{-1}$). 

Next, suppose $H > H_0$. For each $G_0 \in \mc{M}_N$ we can split $I_H(G_0)$ into $\ll q^{H-H_0}$ short intervals $I_{H_0}(G_0 + t^{H_0}M)$, where $\deg{M} < H-H_0$. We then have
\begin{align*}
&\max_{\alpha \in \mb{T}} \frac{1}{|\mc{M}_N|} \sum_{G_0 \in \mc{M}_N} \Big|\sum_{\substack{G \in \mc{M}_N \\ G \in I_H(G_0)}} f(G) e_{\mb{F}}(G\alpha)\Big|\\ 
&\ll q^{H-H_0}\max_{\deg{M} < H-H_0} \max_{\alpha \in \mb{T}} \frac{1}{|\mc{M}_N|} \sum_{G_0 \in \mc{M}_N} \Big|\sum_{\substack{G \in \mc{M}_N \\ G \in I_{H_0}(G_0 + t^{H_0}M)}} f(G) e_{\mb{F}}(G\alpha)\Big| \\
&= q^{H-H_0} \max_{\alpha \in \mb{T}} \frac{1}{|\mc{M}_N|} \sum_{G_0' \in \mc{M}_N} \Big|\sum_{\substack{G \in \mc{M}_N \\ G \in I_{H_0}(G_0')}} f(G) e_{\mb{F}}(G\alpha)\Big|.
\end{align*}
We have thus reduced matters to the case $H = H_0$, which was addressed previously. Since
$H \mapsto M_{\ast}(f;N,H)$ is non-increasing, we see that
$$
\log H_0 \leq \min\{M_{\ast}(f;N,H_0)/100, (\log N)/(1000\log q), H_0/100\}
$$
when $N$ (and therefore $H$) is large enough. 
Selecting $W := 10 \log H_0$ gives an admissible choice in relation to~\eqref{eq_adm_W} with $H = H_0$, picking $P_1,Q_1$ in terms of $H_0$ so that $P_1/Q_1 \ll (\log H_0)/H_0$. We thus have
\begin{align*}
&\max_{\alpha \in \mb{T}} \frac{1}{|\mc{M}_N|} \sum_{G_0 \in \mc{M}_N} \Big|\sum_{\substack{G \in \mc{M}_N \\ G \in I_H(G_0)}} f(G) e_{\mb{F}}(G\alpha)\Big|\\
&\ll q^{H-H_0} \cdot q^{H_0} \Big((\log H_0)H_0^{-1} + N^{-1/40} + e^{-M_{\ast}(f; N, H_0)/20}\Big) \\
&\ll q^H \Big(N^{-1/(2000\log q)} + M_{\ast}(f;N,H) e^{-M_{\ast}(f;N,H)/100} \Big),
\end{align*}
using the fact that $H' \mapsto M_{\ast}(f;N,H')$ is non-increasing and the definition of $H_0$ in the last line. Theorem~\ref{ShortExpSumFFCM} then follows in this case as well. 
\end{proof}

\begin{proof}[Proof of Theorems~\ref{ShortExpSumFFComplex} and~\ref{ShortExpSumFF}]
We will only prove Theorem~\ref{ShortExpSumFFComplex} from the first statement in Theorem~\ref{ShortExpSumFFCM}, as Theorem~\ref{ShortExpSumFF} follows in the same way from the second statement in Theorem~\ref{ShortExpSumFFCM}.

Let $f \colon  \mc{M} \ra \mb{U}$ be a multiplicative function. Define $\tilde{f}$ to be the completely multiplicative function such that $\tilde{f}(P) = f(P)$ for all $P \in \mc{P}$. We may thus find a multiplicative function $h  \colon  \mc{M} \ra \mb{C}$, supported on squarefull monic polynomials (i.e., if $h(P^k) \neq 0$ for $P \in \mc{P}$ and $k \in \mb{N}$ then $k \geq 2$) such that $f = \tilde{f} \ast h$; in particular, $h$ is bounded by the divisor function $d(G) = \sum_{D|G} 1$, and hence $|h(G)| \ll_{\e} q^{\e \deg{G}}$ for any $G \in \mc{M}$, a fact we will use shortly. We thus have
\begin{align*}
&\sup_{\alpha \in \mb{T}} \frac{1}{|\mc{M}_N|} \sum_{G_0 \in \mc{M}_N} \frac{1}{|\mc{M}_{<H}|} \Big|\sum_{\substack{G \in \mc{M}_N \\ G \in I_H(G_0)}} f(G)e_{\mb{F}}(G\alpha)\Big| \\
&\leq \sum_{D \in \mc{M}_{\leq N}} |h(D)| \sup_{\alpha \in \mb{T}} \frac{1}{|\mc{M}_N|} \sum_{G_0 \in \mc{M}_N} \frac{1}{|\mc{M}_{<H}|} \Big|\sum_{\substack{G' \in \mc{M}_{N-\deg{D}} \\ DG' \in I_H(G_0)}} \tilde{f}(G')e_{\mb{F}}(G'D\alpha)\Big| \\
&=: \mc{T}_{< H} + \mc{T}_{\geq H}.
\end{align*}
We first estimate $\mc{T}_{\geq H}$, which corresponds to the terms with $\deg{D} \geq H$ above. If $\deg{D} \geq H$ then if $G_0 \in \mc{M}_N$ is such that $I_H(G_0) \cap D\mc{M}_{N-\deg{D}} \neq \emptyset$ then in fact $|I_H(G_0) \cap D\mc{M}_{N-\deg{D}}| = 1$ and $G_0$ lies in one of at most $\ll q^H$ residue classes modulo $D$. It follows that
\begin{align*}
\mc{T}_{\geq H} \ll \sum_{\substack{D \in \mc{M} \\ \deg{D} \geq H}} |h(D)| \cdot \frac{1}{|\mc{M}_N||\mc{M}_{<H}|} \cdot q^{N+H-\deg{D}} \ll \sum_{\substack{D \in \mc{M} \\ \deg{D} \geq H}} |h(D)|q^{-\deg{D}}.
\end{align*}
Since $h$ is supported on squarefull polynomials, all of which are of the form $A^2B^3$ for some $A,B \in \mc{M}$, and moreover $|h(D)| \ll q^{\deg{D}/5}$ for all $D$, we obtain
\begin{equation}\label{eq_sqfull}
\mc{T}_{\geq H} \ll \sum_{\substack{D \in \mc{M} \\ \deg{D} \geq H}} |h(D)|q^{-\deg{D}} \ll q^{-H/5} \sum_{A^2B^3 \in \mc{M}} q^{-0.6(2\deg{A} + 3\deg{B})} \ll q^{-H/5}.
\end{equation}
Next, we estimate $\mc{T}_{< H}$. Writing $G_0 = DG_0'+B$ for some $B \bmod{D}$ and $D \in \mc{M}_{<H}$, we have
\begin{align*}
\mc{T}_{<H} = \sum_{D \in \mc{M}_{<H}} |h(D)| \sup_{\alpha \in \mb{T}} \frac{1}{|\mc{M}_N|} \sum_{B \bmod{D}} \sum_{G_0' \in \mc{M}_{N-\deg{D}}} \frac{1}{|\mc{M}_{<H}|}\Big|\sum_{\substack{G' \in \mc{M}_{N-\deg{D}} \\ DG' \in I_H(DG_0'+B)}} \tilde{f}(G')e_{\mb{F}}(G'D\alpha)\Big|.
\end{align*}
Since $\deg{B} < \deg{D} < H$, we see that $I_H(DG_0'+B) = I_H(DG_0')$ for all $B \bmod{D}$. Moreover, we also have that $DG' \in I_H(DG_0')$ if and only if $G' \in I_{H-\deg{D}}(G_0')$. Thus, 
\begin{align*}
&\mc{T}_{<H} = \sum_{D \in \mc{M}_{<H}} |h(D)| \sup_{\alpha \in \mb{T}}  \frac{q^{\deg{D}}}{|\mc{M}_N|}\sum_{G_0' \in \mc{M}_{N-\deg{D}}} \frac{1}{|\mc{M}_{<H}|}\Big|\sum_{\substack{G' \in \mc{M}_{N-\deg{D}} \\ G' \in I_{H-\deg{D}}(G_0')}} \tilde{f}(G')e_{\mb{F}}(G'D\alpha)\Big| \\
&\ll \sum_{D \in \mc{M}_{<H}} \frac{|h(D)|}{q^{\deg{D}}} \sup_{\alpha \in \mb{T}} \frac{1}{|\mc{M}_{N-\deg{D}}|} \sum_{G_0' \in \mc{M}_{N-\deg{D}}} \frac{1}{|\mc{M}_{<H-\deg{D}}|}\Big|\sum_{\substack{G' \in \mc{M}_{N-\deg{D}} \\ G' \in I_{H-\deg{D}}(G_0')}} \tilde{f}(G')e_{\mb{F}}(G'D\alpha)\Big| \\
&\leq \sum_{D \in \mc{M}_{<H}} \frac{|h(D)|}{q^{\deg{D}}} \sup_{\beta \in \mb{T}} \frac{1}{|\mc{M}_{N-\deg{D}}|} \sum_{G_0' \in \mc{M}_{N-\deg{D}}} \frac{1}{|\mc{M}_{<H-\deg{D}}|}\Big|\sum_{\substack{G' \in \mc{M}_{N-\deg{D}}\\ G' \in I_{H-\deg{D}}(G_0')}} \tilde{f}(G')e_{\mb{F}}(G'\beta)\Big|.
\end{align*}
Since the supremum over $\beta$ is $\leq 1$ for all $D \in \mc{M}_{<H}$, we may further bound the contribution from $\deg{D} \geq H/2$ (as in~\eqref{eq_sqfull}, with $H$ replaced by $H/2$) by $O(q^{-H/10})$. 
Applying Theorem~\ref{ShortExpSumFFCM} for each $D \in \mc{M}_{<H/2}$, we find
\begin{align}\label{eq_lastEst}\begin{split}
\mc{T}_{<H}&\ll q^{-H/10} + \sum_{D \in \mc{M}_{<H/2}} |h(D)|q^{-\deg{D}} \Big(\frac{\log(H-\deg{D})}{H-\deg{D}}\\
&+ (N-\deg{D})^{-1/(2000\log q)} + M_De^{-M_D/100}\Big),
\end{split}
\end{align}
where we have set $M_D := M_{\text{Hayes}}(\tilde{f};N-\deg{D},H-\deg{D})+1$. We note from its definition that $M_{\text{Hayes}}$ is non-increasing in $H$, and since $\tilde{f}$ takes the same values as $f$ on primes we get
$$
M_D \geq M_{\text{Hayes}}(\tilde{f};N-\deg{D},H) = M_{\text{Hayes}}(f;N-\deg{D},H).
$$
Finally, as $\mc{D}_g(N-\deg{D}) \geq \mc{D}_g(N/2) = \mc{D}_g(N) - O(1)$ for any 1-bounded function $g \colon  \mc{M} \ra \mb{U}$, it follows that
$$
M_D \geq M_{\text{Hayes}}(f;N,H) - O(1) =: M - 1- O(1)
$$
for all $D \in \mc{M}_{<H/2}$. Invoking this in~\eqref{eq_lastEst}, we obtain the bound
\begin{align*}
\mc{T}_{<H} &\ll q^{-H/10} + \Big(\frac{\log H}{H} + N^{-1/(2000\log q)} + Me^{-M/100}\Big)\sum_{D \in \mc{M}} |h(D)|q^{-\deg{D}}\\
&\ll \frac{\log H}{H} + N^{-1/(2000\log q)} + Me^{-M/100}.
\end{align*}
Combining this with our earlier estimate for $\mc{T}_{\geq H}$, the proof of Theorem~\ref{ShortExpSumFFComplex} follows.
\end{proof}

\section{Elliott's Conjecture} \label{LogEllSec}

In this section, we shall prove the two-point case of the logarithmically averaged Elliott's conjecture on correlations of non-pretentious multiplicative functions in function fields, Theorem~\ref{LogEllFF1}. Here, we only treat the case $A = 1$ for simplicity; the proof of the general case of fixed monic $A$, which is essentially the same, is left to the interested reader.

In the sequel, we will adopt the following notational conventions: if $S \subset \mb{F}_q[t]$ and $g\colon \mb{F}_q[t] \to \mb{C}$ then
\begin{align*}
\mb{E}_{G \in S} g(G) &:= |S|^{-1}\sum_{G \in S} g(G),\\
\mb{E}_{G \in S}^{\log} g(G) &:= \Big(\sum_{G \in S} q^{-\deg{G}}\Big)^{-1} \sum_{G \in S} g(G)q^{-\deg{G}}\quad \textnormal{if}\quad  0 \not\in S.
\end{align*}

To prove Theorem~\ref{LogEllFF1}, we will combine the exponential sum estimate of Theorem~\ref{ShortExpSumFF} with a function field version of the entropy decrement argument that Tao developed in~\cite{tao} for the corresponding problem in the integer setting. The key proposition arising from this is the following.

\begin{prop}[Introducing an extra averaging variable]\label{Prop1}
Let $N \geq 100$, and let $B \in \mb{F}_q[t]\bk\{0\}$ be fixed. For any $1\leq K\leq \log \log \log N$, there exists $H \in [K,\exp(\exp(10K))]$ such that the following is true.  Suppose that $f_1,f_2$ satisfy the hypotheses of Theorem~\ref{LogEllFF1}. For each $R \in \mc{P}_H$ set $c_R := \bar{f_1(R)f_2(R)}$. Then
\begin{align*}
&\mb{E}^{\log}_{G \in \mc{M}_{\leq N}} q^{-\deg{G}} f_1(G)f_2(G+B) \\
&= \mb{E}_{P \in \mc{P}_H}  c_P \mb{E}^{\log}_{G \in \mc{M}_{\leq N}} f_1(G)f_2(G+PB) + O(K^{-0.1}).
\end{align*}
\end{prop}

Proposition~\ref{Prop1} will be deduced from the following proposition, which is based on our version of the entropy decrement argument.

\begin{prop}[Entropy decrement argument in function fields] \label{EntDec} Let $k\geq 1$, and let $a_1,\ldots, a_k \colon \mathbb{F}_q[t]\to \mathbb{U}$ be arbitrary $1$-bounded functions. Also let $B_1,\ldots, B_k\in \mathbb{F}_q[t]$ be any fixed polynomials. Then for any large enough $N$ and for $1\leq K\leq \log \log \log N$ there exists $H \in [K,\exp(\exp(10K))]$ such that 
\begin{align*}
\mathbb{E}_{G\in \mathcal{M}_{\leq N}}^{\log}|\mathbb{E}_{P\in \mathcal{P}_H}a_1(G+PB_1)\cdots a_k(G+PB_k)(q^{\deg{P}}1_{P\mid G}-1)|\ll K^{-0.1}.
\end{align*}
\end{prop}

\begin{proof}[Proof of Proposition~\ref{Prop1} assuming Proposition~\ref{EntDec}] We may assume that $K$ is larger than any fixed constant, otherwise the claim of Proposition~\ref{Prop1} (with a suitably large implicit constant in the error term). 

Thus, let $H\in [K,\exp(\exp(10K))]$, which may be assumed to be sufficiently large. By multiplicativity, for each $P \in \mc{P}_H$ we have
$$f_1(G)f_2(G+B) = c_P f_1(GP)f_2(PG+BP),$$
unless $P|G$ or $P|(G+B)$. Averaging over $P \in \mc{P}_H$, for a suitable choice of $H$ we have
\begin{align*}
&\mb{E}_{G \in \mc{M}_{\leq N}}^{\log} f_1(G)f_2(G+B) \\
&= \mb{E}_{P \in \mc{P}_H} c_P \mb{E}_{G \in \mc{M}_{\leq N}}^{\log} f_1(G)f_2(G+PB) q^{\deg{P}}1_{P|G} + O(q^{-H} + N^{-1}),
\end{align*}
since (accounting for $G = -B$ in case this is monic)
\begin{align*}
\mb{E}_{P \in \mc{P}_H} \mb{E}_{G \in \mc{M}_{\leq N}}^{\log} 1_{G\equiv 0\,\, \textnormal{or}\,\, -B\bmod P}  \ll q^{-H} + N^{-1}.
\end{align*}
By Proposition~\ref{EntDec} with $a_i=f_i$ and the triangle inequality, we have
\begin{align*}
&\mb{E}_{P \in \mc{P}_H} c_P\mb{E}_{G \in \mc{M}_{\leq N}}^{\log} f_1(G)f_2(G+PB)  q^{\deg{P}}1_{P|G} \\
&= \mb{E}_{P \in \mc{P}_H}c_P \mb{E}_{G \in \mc{M}_{\leq N}}^{\log} f_1(G)f_2(G+PB) + O(K^{-0.1}),
\end{align*}
and the claim follows.
\end{proof}

In the next subsection, we will establish Proposition~\ref{EntDec}.

\subsection{The Entropy Decrement Argument in Function Fields}
\begin{def1}\label{defn1} Let $\mbf{X},\mbf{Y}$ be random variables on a probability space $(\Omega,\mathbb{P})$ with finite ranges $\mc{X},\mc{Y}$, respectively. We define the entropy
\begin{align*}
\He(\mbf{X}) := \sum_{x\in \mc{X}}\mathbb{P}(\mbf{X}=x)\log \frac{1}{\mathbb{P}(\mbf{X}=x)}
\end{align*}
and the joint entropy
\begin{align*}
\He(\mbf{X},\mbf{Y}) := \sum_{x\in \mc{X},y\in \mc{Y}}\mathbb{P}(\mbf{X}=x,\mbf{Y}=y)\log \frac{1}{\mathbb{P}(\mbf{X}=x,\mbf{Y}=y)}.
\end{align*}
Let $E\subset \Omega$. We define the conditional entropy of $\mbf{X}$ with respect to the event $E$ by
\begin{align*}
\He(\mbf{X}|E)=\sum_{x\in \mc{X}}\mathbb{P}(\mbf{X}=x\mid E)\log \frac{1}{\mathbb{P}(\mbf{X}=x\mid E)}
\end{align*}
and further define the conditional entropy of $\mbf{X}$ given $\mbf{Y}$ by
\begin{align*}
\He(\mbf{X}|\mbf{Y})=\sum_{y\in \mc{Y}}\He(\mbf{X}|\mbf{Y}=y)\mathbb{P}(\mbf{Y}=y)
\end{align*}
Note that this satisfies the identity
$$
\He(\mbf{X},\mbf{Y}) = \He(\mbf{X}|\mbf{Y}) + \He(\mbf{Y}).
$$
Finally, we define the mutual information between $\mbf{X}$ and $\mbf{Y}$ by
\begin{align*}
\mathbb{I}(\mbf{X},\mbf{Y}):=\He(\mbf{X})+\He(\mbf{Y})-\He(\mbf{X},\mbf{Y})
\end{align*}

\end{def1}
The nonnegativity of $\mb{I}(\mbf{X},\mbf{Y})$ follows from the following lemma.

\begin{lem}[Shannon inequalities]\label{le_shannon} Let $\mbf{X},\mbf{Y}$ be random variables on a probability space $(\Omega,\mathbb{P})$ with finite ranges $\mc{X},\mc{Y}$. Then we have the bounds
\begin{align*}
0\leq \He(\mbf{X})\leq \log |\mc{X}|
\end{align*}
and
\begin{align*}
\He(\mbf{X})\leq \He(\mbf{X},\mbf{Y})\leq \He(\mbf{X})+\He(\mbf{Y}).
\end{align*}
\end{lem}

\begin{proof} These inequalities are proved by applying Jensen's inequality to the concave function $x\mapsto x\log \frac{1}{x}$; see~\cite{billingsley} for the details.
\end{proof}

\begin{proof}[Proof of Proposition~\ref{EntDec}] We may assume that $K$ (and thus $N$) are sufficiently large, since otherwise the claim of the proposition is trivial. We adapt Tao's proof in~\cite{tao} to the function field setting. Let $\varepsilon=K^{-0.1}$. It suffices to show that there exists $H\in [K, \exp(\exp(10K))]$ for which
\begin{align*}
|\mathbb{E}_{G\in \mathcal{M}_{\leq N}}^{\log}\mathbb{E}_{P\in \mathcal{P}_H}c_Pa_1(G+PB_1)\cdots a_k(G+PB_k)(q^{\deg{P}}1_{P\mid G}-1)|\ll \varepsilon
\end{align*}
uniformly for all choices of $c_P\in \mathbb{U}$. We discretize the functions $a_i$ by defining $\tilde{a}_i(F)$ for each $F\in \mathbb{F}_q[t]$ to be $a_i(F)$ rounded to the nearest element in the Gaussian lattice $\varepsilon \mathbb{Z}[i]$, breaking ties using the lexicographic ordering, say. Then it suffices to prove 
\begin{align}\label{eqq1}
|\mathbb{E}_{P\in \mathcal{P}_H}\mathbb{E}^{\log}_{G\in \mathcal{M}_{\leq N}}c_P\tilde{a}_1(G+PB_1)\cdots \tilde{a}_k(G+PB_k)(q^{\deg{P}}1_{P\mid G}-1)|\ll \varepsilon
\end{align}
for some $H$ as above and for any $c_P\in \mathbb{U}$. Since each polynomial $G\in \mathcal{M}_{\leq N}$ of degree $\geq H$ belongs to the same number of short intervals $I_H(G_0)$, where $G_0$ ranges through $\mathcal{M}_{\leq N}$, and $H/N \ll \e$, the left-hand side of~\eqref{eqq1} can be rewritten as
\begin{align}\label{eqq2}
|\mathbb{E}_{G_0\in \mathcal{M}_{\leq N}}^{\log}\mathbb{E}_{\substack{G\in \mathcal{M}_{\leq N}\\ G \in I_H(G_0)}}\mathbb{E}_{P\in \mathcal{P}_H}c_P\tilde{a}_1(G+PB_1)\cdots \tilde{a}_k(G+PB_k)(q^{\deg{P}}1_{P\mid G}-1)|+O(\varepsilon).
\end{align}
Let $(\Omega,\mathbb{P})$ be the probability space where $\Omega=\mathcal{M}_{\leq N}$ and $\mathbb{P}$ is the probability measure
\begin{align*}
\mathbb{P}(A):=\mathbb{E}_{G\in \mathcal{M}_{\leq N}}^{\log}1_{A}(G),
\end{align*} 
for any $A\subset \Omega$. Since each $c_P$ and each map $\tilde{a}_i$ is uniformly bounded in absolute value,~\eqref{eqq2} can be bounded trivially by
\begin{align}\label{eqq3}
\ll \varepsilon+\mathbb{P}(G_0\in \mathcal{M}_{\leq N} \colon  |\mathbb{E}_{\substack{G\in \mathcal{M}_{\leq N}\\ G \in I_H(G_0)}}\mathbb{E}_{P\in \mathcal{P}_H}c_P\tilde{a}_1(G+PB_1)\cdots \tilde{a}_k(G+PB_k)(q^{\deg{P}}1_{P\mid G}-1)|\geq \varepsilon).
\end{align} 
Let $b:=\max_{j\leq k}\deg{B_j}$. Introduce the random variables $\mbf{X}_H$ and $\mbf{Y}_H$ defined on $\Omega$ and given by
\begin{align*}
\mbf{X}_H(G_0):=(\tilde{a}_1(F),\ldots, \tilde{a}_k(F))_{F \in I_{H+b}(G_0)},\quad \mbf{Y}_H(G_0):=(G_0\bmod P)_{P\in \mathcal{P}_{H}},\quad G_0\in \Omega.
\end{align*}
Then there is a deterministic function $\mathcal{F}$ such that we can write the probability in~\eqref{eqq3} as 
\begin{align*}
\mathbb{P}(G_0\in \mc{M}_{\leq N} \colon \,\,|\mathcal{F}(\mbf{X}_H(G_0),\mbf{Y}_H(G_0))|\geq \varepsilon);
\end{align*}
more precisely, $\mathcal{F}$ is of the form
\begin{align}\label{eqq5}
\mathcal{F}(x,y)=\mathbb{E}_{P\in \mc{P}_H}c_P\mc{Z}_P(x,y):=\mathbb{E}_{P\in \mc{P}_H}c_P \mathbb{E}_{\deg{J} < H+b}\phi_{J}(x,P)(q^{\deg{P}}1_{P\mid y+J}-1)
\end{align}
for some $1$-bounded functions $\phi_J$ and for $x\in \mathcal{X}_H, y\in \mathcal{Y}_H$, where $\mathcal{X}_H,\mathcal{Y}_H$ are the ranges of $\mbf{X}_H,\mbf{Y}_H$, respectively. Therefore, by the triangle inequality we have the bound $|\mathcal{Z}_P(x,y)|\leq 2$ for all $x\in \mathcal{X}_H, y\in \mathcal{Y}_H$.

It suffices to show that $\mb{P}(|\mc{F}(\mbf{X}_H,\mbf{Y}_H)| \geq \e) \ll \e$, for some $H \in [K, \exp(\exp(10K))]$. To do this, we start by bounding the probabilities $\mathbb{P}(|\mc{F}(x,\mbf{Y}_H)|\geq \varepsilon)$ without conditioning and then we  will deduce a bound on the corresponding conditional probabilities with $\mbf{X}_H = x$. 

By the Chinese remainder theorem, $\mbf{Y}_H(F)=y$ for any $y\in \mc{Y}_H$ corresponds to a unique congruence for $F$ modulo $\prod_{P\in \mc{P}_H} P$. Thus, this happens with probability exactly equal to $q^{-\sum_{P\in \mc{P}_H}\deg{P}}$ as long as 
\begin{align*}
\sum_{P\in \mc{P}_H} \deg{P}<N,
\end{align*}
which by the prime polynomial theorem holds whenever $H\leq\frac{\log N}{4\log q}$ for $N$ large enough, say. Hence, $\mbf{Y}_H$ is a uniform random variable on $\mathcal{Y}_H$ under the aforementioned condition. In particular, all the random variables $G_0 \mapsto G_0\bmod P$ for $P\in \mc{P}_H$ are jointly independent of each other. By~\eqref{eqq5}, we may write
\begin{align*}
\mathcal{F}(x,\mbf{Y}_H)=\mathbb{E}_{P\in \mc{P}_H}c_P\mc{Z}_P(x,\mbf{Y}_H),
\end{align*}
and the random variables $\{\mc{Z}_P(x, \mbf{Y}_H)  \colon  P \in \mc{P}_H\}$ are jointly independent, all having mean $0$. Moreover, the number of different $P$ here is $\geq \frac{1}{2}q^{H}/H$, say, again by the prime polynomial theorem. By Hoeffding's inequality~\cite{hoeffding}, there is an absolute constant $C>0$ such that
\begin{align} \label{eq_hoeff}
\mathbb{P}(|\mc{F}(x,\mbf{Y}_H)|\geq \varepsilon)=\mathbb{P}(|\mathbb{E}_{P\in \mc{P}_H}c_P \mc{Z}_P(x,\mbf{Y}_H)|\geq \varepsilon)\leq \exp(-C\varepsilon^2 q^{H}/H)
\end{align}
for any $x\in \mc{X}_H$.

To bound the conditional probability $\mathbb{P}(|\mc{F}(x,\mbf{Y}_H)|\geq \varepsilon|\mbf{X}_H=x)$, we use a Pinsker-type inequality from~\cite{tao-ter}. This is applicable since $\mbf{Y}_H$ is a uniform random variable. We get
\begin{align*}
\mathbb{P}(|\mc{F}(x,\mbf{Y}_H)|\geq \varepsilon|\mbf{X}_H=x)\leq \frac{\He(\mbf{Y}_H)-\He(\mbf{Y}_H|\mbf{X}_H=x)+\log 2}{\log \frac{1}{\mathbb{P}(|\mc{F}(x,\mbf{Y}_H)|\geq \varepsilon)}}.
\end{align*} 
Since $H\geq K$ and $K$ is large, we may bound this from above using~\eqref{eq_hoeff} and the prime polynomial theorem, obtaining
\begin{align} \label{eqq4}
\leq \varepsilon+C^{-1}\varepsilon^{-2}\frac{\He(\mbf{Y}_H)-\He(\mbf{Y}_H|\mbf{X}_H=x)}{q^{H}/H}.
\end{align}
Recalling that
\begin{align*}
\mb{P}(|\mc{F}(\mbf{X}_H,\mbf{Y}_H)| \geq \e) = \sum_{x\in \mathcal{X}_H}\mathbb{P}(|\mc{F}(x,\mbf{Y}_H)|\geq \varepsilon|\mbf{X}_H=x)\mathbb{P}(\mbf{X}_H=x),
\end{align*}
we multiply the bound in~\eqref{eqq4} by $\mathbb{P}(\mbf{X}_H=x)$ and sum over $x\in \mathcal{X}_H$ to get 
\begin{align*}
\mathbb{P}(|\mc{F}(\mbf{X}_H,\mbf{Y}_H)|\geq \varepsilon)\leq \varepsilon+C^{-1}\varepsilon^{-2}\frac{\He(\mbf{Y}_H)-\He(\mbf{Y}_H|\mbf{X}_H)}{q^{H}/H}=\varepsilon+C^{-1}\varepsilon^{-2}\frac{\mathbb{I}(\mbf{X}_H,\mbf{Y}_H)}{q^{H}/H}
\end{align*}
by the definition of mutual information $\mathbb{I}(\mbf{X}_H,\mbf{Y}_H)$ from Definition~\ref{defn1}. Now what remains to be shown is that
\begin{align}\label{eqq6}
\mathbb{I}(\mbf{X}_H,\mbf{Y}_H)\leq \varepsilon^3\frac{q^H}{H}
\end{align}
holds for some $H$ satisfying the conditions in Proposition~\ref{EntDec}. We will prove~\eqref{eqq6} by appealing to Shannon's inequality (Lemma~\ref{le_shannon}) and pigeonholing in the parameter $H$.

Consider the conditional entropy
\begin{align*} 
\He(\mbf{X}_{H+j}|\mbf{Y}_H)
\end{align*}  
for $H,j\leq \frac{\log N}{4\log q}$, say. We may write
\begin{align*}
\mbf{X}_{H+j}=\bigotimes_{\deg{M}\leq j} \mbf{X}_{H}^{(M)},
\end{align*}
where each $\mbf{X}_{H}^{(M)}$ is a shifted copy of $\mbf{X}_H$ given by
\begin{align*}
\mbf{X}_H^{(M)}(G_0):=\mbf{X}_H(G_0+Mt^{H+b}).
\end{align*}
Define also 
\begin{align*}
\mbf{Y}_H^{(M)}(G_0):=\mbf{Y}_H(G_0+Mt^{H+b}).
\end{align*}
Then by Shannon's inequality
\begin{align*}
\He(\mbf{X}_{H+j}|\mbf{Y}_H)\leq \He(\mbf{X}_{H+j}, \mbf{Y}_H) \leq \sum_{\deg{M}\leq j} \He(\mbf{X}_{H}^{(M)}|\mbf{Y}_H).
\end{align*}
Since the sigma algebra given by $\mbf{Y}_H$ is shift-invariant and $\mathbb{P}$ is \emph{almost} shift-invariant in the sense that
\begin{align*}
\sup_{A\subset \mc{M}_{\leq N}}|\mathbb{P}(G\in A)-\mathbb{P}(G+J\in A)|\leq \deg{J}/N
\end{align*}
we obtain
\begin{align*}
\He(\mbf{X}_{H+j}|\mbf{Y}_H)&\leq \sum_{\deg{M}\leq j} \He(\mbf{X}_{H}^{(M)}|\mbf{Y}_H^{(M)}) + \e q^j\\
&\ll q^{j}\He(\mbf{X}_H\mid \mbf{Y}_H)+\varepsilon q^j\\
&=q^{j}\He(\mbf{X}_H,\mbf{Y}_H)-q^{j}\He(\mbf{Y}_H)+\varepsilon q^j\\
&=q^j\He(\mbf{X}_H)-q^j\mathbb{I}(\mbf{X}_H,\mbf{Y}_H) + \e q^j.
\end{align*}
On the other hand, from Shannon's inequality we also have the lower bound
\begin{align*}
\He(\mbf{X}_{H+j}|\mbf{Y}_H)=\He(\mbf{X}_{H+j},\mbf{Y}_H)-\He(\mbf{Y}_H)\geq \He(\mbf{X}_{H+j})-\He(\mbf{Y}_H).
\end{align*}
Comparing the upper and lower bounds for $\He(\mbf{X}_{H+j}|\mbf{Y}_H)$, we now have 
\begin{align}\label{eqq7}
\frac{\mathbb{I}({\mbf{X}}_{H},\mbf{Y}_H)}{q^H}\leq \frac{\He(\mbf{X}_H)}{q^H}-\frac{\He(\mbf{X}_{H+j})}{q^{H+j}}+\frac{\He(\mbf{Y}_H)+\varepsilon q^{j}}{q^{H+j}}.
\end{align}
Since $\mbf{Y}_H$ is a uniform random variable, we have $\He(\mbf{Y}_H)=\log |\mc{Y}_H|\leq 2q^{H}$ by the prime polynomial theorem. Since $\mbf{X}_H$ has $kq^{H+b}$ components, each taking values in $\varepsilon\mathbb{Z}[i]\cap \mathbb{U}$, we have $\He(\mbf{X}_H)\leq 10(\log \frac{1}{\varepsilon})kq^{H+b}\leq C_{k,b}\varepsilon^{-1} q^{H}$ for some $C_{k,b} > 0$. Now, if we denote $w_H:=\He(\mbf{X}_H)/q^H$, then from~\eqref{eqq7} we have the information bound
\begin{align*}
\frac{\mathbb{I}({\mbf{X}}_{H},\mbf{Y}_H)}{q^H}\leq w_H-w_{H+j}+\frac{\varepsilon}{q^{H}}.
\end{align*}
and $w_H\in [0,\varepsilon^{-1}C_{k,b}].$ Suppose that~\eqref{eqq6} failed for all $H\in [K, \exp(\exp(K/2))]$. Then we would have 
\begin{align}\label{eqq8a}
\frac{\varepsilon^3}{H}\leq w_H-w_{H+j}+\frac{\varepsilon}{q^{H}} + 2q^{-j}
\end{align}
for all $H\in [K,\exp(\exp(K/2))]$, $j\leq \frac{\log N}{4\log q}$. Define $H_1,H_2,\ldots$ recursively by $H_1=\lceil K\rceil$ and $H_{r+1} := H_r+2\log H_r+1000\log \frac{1}{\varepsilon}$. Then $H_r\leq \exp(\exp(K/2))$ for $r\leq \exp(\exp(K/3))$, say. Telescoping~\eqref{eqq8a} with $H=H_r$ and $j=H_{r+1}-H_r$ then yields
\begin{align*}
\sum_{r\leq \exp(\exp(K/3))}\frac{\varepsilon^3}{2H_r}\leq \sum_{r\leq \exp(\exp(K/3))}(w_{H_r}-w_{H_{r+1}}+\frac{\varepsilon}{q^{H_r}} + 2q^{H_r-H_{r+1}})\leq \varepsilon^{-1}C_{k,b}+1.
\end{align*}
Since, by telescoping, we have
$$
H_r = H_1 + 2\sum_{1 \leq j \leq r-1} \log H_j + 1000(r-1) \log(1/\e),
$$
by induction on $r$ we find that $H_r\leq C_0 (r\log r+r\log \frac{1}{\varepsilon})$ for some absolute constant $C_0>0$ whenever $r \geq K \geq 1000\log(1/\e)$. Therefore,
\begin{align*}
\varepsilon^{4}\sum_{K \leq r\leq \exp(\exp(K/3))}\frac{1}{10C_0 r\log r}\leq C_{k,b}+\varepsilon.
\end{align*}
However, given our choice $\e = K^{-0.1}$, the left-hand side is
\begin{align*}
\gg K^{-0.4}(K-O(\log K)),
\end{align*}
which is a contradiction for $K$ large enough. This completes the proof.
\end{proof}

Now that we have established Proposition~\ref{EntDec}, which relates one-variable correlations to two-variable ones, we can apply the circle method to complete the proof of Theorem~\ref{LogEllFF1}.

\begin{prop}\label{Prop2}
Assume the hypotheses of Proposition~\ref{Prop1} and let $H$ be chosen as in the conclusion of that proposition. Let $f_1,f_2 \colon \mathcal{M}\to \mathbb{U}$ be multiplicative functions. Set $H' := H + \deg{B}$. Then for any $\e > 0$,
\begin{align}\label{circlemethod}
&\mb{E}^{\log}_{G \in \mc{M}_{\leq N}} \Big|\mb{E}_{P \in \mc{P}_H} c_P \mb{E}_{\deg{J} < H'} f_1(G + J)f_2(G+J+PB)\Big|\\
&\ll \e^{-8}\Big((\log H)H^{-1} + N^{-1/(100\log q)} +e^{-M_{\text{Hayes}}(f_{1};N/H,H')/100}\Big) + \e^2.\nonumber 
\end{align}
Moreover, if $f_1$ is real-valued and $q$ is odd, we may replace $M_{\textnormal{Hayes}}$ with $M_{\textnormal{Dir}}$ in~\eqref{circlemethod}.
\end{prop}
\begin{proof}
Let $\mc{T}$ denote the expression on the left-hand side in~\eqref{circlemethod}. Fix $G \in \mc{M}_{\leq N}$ for the time being. For each $j = 1,2$, define the sequence $x_{j,J} := f_j(G+J)$ for all $\deg{J} < H + \deg{B}$. For each $G \in \mc{M}_{\leq N}$, consider the double sum
$$
\mc{T}_G := \mb{E}_{P \in \mc{P}_H} c_P \mb{E}_{\deg{J} < H'} x_{1,J}\bar{x}_{2,J+PB},
$$
noting that $\mc{T} = \mb{E}_{G \in \mc{M}_{\leq N}}^{\log} |\mc{T}_G|$. We may view the set of polynomials $J$ with $\deg{J} < H'$ as the representatives of residue classes modulo $t^{H'}$, and thus extending the sequences $\{x_{1,J}\}_J$ and $\{x_{2,J}\}_J$ periodically modulo $t^{H'}$, we can consider them as maps on $\mb{F}_q[t]/(t^{H'}\mb{F}_q[t])$. We may thus expand these sequences in the corresponding Fourier basis, giving in the inner sum over $J$:
\begin{align*}
&\mb{E}_{\deg{J} < H'} x_{1,J}\bar{x}_{2,J+PB}\\
&= \sum_{\xi_1,\xi_2 \bmod{t^{H'}}} \hat{x}_1(\xi_1)\bar{\hat{x}_2(\xi_2)} e_{\mb{F}}(-\xi_2 PB/t^{H'})\mb{E}_{J \bmod{t^{H'}}}e_{\mb{F}}\Big(\frac{J}{t^{H'}}(\xi_1-\xi_2)\Big)\\
&= \sum_{\xi \bmod{t^{H'}}} \hat{x}_1(\xi)\bar{\hat{x}_2(\xi)}e_{\mb{F}}(-\xi PB/t^{H'}),
\end{align*}
where we have defined
$$
\hat{x}_j(\xi) := q^{-H'} \sum_{J \bmod{t^{H'}}} x_{j,J}e_{\mb{F}}(-J\xi/t^{H'}).
$$
Inserting this into the definition of $\mc{T}_G$ thus gives
$$
\mc{T}_G = \sum_{\xi \bmod{t^{H'}}} \hat{x}_1(\xi)\bar{\hat{x}_2(\xi)} \cdot \mb{E}_{P \in \mc{P}_H} c_Pe_{\mb{F}}(-\xi PB/t^{H'}).
$$
Now, define the large spectrum set
$$
\Xi_H := \{\xi \bmod{t^{H'}}  \colon  \Big|\mb{E}_{P \in \mc{P}_H} c_P e_{\mb{F}}(-\xi PB/t^{H'})\Big| \geq \e^2\}.
$$
We decompose $\mc{T}_G = \mc{T}_{G,s} + \mc{T}_{G,l}$, where
\begin{align*}
\mc{T}_{G,s} &:= \sum_{\xi \notin \Xi_H} \hat{x}_1(\xi)\bar{\hat{x}_2(\xi)} \cdot \mb{E}_{P \in \mc{P}_H} c_P e_{\mb{F}}(-\xi PB/t^{H'})\\
\mc{T}_{G,l} &:= \sum_{\xi \in \Xi_H} \hat{x}_1(\xi)\bar{\hat{x}_2(\xi)} \cdot \mb{E}_{P \in \mc{P}_H} c_P e_{\mb{F}}(-\xi PB/t^{H'}).
\end{align*}
If $\xi \notin \Xi_H$ then we can bound the inner sum over $P$ by $\e^2$. It follows from the Cauchy--Schwarz inequality and Plancherel's theorem that 
\begin{align*}
|\mc{T}_{G,s}| &\ll \e^2 \sum_{\xi \notin \Xi_H}|\hat{x}_1(\xi)||\hat{x}_2(\xi)| \leq \e^2 \prod_{j = 1}^2 \Big(\sum_{\xi \bmod{t^{H'}}} |\hat{x}_j(\xi)|^2\Big)^{\frac{1}{2}} \\
&= \e^2 \prod_{j = 1}^2 \Big(q^{-H'}\sum_{J \bmod{t^{H'}}} |x_{j,J}|^2\Big)^{\frac{1}{2}} \ll \e^2.
\end{align*}
It remains to consider the case $\xi \in \Xi_H$. In this case, bounding the exponential sum in $P$ trivially, this contribution is
$$
|\mc{T}_{G,l}|\ll \sum_{\xi \in \Xi_H}|\hat{x}_1(\xi)||\hat{x}_2(\xi)|.
$$
Note that $\|\hat{x}_j\|_{\infty} \leq 1$ for $j = 1,2$. Averaging over $G \in \mc{M}_{\leq N}$ yields
\begin{align*}
\mc{T} &\ll \e^2 + \sum_{\xi \in \Xi_H} \mb{E}^{\log}_{G \in \mc{M}_{\leq N}} q^{-H'}\Big|\sum_{\deg{J} < H'} f_{1}(G+J) e_{\mb{F}}(-\xi J/t^{H'})\Big| \\
&\leq \e^2 +|\Xi_H|\max_{\alpha \in \mb{T}} \mb{E}^{\log}_{G_0 \in \mc{M}_{\leq N}} \Big|q^{-H'} \sum_{\substack{G \in \mc{M}_{\deg{G_0}} \\ G \in I_{H'}(G_0)}} f_{1}(G)e_{\mb{F}}(G\alpha)\Big|.
\end{align*}
To estimate $|\Xi_H|$, we use a 4th moment estimate. Indeed,
\begin{align*}
|\Xi_H| &\leq \frac{\e^{-8}}{|\mc{P}_H|^4}\sum_{\xi \bmod{t^{H'}}} \Big|\sum_{P \in \mc{P}_H} c_Pe(-PB\xi/t^{H'})\Big|^4 \\
&\ll \e^{-8}H^4q^{-4H}\Big|\sum_{P_1,P_2,P_3,P_4 \in \mc{P}_H} c_{P_1}c_{P_2}\bar{c_{P_3}c_{P_4}}\sum_{\xi \bmod{t^{H'}}} e_{\mb{F}}(-B(P_1+P_2-P_3-P_4)\xi/t^{H'})\Big| \\
&\ll \e^{-8} H^4 q^{-3H}\sum_{\substack{P_1,P_2,P_3,P_4 \in \mc{P}_H \\ P_1+P_2= P_3+P_4}} 1,
\end{align*}
since $\deg{B(P_1+P_2-P_3-P_4)} < H'$. By Lemma~\ref{PRIM4TUP}, the sum over tuples $(P_1,P_2,P_3,P_4)$ above is bounded by $O(q^{3H}/H^4)$, and hence $|\Xi_H| \ll \e^{-8}.$ Splitting the average in $G_0 \in \mc{M}_{\leq N}$ according to degree, we get
$$
\mc{T} \ll \e^2 + \e^{-8} \frac{1}{N}\sum_{k \leq N} \max_{\alpha \in \mb{T}} q^{-k}\sum_{G_0 \in \mc{M}_k} \Big|q^{-H'} \sum_{\substack{G \in \mc{M}_{k} \\ G\in I_{H'}(G_0)}} f_{j_0}(G)e_{\mb{F}}(G\alpha)\Big|.
$$The inner sum is trivially bounded as $\ll 1$ for $1 \leq k \leq N/H$, which contributes a term of size $\ll H^{-1}$. Since $H < N^{1/4}$, say, for each $N/H < k \leq N$ we may apply Theorem~\ref{ShortExpSumFFComplex} to get
$$
q^{-k}\sum_{G_0 \in \mc{M}_k}  \Big|q^{-H'}\sum_{\substack{G \in \mc{M}_k \\ G \in I_{H'}(G_0)}} f_{j_0}(G)e_{\mb{F}}(G\alpha)\Big| \ll (\log H)H^{-1} + N^{-1/(2000\log q)} + e^{-M_{\text{Hayes}}(f_{1}; N/H,H')/100}
$$
in this range. Averaging this estimate over $N/H < k \leq N$ gives
$$
\mc{T} \ll \e^2 + \e^{-8}\Big((\log H)H^{-1}+N^{-1/(2000\log q)} + e^{-M_{\text{Hayes}}(f_{1};N/H,H')/100} \Big).
$$
This implies the first claim.

The second claim is proved in an identical manner, except that at the end we appeal to Theorem~\ref{ShortExpSumFF}.

\end{proof}
\begin{proof}[Proof of Theorem~\ref{LogEllFF1}]
Let $W$ be fixed but large, and let $K=W/100$. Set
$$
\e := \min\{e^{-M_{\text{Hayes}}(f_{j_0};N/H,H+\deg{B})/1600}, K^{-0.1}\},
$$
where $H$ is chosen as in Proposition~\ref{Prop1}. Combining Propositions~\ref{Prop1} and~\ref{Prop2}, we find
\begin{align*}
\Big|\frac{1}{N}\sum_{G \in \mc{M}_{\leq N}} q^{-\deg{G}} f_1(G)f_2(G+B)\Big| &\ll K^{0.8} \Big(N^{-1/(2000\log q)} + (\log H)H^{-1}\Big) \\
&+ e^{-M_{\text{Hayes}}(f_{1};N/H,H+\deg{B})/200} + K^{-0.1},
\end{align*}
where $H \in [K,\exp(\exp(10 K)))]$ is chosen as in Proposition~\ref{Prop1}. Since $f_{1}$ is Hayes non-pretentious to level $W$ and $H + \deg{B} \leq 2H < (\log N)/(2\log q) \leq \log N$, it follows that $M_{\text{Hayes}}(f_{1};N/H,H+\deg{B}) \to \infty$ as $N \to \infty$. Since $H \geq K\geq W/100$ the above is $o_{W\to \infty}(1)$ as $N \to \infty$, and letting $W$ tend to infinity very slowly in terms of $N$, the first part of Theorem~\ref{LogEllFF1} follows.

Consider then the second part of the theorem, where $f_1$ is real-valued and $q$ is odd. Applying the same argument as before, save that $M_{\text{Hayes}}$ is replaced in every instance by $M_{\textnormal{Dir}}$, we see that~\eqref{correlation2} holds unless there exists an infinite sequence $N_j\to \infty$, Dirichlet characters $\psi_j \bmod{M_j}$ with $\deg{M_j}=O(1)$, and $\theta_j\in [0,1]$ such that 
\begin{align}\label{distance}
\mathbb{D}(f_1,\psi_j e_{\theta_j};N_j)= O(1).
\end{align}
If~\eqref{distance} holds, then by the pretentious triangle inequality also
\begin{align*}
\mathbb{D}(f_1^2,\psi_j^2 e_{2\theta_j};N_j)= O(1).
\end{align*}
By pigeonholing, we may assume that $\psi_j=\psi$ is independent of $j$. Moreover, by passing to a subsequence, we may assume that $\theta_j$ converges to some $\theta\in [0,1]$. 
Then
\begin{align}\label{distance3}
\mathbb{D}(f_1^2,\psi^2 e_{2\theta};N_j)= O(1),
\end{align}
since by~\eqref{distance} and the triangle inequality we  have
\begin{align*}
\mathbb{D}(e_{\theta},e_{\theta_j};N_j)=\limsup_{k\to \infty}\mathbb{D}(e_{\theta_{j+k}},e_{\theta_j};N_j)\leq \mathbb{D}(f_1\overline{\psi},e_{\theta_j};N_j)+\limsup_{k\to \infty}\mathbb{D}(f_1\overline{\psi},e_{\theta_{j+k}};N_{j+k})=O(1).     
\end{align*}

Assume first that $\mathbb{D}(1,f_1^2;\infty)<\infty$. Then by another application of the pretentious triangle inequality, we deduce that
\begin{align}\label{distance2}
 \mathbb{D}(1,\psi^2 e_{2\theta};N_j)= O(1).   
\end{align}
By Lemma~\ref{lem_hayesdist}, this implies that $\psi^2$ is principal, so we may assume that $\psi^2\equiv 1$ in~\eqref{distance2}. Then arguing as in~\cite[p. 15]{GrHaSoFF} we have
$$
\mb{D}(1,e_{2\theta};N_j)^2 = \log N_j - \sum_{n \leq N_j} \frac{\cos(4\pi \theta n)}{n} + O(1) = \log(\max\{N_j\|2\theta\|,1\}) + O(1),
$$
which in view of~\eqref{distance2} implies that $2\theta \equiv 0 \bmod{1}$. But this contradicts~\eqref{nonpret-b}. Hence, we must have $\mathbb{D}(1,f_1^2;\infty)=\infty$. But as $f_1^2$ is nonnegative, we have
\begin{align*}
\mathbb{D}(f_1^2,\psi^2e_{2\theta};N_j)^2 \geq \sum_{P \in \mc{P}_{\leq N_j}} \frac{1-|f_1(P)\bar{\psi}(P)e_{-\theta}(P)|^2}{q^{\deg{P}}} \geq \sum_{P \in \mc{P}_{\leq N_j}} \frac{1-f_1(P)^2}{q^{\deg{P}}} = \mathbb{D}(f_1^2,1;N_j)^2,
\end{align*}
so upon letting $j\to \infty$ this contradicts~\eqref{distance3}. The claim follows. 
\end{proof}

\section{A Conjecture of K\'atai in Function Fields} \label{KataiSec}

In this section, we establish Theorem~\ref{thm_katai} as an application of our two-point Elliott conjecture result (Theorem~\ref{LogEllFF1}). Since short interval characters and Archimedean characters satisfy
\begin{align}\label{shortinterval}
\xi(QG+1)=\xi(QG)=\xi(Q)\xi(G)\quad \textnormal{and}\quad e_{\theta}(QG+1)=e_{\theta}(Q)e_{\theta}(G)    
\end{align}
whenever $\deg{QG}$ is sufficiently large relative to $\text{len}(\xi)$, the function $f=\xi e_{\theta}$ clearly obeys~\eqref{katai} for suitably chosen $z\in S^1$. Thus, the essence of Theorem~\ref{thm_katai} lies in showing that there are no other such functions.

Before beginning with the proof of Theorem~\ref{thm_katai}, we state the following useful proposition.
\begin{prop}[Concentration inequality for multiplicative functions] \label{prop_concentration}
Let $f\colon\mc{M} \to \mathbb{U}$ be a multiplicative for which $\mb{D}(f,1;N) \ll 1$ as $N \ra \infty$, and let $\e > 0$. Then there is an infinite increasing sequence $\{M_j\}_{j\geq 1} \subset [1,\infty)$, depending only on $f$ and $\e$, for which the following holds: \\
Let $W\in \mc{M}$ satisfy $P\mid W$ for all $P\in \mc{P}_{\leq M_j}$. Then for any $B$ coprime to $W$ and of degree $<\deg{W}$, and for $N := M_{j+j'}$ with $j'$ large enough as a function of $M_j$ and $\e$, we have
\begin{align*}
\sum_{G\in \mc{M}_{\leq N}}|f(WG+B)-1|\ll q^N\left( \e + \mathbb{D}(f,1;M_j,\infty)+o_{j\to \infty}(1)\right). 
\end{align*}
\end{prop}

To prove this proposition we begin with the following general lemma. In the sequel, given scales $1 \leq A < B$ we define
$$
\mf{I}_f(A,B) := \sum_{\ss{P \in \mc{P} \\ A < \text{deg}(P) \leq B}} \frac{\text{Im}(f(P))}{q^{\deg{P}}}.
$$
\begin{lem}\label{lem_concentration} Let $N\geq M \geq 1$ and let $f \colon \mc{M}\to \mathbb{U}$ be multiplicative. Let $W\in \mc{M}$ satisfy $P\mid W$ for all $P\in \mc{P}_{\leq M}$. Then for any $B$ coprime to $W$ and of degree $<\deg{W}$, and for $N$ large enough as a function of $M$, we have
\begin{align*}
\sum_{G\in \mc{M}_{\leq N}}|f(WG+B)-e^{\mf{I}_f(M,N)}|\ll q^N\max_{\beta \in \{1,2\}} \mathbb{D}(f,1;M,\infty)^\beta+o_{M\to \infty}(q^N). 
\end{align*}
\end{lem}

\begin{proof}
Let $h  \colon  \mc{M}\to \mathbb{C}$ be the additive function given by $h(P^{\alpha})=f(P^{\alpha})-1$. Note that $\text{Re}(h(P^{\alpha})) \leq 0$, so that $e^{h(P^\alpha)} \in \mb{U}$ for all $P$ and $\alpha \geq 1$. 

We apply the Taylor approximation 
$$
z=e^{z-1}+O(|z-1|^2), \text{ for } |z|\leq 1
$$ 
with $z = f(P^\alpha) = 1+h(P^\alpha)$ for $P \in \mc{P}$ and $\alpha \geq 1$, together with the simple inequality
$$
|z_1\cdots z_k - w_1\cdots w_k| \leq \sum_{1 \leq j \leq k} |z_j-w_j|,
$$
valid whenever $z_j,w_j \in \mb{U}$ for all $1 \leq j \leq k$ (with $z_j$ and $w_j$ respectively playing the roles of $e^{h(P^\alpha)}$ and $f(P^\alpha)$ here). Ultimately, this yields
\begin{align*}
f(WG+B)
= e^{h(WG+B)} + O\left(\sum_{P^\alpha \mid \mid WG+B} |h(P^\alpha)|^2\right).
\end{align*}
Since $(B,W) = 1$, note that $P^{\alpha} \mid \mid WG+B \Rightarrow \text{deg}(P) > M$ and $P \nmid W$. Summing over $G \in \mc{M}_{\leq N}$ thus leads to
\begin{align}\label{eq:ftohBd}
\sum_{G \in \mc{M}_{\leq N}} \Big|f(WG+B) - e^{h(WG+B)}\Big| \ll q^N\sum_{\ss{P \in \mc{P} \\ \text{deg}(P) > M \\ \text{deg}(P^\alpha) \leq N}} |h(P^\alpha)|^2 q^{-\text{deg}(P^\alpha)}.
\end{align}
Next, set
$$
A_h(Y,X) := \sum_{\ss{P \in \mc{P} \\ Y < \text{deg}(P^\alpha) \leq X}} h(P^\alpha)q^{-\deg{P^\alpha}}(1-q^{-\deg{P}}), \quad X > Y \geq 1.
$$
Since $\text{Re}(h(M)) \leq 0$ for all $M \in \mc{M}$, we have $\text{Re}(A_h(M,N)) \leq 0$ as well, thus
$$
|e^{h(WG+B)} - e^{A_h(M,N)}| \ll |h(WG+B)-A_h(M,N)|.
$$
Summing this expression over $G \in \mc{M}_{\leq N}$, then applying the Cauchy-Schwarz inequality followed by the Tur\'{a}n-Kubilius inequality for $h$ (see~\cite[Lemma 7]{darbar} for the function field version of this\footnote{In~\cite{darbar}, the Tur\'an--Kubilius inequality was stated for the linear forms $G\mapsto G+B$, but the same proof works for any linear forms $G\mapsto WG+B$.}), we obtain
\begin{align*}
\sum_{G \in \mc{M}_{\leq N}} |e^{h(WG+B)} - e^{A_h(M,N)}| &\ll q^{N/2}\left(\sum_{G \in \mc{M}_{\leq N}} |h(WG+B)-A_h(M,N)|^2\right)^{1/2} \\
&\ll q^N\left(\sum_{\ss{P \in \mc{P} \\ \text{deg}(P) > M \\ \text{deg}(P^\alpha) \leq N}} |h(P^\alpha)|^2 q^{-\text{deg}(P^\alpha)}\right)^{1/2}.
\end{align*}
We note that
$$
\sum_{\ss{P \in \mc{P} \\ \text{deg}(P) > M \\ \text{deg}(P^\alpha) \leq N}} |h(P^\alpha)|^2 q^{-\text{deg}(P^\alpha)} = \sum_{\ss{P \in \mc{P} \\ M< \text{deg}(P) \leq N}} |1-f(P^\alpha)|^2 q^{-\text{deg}(P^\alpha)} + O(M^{-1/2}),
$$
and this simplifies to $2\mb{D}(f,1;M,N)^2 + O(M^{-1/2})$.
Combining this with~\eqref{eq:ftohBd}, we thus find that
\begin{align}\label{eq:concwithAh}
\sum_{G \in \mc{M}_{\leq N}} |f(WG+B) - e^{A_h(M,N)}| 
&\ll q^N\left(\mb{D}(f,1;M,\infty) + \mb{D}(f,1;M,\infty)^2 + M^{-1/4}\right).
\end{align}
Now, observe that
\begin{align*}
A_h(M,N) &= \sum_{\ss{P \in \mc{P} \\ M < \text{deg}(P) \leq N}} \frac{\text{Re}(f(P))-1}{q^{\text{deg}(P)}} + i\sum_{\ss{P \in \mc{P} \\ M < \text{deg}(P) \leq N}} \frac{\text{Im}(f(P))}{q^{\text{deg}(P)}} + O(M^{-1/2})  \\
&= -\mb{D}(f,1;M,N)^2 + i\mf{I}_f(M,N) + O(M^{-1/2}).
\end{align*}
When $M$ is large enough we thus have 
$$
e^{A_h(M,N)} = e^{i\mf{I}_f(M,N)} + O(\mb{D}(f,1;M,\infty)^2 + M^{-1/2}),
$$ 
which, when combined with~\eqref{eq:concwithAh} yields the claim.
\end{proof}

The following result allows us to pick suitable scales $M$ and $N$ in order to control the distribution of $\mf{I}_f(M,N) \bmod{2\pi}$, and therefore the direction of $e^{i\mf{I}_f(M,N)}$.  Below, as usual we write $\|t\| := \min_{m \in \mb{Z}} |t-m|$ for $t \in \mb{R}$.
\begin{lem}\label{lem:ctrlIF}
Let $\eta > 0$. Then there is an infinite increasing sequence $\{M_j\}_{j\geq 1} \subset [1,\infty)$ such that 
$$
\|\mf{I}_f(M_j,M_{j+j'})/2\pi \| < \eta
$$ 
for any choice of $j, j'$ sufficiently large relative to $\eta$.
\end{lem}
\begin{proof}
A proof of this claim appears implicitly in the proof of~\cite[Lemma 2.11]{klu_man_orb} in the integer setting. We give here a shorter proof in the function field setting that would also be applicable (without change) over the integers. 

Write $\mf{I}_f(T) := \mf{I}_f(1,T)$ for $T \geq 1$, and define $\mf{I}_f(\infty) := \lim_{T \ra \infty} \mf{I}_f(T)$. 

As $[0,1]$ is compact, the sequence $\{\mf{I}_f(n)/2\pi \bmod{1}\}_n$ has a limit point, say $\alpha$. We select $\{M_j\}_{j\geq 1}$ to be a sequence for which $\tfrac{\mf{I}_f(M_j)}{2\pi} \bmod{1} \ra \alpha$. Let $j' \geq 1$. By the triangle inequality, we then have
$$
\|\mf{I}_f(M_j,M_{j+j'})/2\pi\| = \|\tfrac{\mf{I}_f(M_{j+j'})}{2\pi} - \tfrac{\mf{I}_f(M_j)}{2\pi}\| \leq \|\tfrac{\mf{I}_f(M_j)}{2\pi}-\alpha\| + \|\alpha - \tfrac{\mf{I}_f(M_{j+j'})}{2\pi}\| < \eta,
$$
provided $j$ is chosen large enough. 
\end{proof}

\begin{proof}[Proof of Proposition~\ref{prop_concentration}]
Let $\e > 0$. Applying Lemma~\ref{lem:ctrlIF} with $\eta = \e$, we may choose an infinite increasing sequence $\{M_j\}_{j \geq 1}$ such that, if $j,j'$ are large then upon setting $M := M_j$ and $N := M_{j+j'}$ we find that
$$
\Big|e^{i\mf{I}_f(M,N)} -1\Big| \ll \|\mf{I}_f(M,N)\| \ll \e.
$$
Combining this with Lemma~\ref{lem_concentration} and the condition $\mb{D}(f,1;\infty) < \infty$, we deduce that
$$
\sum_{G \in \mc{M}_{\leq N}} |f(WG+B)-1| \ll q^N\left(\e + \mb{D}(f,1;M,\infty) + o_{M \ra \infty}(1)\right),
$$
which implies the claim.
\end{proof}

\begin{proof}[Proof of Theorem~\ref{thm_katai}] By partial summation, if $(S_n)$ is a non-negative sequence for which $q^{-N}\sum_{n\leq N}S_n=o(1)$, then $\frac{1}{N}\sum_{n\leq N}S_n/q^n=o(1)$. Thus,~\eqref{katai} implies 
\begin{align*}
\sum_{G\in \mc{M}_{\leq N}}|f(QG+1)+zf(G)|/q^{\deg{G}}=o(N).      
\end{align*}
Since $|f(QG+1)+zf(G)|\leq 2$, this further gives
\begin{align*}
\sum_{G\in \mc{M}_{\leq N}}|f(QG+1)+zf(G)|^2/q^{\deg{G}}=o(N),   
\end{align*}
so that expanding the modulus squared and recalling that $f$ is unimodular, we find
\begin{equation}\label{eq:expsqcorrel}
\sum_{G \in \mc{M}_{\leq N}} (1+\text{Re}(zf(G)\bar{f}(QG+1)))q^{-\deg{G}} = o(N),
\end{equation}
We will use this in two ways as follows. First, since the summands are all $\in [0,2]$, for a logarithmic proportion $1-o(1)$ of $G \in \mc{M}_{\leq N}$ we have
\begin{equation} \label{eq:fullLogMeas}
\text{Re}(z f(G)\bar{f}(QG+1)) = -1 + o(1), \text{ i.e., } f(G)\bar{f}(QG+1) = -\bar{z} + o(1),
\end{equation}
by unimodularity. This will be applied shortly. 

Secondly, from~\eqref{eq:expsqcorrel} and the triangle inequality we deduce that
\begin{align}\label{correlation0}
1 + o(1) \leq \frac{1}{N}\Big|\text{Re}\Big(z \sum_{G \in \mc{M}_{\leq N}} f(G)\bar{f}(QG+1)q^{-\deg{G}}\Big)\Big| \leq \frac{1}{N}\Big|\sum_{G \in \mc{M}_{\leq N}} f(G)\bar{f}(QG+1)q^{-\deg{G}}\Big|.
\end{align}
By Theorem~\ref{LogEllFF1},~\eqref{correlation0} implies that for every $N\geq 1$ there exists a Dirichlet character $\chi_N$ of bounded conductor, a short interval character $\xi_N$ of bounded length and an angle $\theta_N\in [0,1]$ such that
\begin{align*}
\mathbb{D}(f,\chi_N\xi_N e_{\theta_N};N)\ll 1.     
\end{align*}
By pigeonholing, we may assume that $\chi_N=\chi$ and $\xi_N=\xi$ for some fixed Dirichlet character $\chi$, short interval character $\xi$ and for an infinite sequence of integers $N$. Since the interval $[0,1]$ is compact, we may find an infinite strictly increasing subsequence $(N_j)$ and a fixed $\theta\in [0,1)$ such that $\lim_{j\to \infty}\theta_{N_j}=\theta$ exists and
\begin{align}\label{pret}
\mathbb{D}(f,\chi\xi e_{\theta_{N_j}};N_j)\ll 1.
\end{align}
By the triangle inequality and the fact that $N_{j}<N_{j+k}$, from~\eqref{pret} we see that
\begin{align*}
 \mathbb{D}(e_{\theta_{N_j}},e_{\theta_{N_{j+k}}};N_j)\ll 1
\end{align*}
uniformly for $k\geq 1$. Letting $k\to \infty$ yields
\begin{align*}
 \mathbb{D}(e_{\theta_{N_j}},e_{\theta};N_j)\ll 1,   
\end{align*}
and hence
\begin{align*}
\mathbb{D}(f,\chi\xi e_{\theta};N_j)\leq \mathbb{D}(f,\chi\xi e_{\theta_{N_j}};N_j)+ \mathbb{D}(e_{\theta},e_{\theta_{N_j}};N_j)\ll 1.   
\end{align*}
Since every $N$ belongs to some interval $[N_j,N_{j+1})$, we finally see that
\begin{align*}
\mathbb{D}(f,\chi\xi e_{\theta};N)\ll 1    
\end{align*}
uniformly in $N$.

Let us now write
\begin{align}\label{factorization}
f(G)=\chi_1(G)\xi(G)f_1(G),    
\end{align}
where $\chi_1$ is the completely multiplicative function given at irreducibles $P$ by $\chi_1(P)=\chi(P)$ if $P\nmid \textnormal{cond}(\chi)$ and $\chi_1(P)=1$ otherwise, and where $f_1$ satisfies $\mathbb{D}(f_1,1;N)\ll 1$.

Recalling~\eqref{shortinterval},~\eqref{eq:fullLogMeas} gives 
\begin{align}\label{func1}
\chi_1f_1(G)\overline{\chi_1f_1}(QG+1)=z'+o(1).
\end{align}
for logarithmic proportion $1-o(1)$ of $G\in \mc{M}$ and some complex number $z'\in S^1$. 

Suppose first that $\{P\in \mc{P} \colon \,\, f_1(P)\neq z'\chi_1f_1(Q+1)\}$ is infinite. Let $P_0$ be an element of this set of degree $>\max\{\textnormal{cond}(\chi),\deg{Q}\}$, and let $\eta>0$ be such that $|f_1(P_0)-z'\chi_1f_1(Q+1)|>\eta$; since this condition becomes less stringent as $\eta$ decreases, we may assume that $\eta$ is smaller than any fixed constant. Let $w$ be a large integer to be chosen shortly, subject in particular to the condition $w>\deg{Q}\textnormal{cond}(\chi)$. Consider the infinite sets
\begin{align*}
\mathcal{A} &:=\{G\in \mc{M} \colon \,\, G \equiv 1\bmod{\prod_{P\in \mc{P}_{\leq w}\setminus\{P_0\}}P^{w}},\,\, G\equiv P_0 \bmod{P_0^2}\}, \\
\mathcal{B} &:= \{(QG+1)/(Q+1) \colon G \in \mc{A}\}. 
\end{align*}
By the Chinese remainder theorem, the elements of $\mc{A}$ may be parametrized by an arithmetic progression $P_0(WG + B)$, where $W \in \mc{M}$ is divisible by all $P \in \mc{P}_{\leq w}$, and $B = B(P_0)$ is some residue class modulo $W$, necessarily coprime to $W$. Moreover, as $G-1$ is divisible by a $P^{\deg{Q}}$ for every $P|(Q+1)$, it follows that whenever $G \in \mc{A}$,
$$
\frac{QG+1}{Q+1} = G - \frac{G-1}{Q+1} \in \mc{M}.
$$
Hence, $\mc{B} \subset \mc{M}$, and as $\deg{P_0} > \deg{Q+1}$ the set $\mc{B}$ may similarly be parametrized as $P_0WG' + D$, for $D = D(P_0)$ coprime to $P_0$. 

For $G \in \mc{A}$ we have $\chi_1(G)=1$, $\chi_1(QG+1)=\chi_1(Q+1)$, and $G/P_0, (QG+1)/(Q+1)$ are both coprime to $\prod_{P\in \mc{P}_{\leq w}}P$. We apply Proposition~\ref{prop_concentration} with $f = f_1$ and $\e = \eta^2$, say, along both $\mc{A}$ and $\mc{B}$: thus, we may find a common choice of parameters $M,N$ (depending only on $f_1$ and $\eta$) such that, upon taking $w = M$, along both $\mc{A}$ and $\mc{B}$ a proportion $1-o_{w\to \infty}(1)$ of $G\in \mc{M}_{\leq N}$ satisfy
\begin{align*}
|f_1(G)-f_1(P_0)|< \eta/10, \quad |f_1(QG+1)-f_1(Q+1)|<\eta/10.    
\end{align*}
Combined with~\eqref{func1} restricted to $\mc{A}$, we see that  $|f_1(P_0)-z'\chi_1f_1(Q+1)|<\eta$. However, this is a contradiction to our assumption, so $\{P\in \mc{P} \colon \,\, f_1(P)\neq z'\chi_1f_1(Q+1)\}$ must be finite. Now, since $f_1$ pretends to be $1$, we must have $z'\chi_1f_1(Q+1)=1$.   

Now, let $N_0$ be such that $f_1(P)=1$ whenever $P\in \mathcal{P}$, $\deg{P}\geq N_0$. Let $M_0$ be the modulus of $\chi$. Let $w'$ be large enough in terms of the aforementioned quantities, and set $W'=\prod_{P\in \mathcal{P}_{\leq w'}}P^{\max\{1,v_P(M_0)\}}$. Let $C$ be arbitrary, subject to $(C,W')=1$. By the Chinese remainder theorem, there exists a residue class $G_0\bmod{W'}$ such that $G\equiv G_0\bmod{W'}$ implies $G\equiv C\bmod M_0$ and $G\equiv 1\bmod{W'/M_0}$. Thus if $G=W'F$, for any $F\equiv G_0\bmod{W'}$ then $f_1(F) = f_1(QG+1) = 1$, and thus
\begin{align*}
\chi_1(G)f_1(G)=\chi_1(W')f_1(W')\chi(C),\quad \chi_1(QG+1)f_1(QG+1)=f_1(QG+1)=1.
\end{align*}
By~\eqref{func1} restricted to such $G$, we conclude that
\begin{align*}
\chi_1(W')f_1(W')\chi(C)=z'+o(1).    
\end{align*}
But this implies that $\chi$ is constant on residue classes coprime to $M_0$, so $\chi$ is principal.

Now,~\eqref{func1} simplifies to 
\begin{align*}
f_1(G)=z'f_1(QG+1)+o(1)
\end{align*}
for logarithmic proportion $1-o(1)$ of $G$. Let us restrict to polynomials $G$ of the form $G=W'F$, where $W'$ is as above (in particular, $P\mid W'$ for $P\in \mc{P}_{\leq N_0}$). Since $(QW'F+1,W')=1$, we deduce that $f_1(W')f_1(F)=z'+o(1)$ for logarithmic proportion $1-o(1)$ of $F\in \mc{M}$. Since $f_1(W') = \prod_{P \in \mc{P}_{\leq N_0}} f_1s(P)^{\max\{1,\nu_P(M)\}}$, which is independent of $w$, there exists a constant $c$ such that $f_1(F)=c+o(1)$ log-almost everywhere. But now if $P_0\in \mc{P}$ is arbitrary, we can find an infinite sequence of polynomials $G$ for which $f_1(P_0G)=f_1(G)+o(1)$, so $f_1(P_0)=1+o(1)$, which means that $f_1(P_0)=1$. Thus $f_1\equiv 1$, and so $f=\xi e_{\theta}$.
\end{proof}

\bibliography{FFChowla}
\bibliographystyle{plain}
\end{document}